\documentclass{article}

\RequirePackage{amsthm,amsmath,amsfonts,amssymb}
\RequirePackage[numbers]{natbib}
\RequirePackage{graphicx}%% uncomment this for including figures

\usepackage[utf8]{inputenc}
\usepackage{float}
\usepackage[english]{babel}
\usepackage{subcaption}
\usepackage{hyperref}
\usepackage{graphicx}
\usepackage[T1]{fontenc}
\usepackage{amsthm,amsmath,amssymb}
\usepackage{float}
\usepackage{mathtools}
\usepackage{stmaryrd} 
\usepackage[total={7in, 9in}]{geometry}

\renewcommand{\phi}{\varphi}

\renewcommand{\P}{\mathbb P}

\newcommand{\E}{\mathbb E}
\newcommand{\R}{\mathbb R}

\newcommand{\N}{\mathbb N}

\newcommand{\ind}{1\!\kern-1pt \mathrm{I}}
\newcommand{\rsto}{]\!\kern-1.8pt ]}
\newcommand{\lsto}{[\!\kern-1.7pt [}

\newcommand\F{\mbox{I\kern-2pt F}}

\newcommand{\vl}{\varphi_{\lambda(\xi)}}
\newcommand{\pl}{\psi_{\lambda(\xi)}}
\newcommand{\la}{\lambda}
\newcommand{\linf}{\lambda_1^\infty}
\newcommand{\eps}{\varepsilon}
\newcommand{\vep}{\varepsilon}

\theoremstyle{plain}
\newtheorem{lemma}{Lemma}
\newtheorem{cor}{Corollary}
\newtheorem{proposition}{Proposition}
\newtheorem{theorem}{Theorem}
\newtheorem{challenge}{Challenge}

\theoremstyle{remark}
\newtheorem{definition}{Definition}

\newtheorem{ex}{Example}
\newtheorem{rem}{Remark}

\begin{document}

\title{Spectral analysis and k-spine decomposition of inhomogeneous branching Brownian motions. Genealogies in fully pushed fronts}
\author{Emmanuel Schertzer\thanks{Faculty of Mathematics, University of Vienna,
		Oskar-Morgenstern-Platz 1, 1090 Wien, Austria}, Julie Tourniaire~\thanks{Institute of Science and Technology Austria (ISTA), Klosterneuburg, Austria}}

\maketitle
\begin{abstract}
	We consider a system of particles performing a one-dimensional dyadic branching Brownian motion with space-dependent branching rate, negative drift $-\mu$ and killed upon reaching $0$. More precisely, the particles branch at rate $r(x)=(1+W(x))/2,$ where $W$ is a compactly supported and non-negative smooth function and the drift $\mu$ is chosen in such a way that the system is critical in some sense.

	This particle system can be seen as an analytically tractable model for fluctuating fronts, describing the internal mechanisms driving the invasion of a habitat by a cooperating population. Recent studies from Birzu, Hallatschek and Korolev suggest the existence of three classes of fluctuating fronts: pulled, semipushed and  fully pushed fronts. Here, we focus on the fully pushed regime. We establish a Yaglom law for this branching process and prove that the genealogy of the particles converges to a Brownian Coalescent Point Process using a method of moments.

	In practice, the genealogy of the BBM is seen as a random marked metric measure space and we use spinal decomposition to prove its convergence in the Gromov-weak topology. We also carry out the spectral decomposition of a  differential operator related to the BBM to determine the invariant measure of the spine as well as its mixing time.
\end{abstract}

\tableofcontents

\section{Introduction}
\label{sect:intro}

\subsection{The model and assumptions}\label{sec:model}
We consider a dyadic branching Brownian motion $(\mathbf{X}_t)_{t>0}$ (BBM) with killing at $0$, negative drift $-\mu$ and position-dependent branching rate
\begin{equation}\label{def:r}
	r(x)=\frac{1}{2}W(x)+\frac{1}{2},
\end{equation}
for some function $W:[0,+\infty)\rightarrow \mathbb{R}$. We assume that $W$ satisfies the following assumptions:
\begin{itemize}
	\item[(A1)] the function $W$ is non-negative, continuously differentiable  and compactly supported.
	\item[(A2)] the support of $W$ is included in $[0,1]$.
\end{itemize}
We denote by $\mathcal N_t$ the set of particles in the system at time $t$ and for all $v\in\mathcal N_t$, we denote by {$x_v=x_v(t)$} the position of the particle $v$. Furthermore, we write $Z_t:=|{\cal N}_t|$ for the number of particles in the system at time $t$. {We write $\mathbb{P}_x$ for the law of the process initiated from a single particle at $x\geq 0$ and $\mathbb{E}_x$ for the corresponding expectation. }

\noindent{\bf Critical regime.} We aim  at choosing $\mu$ in such a way that the number of particles in the system  stays roughly constant.

Fix $L>1$ and consider the BBM $(\mathbf{X}^L_t)_{t>0}$ with branching rate $r(x)$, drift $-\mu$ and killed at $0$ and $L$.
Denote by $\mathcal N_t^L$ the set of particles in this system at time $t$ and define $Z^L_t=|\mathcal N_t^L|$. By a slight abuse of notations, we will also denote by $x_v$ the positions of the particles in the BBM $\mathbf{X}^L$.  
Let $(t,x,y)\mapsto p_t(x,y)$ be the fundamental solution of the linear equation
\begin{equation}
	\begin{cases}
		\partial_tu(t,y) = \frac{1}{2}\partial_{yy}u(t,y) + \mu \partial_{y}u(t,y) + r(y)u(t,y), \\
		u(t,0) = u(t,L) = 0.
	\end{cases}\label{PDE:A}\tag{A}
\end{equation}
We say that $p_t \equiv p_t^L$ is the density of particles in $\mathbf{X}^L$ in the sense that for any measurable set $B\subset[0,L]$, the expected number of particles in $B$ at time $t$ starting from a single particle at $x$ is given by $\int_Bp_t(x,y)dy$ (see e.g.~\cite[p.188]{Lawler:2018vn}).
Let us now define
\begin{equation}\label{def:q}
	{g}_t(x,y): =e^{\mu(y-x)}e^{\frac{\mu^2-1}{2}t}p_t(x,y).
\end{equation}
A direct computation shows that $g_t$ is the fundamental solution of the self-adjoint PDE
\begin{equation}
	\begin{cases}
		\partial_{t}u(t,y) = \frac{1}{2}\partial_{yy}u(t,y) +\frac{1}{2}W(y)u(t,y), \\
		u(t,0) = u(t,L) = 0.
	\end{cases}\label{PDE:B}\tag{B}
\end{equation}
Let  $\lambda_1=\lambda_1(L)$ be the maximal eigenvalue \cite[Chapter 4]{zettl10} of the Sturm--Liouville problem
\begin{equation}\tag{SLP}
	\frac{1}{2}v''(x)+\frac{1}{2}W(x)v(x)=\lambda v(x),
	\label{SLP}
\end{equation}
with  boundary conditions
\begin{equation}\label{BC}
	v(0)=v(L)=0.\tag{BC}
\end{equation}
It is known  that $\lambda_1$ is an increasing function of $L$ \cite[Theorem 4.4.1]{pinsky95} and that it converges to a finite limit $\lambda_1^\infty\in(-\infty,+\infty)$ as $L\to\infty$ \cite[Theorem 4.3.2]{pinsky95}.
We now choose $\mu$ in such a way that the expected number of particles is neither increasing nor decreasing exponentially.
According to (\ref{def:q}), we expect that for large $t$
$$
	p_{t}(x,y) \approx e^{\mu(x-y)}e^{\frac{1-\mu^2}{2}t} e^{\linf t} \frac{v_{1}(x)v_1(y)}{||v_1||^2},
$$
where $v_1$ denotes an eigenfunction associated to $\lambda_1$ for the Sturm--Liouville problem (\ref{SLP}) and $\|\cdot\|$ refers to the $\mathrm{L}^2$-norm.
This motivates the following definition.
\begin{definition}[Critical regime]
	\label{def:critical} The BBM is in the critical regime iff
	\begin{equation}
		\label{def:mu}
		\mu=\sqrt{1+2\linf}.
	\end{equation}
\end{definition}

\paragraph*{Pushed and pulled waves} The next definitions are motivated by recent
numerical simulations and heuristics \cite{birzu2018fluctuations, birzu2021genealogical} for the noisy F-KPP equation with Allee effect
\begin{equation}
	u_t = \frac{1}{2} u_{xx} + u(1-u)(1+Bu) + \sqrt{\frac{u}{N}}{\eta}.
	\label{eq:noisyFKPP}
\end{equation}
From a biological standpoint, this equation describes the invasion of a one-dimensional habitat by a cooperating population: $u(t,x)$ stands for the population density, $B>0$ measures the strength of the cooperation between the individuals, $N$ scales the local density of individuals and $\eta$
is a space-time white noise. It is known that \eqref{eq:noisyFKPP} exhibits two phase transitions as the strength of the cooperation $B$ increases. In the corresponding PDE ($N=\infty$), the fronts are said to be pulled if the speed of the limiting travelling wave solutions is equal to $1$ and pushed if it is larger than $1$. The transition between pulled and pushed waves occurs at $B=2$ \cite{hadeler1975travelling}.
The numerical observations made in \cite{birzu2018fluctuations} indicate a second phase transition in \eqref{eq:noisyFKPP} for $N<\infty$. For pulled waves ($B\in(0,2)$), macroscopic fluctuations in the position of the front are observed on a time scale of order $\log(N)^3$. In the pushed regime ($B>2$), they emerge on the time scale $N$ for $B>4$, and  $N^{\tilde \alpha-1}$, for some $\tilde \alpha\in(1,2)$, for $B\in(2,4)$. This leads to the distinction of two classes of pushed waves: semipushed waves for $B\in(2,4)$ and fully pushed waves for $B>4$.   We refer to \cite{tourniaire21} for more details on the biological interpretation of this model and to Section \ref{sec:bib} for a brief overview on rigorous  results in the pulled and semipushed regimes.

\begin{definition}[Pulled, semipushed, fully pushed regimes]
	\label{def:pushed-pulled}

	Consider the BBM $(\mathbf{X}_t)$ in the critical regime. Define
	\begin{equation}
		\label{def:beta}
		\beta:=\sqrt{2\linf}, \quad \text{and} \quad \alpha:=\frac{\mu+\beta}{\mu-\beta}.
	\end{equation}
	\begin{enumerate}
		\item If $\lambda_1^\infty=0$, or equivalently $\alpha=1$,
		      the BBM is said to be \textbf{pulled}.
		\item If $\lambda_1^\infty\in(0,1/16)$ or equivalently
		      $$
			      \alpha\in(1,2) \Longleftrightarrow \mu> 3\beta,
		      $$
		      the BBM  is said to be \textbf{semipushed}.
		\item If $\lambda_1^\infty>1/16$ or equivalently
		      \begin{equation}\tag{$H_{fp}$}
			      \label{hfp}
			      \alpha>2 \Longleftrightarrow \mu< 3\beta,
		      \end{equation}
		      the BBM is said to be \textbf{fully pushed}.
	\end{enumerate}
	We say that the BBM is \textbf{pushed} if it is either semi or fully pushed, that is when
	\begin{equation}\label{hpushed}
		\tag{$H_p$} \lambda_1^\infty >0.
	\end{equation}
\end{definition}

It is conjectured that, up to rescaling, the size and the genealogy at large
times is indistinguishable from those of a continuous-state branching process (CSBP).
More precisely,
\begin{enumerate}
	\item In the pulled regime, the population size should converge to  Neveu's continuous-state
	      branching process and the genealogy of  the BBM  to the Bolthausen--Sznitman coalescent (see \cite{berestycki13} in the case $W\equiv 0$) on the time scale $\log(N)^3$.
	\item In the semipushed regime, the population size should converge to an $\alpha$-stable CSBP on the time scale $N^{\alpha-1}$, for $\alpha$ as in \eqref{def:beta}. In the previous example, this has been proved only in the case where $W=\mathbf{1}_{[0,1]}$ \cite{tourniaire21}. Therein, it is also conjectured that the genealogy should converge to a time-changed $\mbox{Beta}(2-\alpha,\alpha)$-coalescent \cite{pitman_coalescents_1999,sagitov_1999, birkner2005alpha}.
	\item In the fully pushed regime, the rescaled population size should converge to a Feller diffusion on the time scale $N$, and the genealogy should be indistinguishable from the genealogy of a large critical Galton-Watson process
	      with finite second moment. This is the content of the present article.
\end{enumerate}
Note that the time scales over which the demographic fluctuations emerge in the BBM are similar to those observed in fluctuating fronts solution to \eqref{eq:noisyFKPP}. In addition, the transition between the pulled and the pushed regimes occurs precisely when the speed of the killing boundary (0) becomes larger than $1$ (see \eqref{def:mu} and \eqref{hpushed}).

\begin{ex}
	Let $\eps>0$. Let  $\tilde{W}$ be a function satisfying assumption (A1).  Consider the  BBM with inhomogeneous branching rate $r_\eps(x) = \frac{1}{2} + \eps \tilde W(x)$ in the critical regime (see Definition \ref{def:critical}).
	By \cite[Theorem 4.6.4 and Theorem 4.4.3]{pinsky95}, there exists $0<\vep_1<\vep_2$ such that
	\begin{enumerate}
		\item The BBM is pulled if $\vep\in(0,\vep_1)$.
		\item The BBM is semipushed if $\vep\in(\vep_1,\vep_2)$.
		\item The BBM is fully pushed if $\vep>\vep_2$.
	\end{enumerate}
	\label{ex:2}
\end{ex}

\begin{rem}
	We believe that our results could be extended to  a certain class of perturbations $W$ decreasing exponentially fast to $0$. However, this raises technical challenges that we do not tackle in this work.
\end{rem}

\subsection{Main results}
\label{sec:result}

Throughout this paper, we assume that the BBM $\textbf{X}$ is critical (see \eqref{def:mu}).
\begin{proposition}\label{prop:first}
	Let $v_1$ be the eigenfunction associated to the eigenvalue $\lambda_1$ for the Sturm--Liouville problem \eqref{SLP} with boundary conditions \eqref{BC}, normalised  in such a way that $v_1(1)=1$.
	Under \eqref{hpushed}, $v_1$ converges pointwise and in $\mathrm{L}^2$ to a {positive} limiting function $v_1^\infty$ as $L\to\infty$. Furthermore, if in addition (\ref{hfp}) holds, then
	$$
		\int_{\R_+} e^{\mu x} (v_1^\infty)^3(x) dx <\infty.
	$$
\end{proposition}
To see why the latter proposition may hold true, recall that $W\equiv0$ on $[1,\infty)$. Hence, on this interval,
the problem reduces to
\begin{equation*}
	\frac{1}{2} v_1''(x) = \lambda_1 v_1(x),\quad x\in[1,L],\quad   v_1(L)=0.
\end{equation*}
If we impose the condition $v_1(1)=1$, a direct computation
shows that $v_1(x) = \frac{\sinh(\sqrt{2\lambda_1}(L-x) )}{\sinh(\sqrt{2\lambda_1}(L-1))}$ on $[1,L]$
so that,  for all $x\in[1,\infty)$,
$v_1(x)\to v_1^\infty(x) = e^{-\beta(x-1)}$.
The integrability condition then holds under the extra assumption (\ref{hfp}).

\bigskip

In the following, we consider

\begin{equation}\label{def:h}
	\tilde h^{\infty}(x)  :=  \tilde c e^{-\mu x} v_1^\infty(x),\quad  \text{and} \quad  h^{\infty}(x)  :=  \frac{1}{\tilde c \|v_1^\infty\|^2} e^{\mu x} v_1^\infty(x),
\end{equation}
where
\begin{equation*}
	\tilde c:=\left(\int_0^\infty e^{-\mu x}v_1^\infty (x) dx\right)^{-1}.
\end{equation*}
The constant $\tilde c$ is thought as a Perron-Frobenius renormalisation constant (see e.g.~\cite[p.185]{athreya1972}), in the sense that $h^\infty$  (resp.~$\tilde h^\infty$) is  a right (resp.~left) eigenfunction associated to the maximal eigenvalue of the differential operator \begin{equation}\label{eq:mean:matrix}
	\mathcal{L} u=\frac{1}{2}\partial_{xx}u - \mu \partial_{x}u+r(x)u,
\end{equation}
normalised in such a way that
\begin{equation*}
	\int_0^\infty \tilde h^\infty(x)dx=1 \quad \text{and} \quad \int_0^\infty h^\infty(x)\tilde h^\infty(x)dx=1.
\end{equation*}
From this perspective, the function $\tilde h^\infty$ should correspond to the \emph{stable configuration} of the system and the function $h^\infty$ to the \emph{reproductive values} of the individuals as a function of their positions. We will write $\Pi^\infty$ for the probability distribution whose density is given by
\begin{equation}\label{eq:def:pi:inf}
	\Pi^\infty(x):=h^\infty(x)\tilde h^\infty(x) = (v_1^\infty(x)/
	||v_1^\infty||)^2, \quad x\geq 0.
\end{equation}

\begin{theorem}[Kolmogorov estimate]\label{thm:Kolmogorov} Assume that \eqref{hfp} holds.
	As $N\to\infty$, for all $x,t>0$,
	$$
		{N}\P_x\left( Z_{tN}>0 \right) \to \frac{2}{\Sigma^2 t} h^\infty(x), \ \ \mbox{where} \quad   \frac{\Sigma^2}{2} \ := \ \int_{\R_+} r(z) (h^\infty(z))^2 \tilde h^\infty(z)dz.
	$$
\end{theorem}
This theorem is a continuous analogous of Kolmogorov estimate for multi-type Galton-Watson processes \cite[p.187]{athreya1972}.
\bigskip

We now turn to the description of the genealogy and the Yaglom law.

Intuitively, the next result states that the genealogy is asymptotically identical to the one of a critical
Galton Watson \cite{lambert2018, harris2020, Johnston2019} and that the marks are assigned independently
according to $\tilde h^\infty$. Let us now give a more precise description of our result.

From now on, we condition on the event $\{Z_{tN}>0\}$.
Let $(v_1, \dots, v_k)$ be $k$
individuals chosen uniformly at random from ${\cal N}_{tN}$. Denote by $d_{tN}(v_i,v_j)$ the time
to the most recent common ancestor of $v_i$ and $v_j$. We write  $x_{v_i}$
for the position of the $i^{th}$ individual $v_i$ at time $tN$.
Let $U$ be a uniform r.v. on $[0,t]$ and $\theta>0$. Define $U^\theta$ such that
\begin{align} \label{eq:definition_H_theta}
	\forall s \le t, \qquad	\P( U^{\theta} \le s ) := \frac{(1+\theta) \P(U \le s)}{1+\theta \P(U \le s)}.
\end{align}
Let $(U^{\theta}_{i}; i\in[k])$ be $k$ i.i.d. copies of $U^\theta$ and set
$$
	\forall 1\leq i< j\leq k, \quad U^{\theta}_{i,j} = U^{\theta}_{j,i} :=  \max\{U^\theta_{l}: l\in\{i,\cdots,j-1\}\}.
$$
Define the random distance matrix $(H_{i,j}):= (H_{i,j}; i\neq j \in[k])$ such that for every  bounded and continuous function $\phi :\R^{k^2}\to \R$,
\begin{align} \label{eq:moments-CPP}
	\E\big[ \phi\big( (H_{i,j}) \big) \big]
	= k\int_0^\infty \frac{1}{(1+\theta)^2}
	\Big(\frac{\theta}{1+\theta}\Big)^{k-1}
	\E\big[\phi\big( (U^{\theta}_{i,j}) \big)\big]
	d\theta.
\end{align}
Finally, $(W_i):=(W_i; i\in[k])$ will denote a sequence of i.i.d. copies of a random variable with law $\tilde h ^\infty$.

\begin{theorem}[Yaglom law and limiting genealogies]\label{thm:Yaglom} Assume that \eqref{hfp} holds.
	Let $t>0$. {Suppose that the BBM starts with a single particle at $x>0$.} Conditional on $\{Z_{tN} > 0\}$, { as $N\to\infty$,}
	\begin{enumerate}
		\item[(i)] we have
		      \[
			      \frac{Z_{tN}}{N}
			      \to
			      \frac{\Sigma^2 t}{2} \ \mathcal{E}, \ \ \ \mbox{in distribution,}
		      \]
		      where $\mathcal{E}$ is a standard exponential distribution.
		\item[(ii)] $\left(\big(\tfrac{d_{ tN}(v_i, v_j)}{N}\big),\big(x_{ v_i}\big)\right)$
		      converges to the distribution of $\left((H_{\sigma_i,\sigma_j}),(W_{\sigma_i})\right)$
		      where
		      $\sigma$ is a random uniform permutation of $\{1, \dots, k\}$
		      and $\sigma$, $(H_{i,j})$ and $(W_i)$ are independent. 			
	\end{enumerate}
\end{theorem}

\begin{rem}
	The random distance matrix $(H_{i,j})$ is the one obtained from a critical Galton Watson with finite second moment conditioned on surviving up to a large time.
	See \cite{lambert2018, harris2020, Johnston2019}.
\end{rem}

\subsection{Comparison with previous work} \label{sec:bib}

Branching Brownian motions with inhomogeneous branching rates have received quite a lot of attention in the recent past \cite{horton2020stochastic, harris2020stochastic, harris2022yaglom, gonzalez2022asymptotic, foutel22, tourniaire21,roberts2021gaussian, liu2021particle}.

The general approach always  relies
on a spinal decomposition of the BBM. Roughly speaking, the spine is constructed
by conditioning a typical particle to survive. This conditioning is achieved thanks to a Doob-$h$
transform. In our setting, the harmonic function is approximated by $h(x)\propto e^{\mu x}v_1(x)$
and
the resulting $h$-transform is given by
\begin{equation}\label{spine-00}
	dx_t = \frac{v_1'(x_t)}{v_1{(x_t)}} dt \ + \ dB_t,
\end{equation}
where $B_t$ is a standard Brownian motion (see Section \ref{sect:many-to-few}).

A key assumption underlying \cite{powell19,horton2020stochastic, harris2020stochastic, harris2022yaglom, gonzalez2022asymptotic} is that the harmonic function $h$
is bounded. From a technical stand point, we emphasise that this assumption is the one
distinguishing our work from the previous ones.
Indeed, in the pushed regime, we shall see that $v_1$ decreases exponentially at rate $\beta$ so that  the harmonic function $h(x)$ blows up as $x$ tends to $\infty$.

Due to the explosion of the harmonic function, many of the previously developed techniques break down in our case.

At  first sight, this assumption may only seem technical. However,
it is the key assumption which makes possible a transition from the semi to the fully pushed regime. In the pushed regime, the invariant distribution for the spine dynamics (\ref{spine-00})
is given by
\begin{equation}
	\Pi(x) = \frac{v_1^2(x)}{||v_1||^2}.\label{eq:int:invd}
\end{equation}
Hence, for $x$ large enough,
\begin{equation}
	h(x) \Pi(x)\approx  e^{(\mu-3 \beta) x}.\label{cond:int}
\end{equation}
It then becomes clear from Definition \ref{def:pushed-pulled} that, in the fully pushed regime (resp.~semipushed),
the harmonic function is integrable (resp.~non-integrable) with respect to the invariant measure of the spine.
As a consequence, relaxing the assumption under which  $h$ is bounded is crucial
for understanding the transition between these two regimes.

This generalisation raises interesting technical challenges.
A large fraction of the present work (Section \ref{sec:spectral})
is devoted to estimating the speed of convergence of the spine to its invariant
measure $\Pi$
in the pushed regime.

More precisely, we use Sturm--Liouville theory
in order to derive the spectral decomposition of the differential operator (\ref{PDE:B}) and show that the relaxation time of the system is of order $\log(N)$.
The difficulty arises from the fact that the negative part of the spectrum of the Sturm--Liouville problem (\ref{SLP})
becomes continuous as $L\to\infty$ (see Figure \ref{fig:spectrum}).
Not only is this contribution relevant to the fully pushed regime, but also to the semipushed one. This will be the subject of future work.

One of the main contribution
of the present work is the description
of the genealogy spanned by the population at a large time horizon.
Beyond our Kolmogorov estimate and the Yaglom law reminiscent of  \cite{powell19,horton2020stochastic, harris2020stochastic, harris2022yaglom, gonzalez2022asymptotic}, we use $k$-spine decomposition \cite{Harris2017} and the method of moments developed
in  \cite{foutel22} to prove convergence of the genealogy to a continuum random metric space known as the Brownian Coalescent Point Process \cite{popovic2004} in the Gromov weak topology.
This approach will be further explained in the next section.

The demography and the genealogy of the critical BBM $\mathbf{X}$ have already been investigated  in the pulled and semipushed regimes \cite{berestycki13,maillard20,tourniaire21}. For $W\equiv 0$, it was shown \cite{berestycki13,maillard20} that $\linf=0$ (pulled regime) and that the rescaled population size $\left(\tfrac{1}{N}Z_{t\log(N)^3}, t\geq 0\right)$ converges to  Neveu's CSBP. Using this scaling limit, Berestycki, Berestycki and Schweinsberg \cite{berestycki13} established the convergence of the genealogy  to a Bolthausen-Sznitman coalescent on the time scale $\log(N)^3$. In \cite{berestycki13,maillard20}, the limiting spectrum is also continuous but
the spectral analysis of \eqref{SLP} is straightforward: the spectrum is explicit and given by $\lambda_{i,L}=-\frac{i\pi^2}{2L^2}$ and $v_{i,L}(x)=\sin\left(\frac{i\pi x}{L}\right)$.
In \cite{tourniaire21}, it was proved that, if $W$ is a step function and $\linf\in(0,\tfrac{1}{16})$ (semipushed regime), the exponent $\alpha$ defined in  \eqref{def:beta} belongs to $(1,2)$ and the process $(\tfrac{1}{N}Z_{tN^{\alpha-1}},t\geq0)$ converges to an $\alpha$-stable CSBP. This indicates  that the genealogy of the BBM in the semipushed regime is given by a time-changed Beta-coalescent \cite{birkner2005alpha}. Note that, for $\alpha>1$, the time change depends on the size of the population and is thus random.  When $W=\mathbf{1}_{[0,1]}$, the spectrum of \eqref{SLP} is also explicit.

In the present work, we deal with the case $\linf>\frac{1}{16}$ for continuous compactly supported perturbations $W.$
In particular, the spectral decomposition of \eqref{SLP} is not explicit and we use the Prüfer transformation to derive the required estimates on the $(v_{i,L})$ and the $(\lambda_{i,L})$. Moreover, our strategy is conceptually different from that used in \cite{berestycki13,maillard20,tourniaire21}. Extending the method of moments from \cite{foutel22} to branching diffusions allows us to characterise the genealogy of the system without describing its demographic fluctuations. In fact, the joint convergence of the fluctuations, the genealogy and the configuration of the BBM is a consequence of the convergence of the marked metric measure space associated to the BBM in the Gromov weak topology. In addition, this method allows us to describe the limiting genealogy of the BBM on a deterministic time scale that only depends on the parameters of the model.

\begin{ex} \label{ex:spectre}Consider $W=10.\mathbf{1}_{[0,1]}$. It was calculated
	in \cite{tourniaire21} that the negative part of the spectrum of \eqref{SLP}
	with boundary conditions \eqref{BC} consists of the solutions to
	\begin{equation}\label{eq:ex}
		\frac{\tan(\sqrt{9-2\lambda})}{\sqrt{9-2\lambda}}=-\frac{\tan(\sqrt{-2\lambda}(L-1))}{\sqrt{-2\lambda}}.
	\end{equation}
	The solutions of this equation are plotted on Figure \ref{fig:spectrum}
	\begin{figure}[H]
		\centering
		\includegraphics[width=0.9\textwidth]{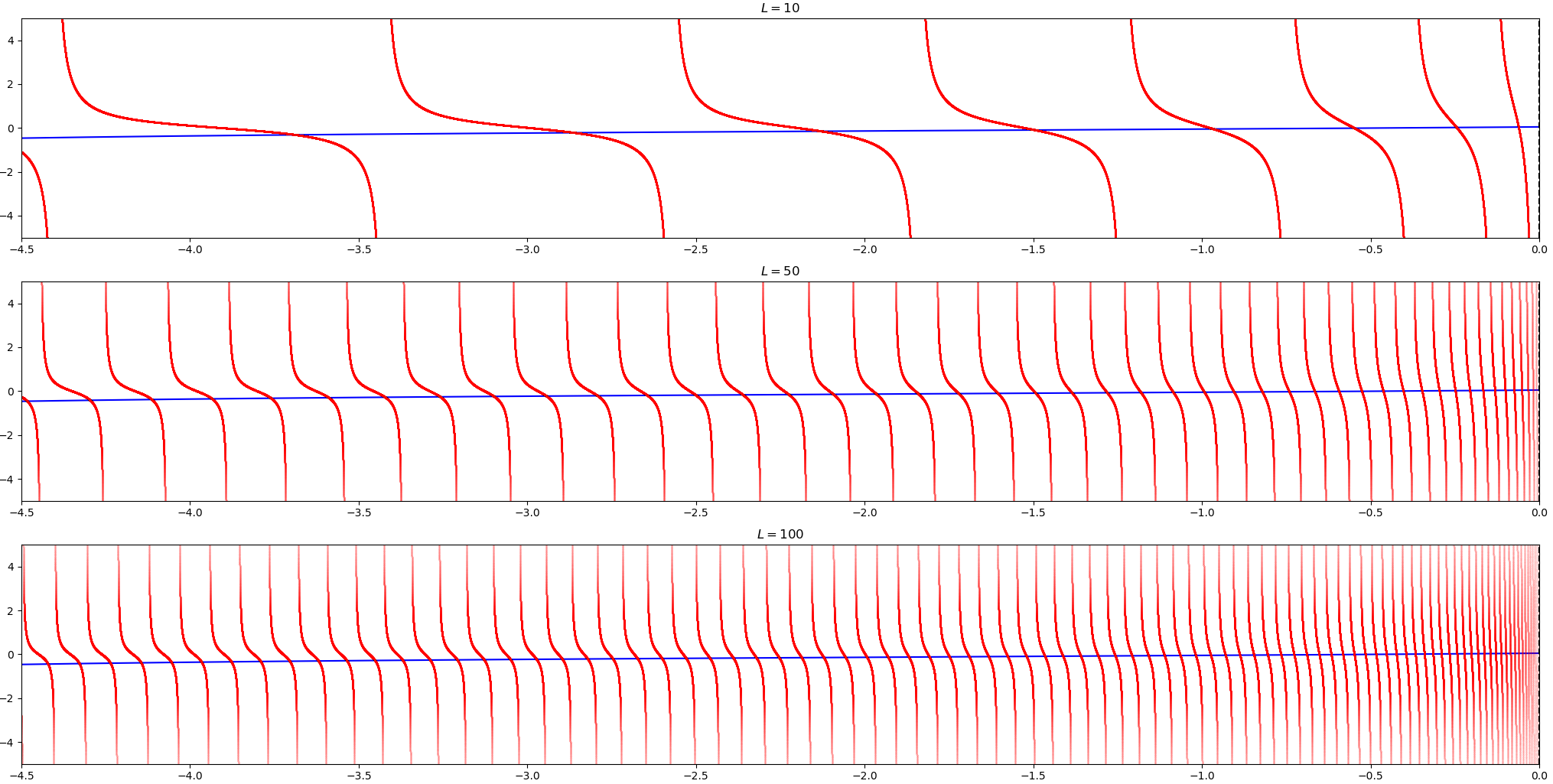}
		\caption{Negative spectrum of the Sturm--Liouville problem \eqref{SLP} with boundary conditions \eqref{BC} for $W$ defined as in Example \ref{ex:spectre} and different values of $L$. The blue line corresponds to RHS of \eqref{eq:ex} and the red line to the LHS of \eqref{eq:ex}.}
		\label{fig:spectrum}
	\end{figure}
\end{ex}

Various models have been suggested to explain the effect of cooperation
(or Allee effect) on the genetic diversity of expanding populations
(see e.g.~\cite{etheridge2022genealogies,birzu2018fluctuations,Roques2012}).
In this work, we consider a toy model for what happens at the tip of the
invasion front. Our BBM can be seen as a moving frame following the particles leading the invasion. At the leading edge, the competition for resources is negligible and, as a first approximation, the particles evolve independently.    In this framework, the killing boundary at $0$ models the beginning of the bulk and the perturbation $W$ encodes the Allee effect.  Despite its simplicity (independence of the particles, compactly supported perturbation),  this system exhibits the same phase transitions and the same macroscopic behaviour as the  model with interaction \eqref{eq:noisyFKPP} (see Section \ref{sec:model} and \cite[Section 1.4]{tourniaire21}), which suggests that the two models belong to the same universality class. In particular, this indicates that the phase diagram introduced in Definition~\ref{def:pushed-pulled} may not depend on the precise form of the interactions between the individuals. In a nutshell,  our model can be seen as a linearisation of \eqref{eq:noisyFKPP} that captures the non-linear signature of Allee effects.

In the present work, we show that the genealogy of fully pushed waves is similar to that  of a neutral unstructured population with finite reproductive variance, as conjectured in \cite{birzu2021genealogical} for fluctuating fronts. We refer to \cite{etheridge2022genealogies} for a similar results on bistable waves in a model with interactions.

We conclude this section by highlighting that our integrability condition already appeared in different forms:

\begin{itemize}
	\item In \cite{birzu2018fluctuations}, the authors derive an explicite formula for the variance in the position of the front. They show that this variance  is given by $2\sigma^2 t$, for some constant $\sigma^2\equiv \sigma^2(N)$ such that (see \cite{birzu2018fluctuations}, Equation [5])
	      \begin{equation*}
		      \sigma^2\propto \frac{1}{N}\int_\mathbb{R}\rho'(x)^2\rho(x)e^{4 v x}dx,
	      \end{equation*}
	      where $v$  is the propagation speed of the determinist dynamic associated to \eqref{eq:noisyFKPP} ($N=\infty$) and $\rho$ refers to the corresponding travelling front.
	      The mass of the above integral is concentrated  where $\rho\approx 0$. The dispersion relation then shows that $\rho(x)\sim e^{-(v-\sqrt{v^2-1})x}$ in this region so that $\sigma^2$ is infinite and the fluctuations are much larger than the one provided by the central limit theorem when $v>3\sqrt{v^2-1}.$ This is precisely our integrability condition~\eqref{hfp}.
	      
	\item The harmonic function $h(x)$ can be interpreted as the \emph{reproductive value} of an individual at $x$, that is, its relative contribution to the travelling wave. Interestingly enough, this quantity also appears in  the context of deterministic bistable fronts, see e.g.~\cite[Equations 8]{Roques2012}.
	      Recalling that $p_t(x,\cdot)$ refers to the density of particles in the BBM starting from a single particle at $x$ (see \eqref{PDE:A}), the quantity
	      \begin{equation}\label{eq:var:intro}
		      \int_0^\infty r(y)h(y)^2 \; p_t(x,y)dy\end{equation}
	      can be interpreted as the reproductive variance in the BBM at time $t$. We will show that, for sufficiently large time $t$, the density of particles in the process $p_t(x,y)$ is roughly proportional to $\tilde{h}(y)\approx e^{-\mu y}v_1(y)$.
	      Hence, \eqref{eq:var:intro} is finite precisely when \eqref{cond:int} is integrable. From this perspective, the Kolmogorov estimate and the Yaglom law established in this article are very similar to those obtained for multi-type Galton-Watson processes with finite variance \cite[p.185]{athreya1972} (see Section \ref{sec:result}).
	\item In \cite{etheridge2022genealogies,birzu2018fluctuations}, the integral $\int h\Pi$ is interpreted as the rate of coalescence of the lineages at the tip of the front. As a first approximation, the lineages coalesce instantaneously at the leading edge when they meet on the same site. Hence, the time to coalescence is roughly given by an exponential random variable of parameter proportional to
	      \begin{equation}\label{eq:parameter}
		      \int_0^\infty \frac{1}{p_t(x,y)}r(y)\Pi(y)^2dy.
	      \end{equation}
	      When this parameter is finite, the genealogy of the sample should converge, after suitable scaling, to Kingman's coalescent.
	      Recalling that  $p_t(x,y)\propto e^{-\mu y}v_1(y)$ for $t\gg1$,  we see that the parameter \eqref{eq:parameter} is finite precisely when \eqref{cond:int} is integrable.
\end{itemize}

\subsection{Notation}\label{notation}
Given two sequences of positive real numbers $(a_N)$ and $(b_N)$, we write $a_N\ll b_n$ if $a_N/b_N\to 0$ as $N\to\infty$. We write $a_N\lesssim b_N$ if $a_N/b_N$ is bounded in absolute value by a positive constant and $a_N\asymp b_N$ if $a_N\lesssim b_N$ and $b_N\lesssim a_N$. We write $O(\cdot)$ to refer to a quantity bounded in absolute value  by a constant times the quantity inside the parentheses. \textbf{Unless otherwise specified}, these constants only depend on $\linf$. For $k\in\mathbb{N}$, we write $[k]$ for the set of natural numbers $\{1,...,k\}$.

\section{Outline of the proof}

Our approach relies on the method of moments devised in \cite{foutel22}. To illustrate
this approach, let us first think about the Yaglom law
of Theorem \ref{thm:Yaglom}. To prove this result, one needs to show that the
moments of $Z_{tN}/N$ converge to the moments of an exponential random variable.
It turns out that this approach can be extended to genealogies.

In Section \ref{sect:mmm},  following  \cite{depperschmidt_marked_2011},
we encode the genealogy at time $tN$
as a random marked metric measure  space.  The moments of this random marked metric measure  space
are obtained by biasing the population by its $k^{\text{th}}$ moment and then picking $k$ individuals uniformly at random (see Remark \ref{rem:moments} below).
In section \ref{sect:CPP}, we introduce a limiting random marked metric measure  space called the marked Coalescent Point Process (CPP) which corresponds to the limiting genealogy of a critical Galton-Watson process \cite{popovic2004}.
The remainder of the section is dedicated to the sketch of the proof for the convergence of the moments of our BBM to the moments of
the marked CPP using the spinal decomposition introduced in \cite{foutel22}.

\subsection{Marked Metric Measure Spaces}\label{sect:mmm}

Let $(E, d_E)$ be a fixed complete separable metric space, referred to as
the \emph{mark space}. In our application, $E = (0, \infty)$ is
endowed with the usual distance on the real line. A \emph{marked metric
	measure space} (mmm-space for short) is a triplet $[X, d, \nu]$, where
$(X,d)$ is a complete separable metric space, and $\nu$ is a finite
measure on $X \times E$. We define
$|X|:=\nu(X\times E)$.
(Note that $\nu$ is
not necessarily a probability distribution.)
To define a topology on the set of mmm-spaces, for each $k \ge 1$, we
consider the map
\[
	R_k \colon
	\begin{cases}
		(X \times E)^k \to \R_+^{k^2} \times E^k \\
		\big( (v_i, x_i);\, i \le k \big) \mapsto
		\big( d(v_i, v_j), x_i ;\, i,j \le k \big)
	\end{cases}
\]
that maps $k$ points in $X\times E$ to the matrix of pairwise distances and
marks. We denote by $\nu_{k, X} = \nu^{\otimes k} \circ R_k^{-1}$, the
\emph{marked distance matrix distribution} of $[X, d, \nu]$, which is the
pushforward of $\nu^{\otimes k}$ by the map $R_k$. Let $k \ge 1$ and consider a
continuous bounded test function $
	\phi \colon \R_+^{k^2} \times E^k \to \R$.
One can define a functional
\begin{equation} \label{eq:polynomials}
	\Phi\big( X, d, \nu \big) = \left<\nu_{k,X}, \phi \right> = \int_{(X\times E)^k} \phi\bigg(d(v_i,v_j), x_i; i\neq j \in[k] \bigg) \prod_{i=1}^k \nu(dv_i \otimes dx_i).
\end{equation}
Functionals of the previous form are called \emph{polynomials}, and the
set of all polynomials, obtained by varying $k$ and $\phi$, is denoted by
$\mathbf{\Pi}$.
Let $\phi$ be of the form
$$
	\phi\bigg(d(v_i,v_j), x_i; i\neq j \in[k] \bigg)   \ = \ \prod_{i,j} \psi_{i,j}(d(v_i,v_j)) \prod_{i} \phi_i(x_i),
$$
where $\psi_{i,j},\phi_i$ are bounded continuous functions. In this case, we say that $\Phi(X,d,\nu)$ is a product polynomial.
We denote by $\tilde {\mathbf \Pi}$ the set of product polynomials.

\begin{definition}
	The marked Gromov-weak  (MGW) topology is the topology on mmm-spaces induced
	by $\mathbf{\Pi}$. A random mmm-space is a r.v.~with values in $\mathbb{M}$ -- the set of
	(equivalence classes of) mmm-spaces -- endowed with the marked  Gromov-weak
	topology and the associated Borel $\sigma$-field.
	The marked Gromov-weak  (MGW) topology is identical to the topology induced
	by the product polynomials $\tilde {\mathbf{\Pi}}$.
\end{definition}

For a random mmm-space $[X, d, \nu]$, the moment of $[X,d,\mu]$ associated to $\Phi$ is defined as $\E[ \Phi\big( X, d, \nu \big) ]$.

\begin{rem}\label{rem:moments}
	Consider a random mmm-space $[X,d,\nu]$.
	Assume that $\E[|X|^k]<\infty$.
	The moments of $[X,d,\nu]$ can be rewritten as
	$$
		\E\left[\Phi\big( X, d, \nu \big)\right] \ = \E[|X|^k]
		\times \frac{1}{\E[|X|^k]} \E\left[|X|^k
			\phi( d(v_i,v_j), x_{v_i}, i\neq j \in[k]   ) \right],
	$$
	where $(v_i,x_{v_i})$ are $k$ points sampled uniformly at random with their marks and $|X|:=\nu(X\times E)$
	is thought as the total size of the population. As a consequence, the moments of a random mmm-space
	are obtained by biasing the population size by its $k^{th}$ moment and then picking $k$ individuals uniformly at random.
\end{rem}

Many properties of the marked Gromov-weak topology are derived in
\cite{depperschmidt_marked_2011} under the further assumption that
$\nu$ is a probability measure. In particular,
the following result shows that $\mathbf{\Pi}$ forms a convergence determining
class only when the limit satisfies a moment condition, which is a
well-known criterion for a real variable to be identified by its moments,
see for instance \cite[Theorem~3.3.25]{durrett_probability_2019}.
This result was already stated for metric measure spaces without marks
in \cite[Lemma~2.7]{depperschmidt2019treevalued} and was proved in \cite{foutel22}.

\begin{proposition} \label{lem:convDetermining}
	Suppose that $[X,d,\nu]$ is a random mmm-space verifying
	\begin{equation} \label{eq:momentCondition}
		\limsup_{p \to \infty} \frac{\E[|X|^p]^{1/p}}{p} < \infty.
	\end{equation}
	Then, for a sequence $[X_n, d_n, \nu_n]$ of random mmm-spaces to
	converge in distribution for the marked Gromov-weak topology to
	$[X,d,\nu]$ it is sufficient that
	\[
		\lim_{n \to \infty} \E\big[ \Phi\big(X_n, d_n, \nu_n\big) \big]
		= \E\big[ \Phi\big(X, d, \nu\big) \big]
	\]
	for all $\Phi \in \mathbf{\Pi}$.
\end{proposition}

\subsection{Marked Brownian Coalescent Point Process (CPP)}
\label{sect:CPP}
Let $T>0$ and $m$ be a measure on $\R_+$.
Assume that $|m|:=m(\R_+)>0$.
Consider ${\cal P}$ a $PPP\left(\frac{dt}{t^2} \otimes d\ell\right)$.
Define
$$
	Y_T \ = \ \inf\left\{y : (t,y)\in {\cal P}, t\geq T\right\},
$$
and
$$
	d_T(x,y) \ = \ \sup\{t  :  (t,z) \in {\cal P}\ \ \mbox{and} \ \ x \leq z \leq y\}, \quad  0<x<y<Y_T.
$$
The marked Brownian Coalescent Point Process (CPP) is defined as
$$
	M_{\text{CPP}_T} \ :=  \bigg([0,Y_T],d_T, \text{Leb} \otimes m(dx) \bigg),
$$
where Leb refers to the Lebesgue measure on $[0,Y_T]$. This definition is illustrated in Figure~\ref{fig:cpp_simulation}.
This object is a natural extension of the standard Brownian CPP \cite{popovic2004}.

\begin{rem}\label{rem:exp}
	A direct computation shows that $Y_T m(\R_+)$ (which can be thought as the population size at time $T$) is distributed as an exponential random variable
	with mean $T |m|$. If the CPP encodes the size and the genealogy of critical
	branching processes, this is consistent with Yaglom's law for such processes.
\end{rem}

\begin{figure}[H]
	\centering
	\includegraphics[width=.55\textwidth, angle=180]{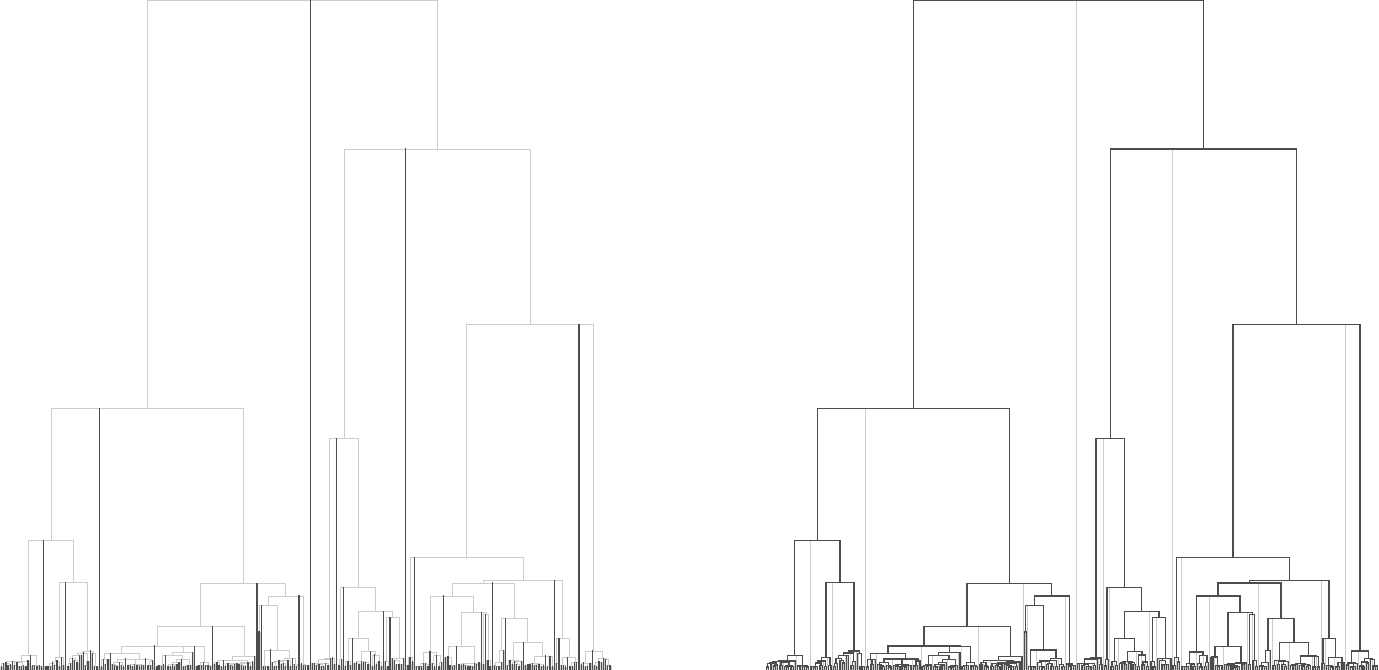}
	\caption{ Simulation of the unmarked Brownian CPP. On the RHS, a vertical line of height $t$ at location
		$y$ represents an atom $(t,y)$ of $\cal{P}$. On the LHS, the tree corresponding
		to the right CPP; the distance $d_T$ is the tree distance of the
		leaves.}
	\label{fig:cpp_simulation}
\end{figure}

\begin{proposition}\label{moments:CPP}
	Let $k\in{\mathbb N}$.
	Let $(\phi_i; i\in[k])$ and $(\psi_{i,j}; i,j \in[k])$ be  continuous bounded functions.
	Consider a product polynomial of the form
	$$
		\forall M=[X,d,\nu], \ \ \Psi(M) \ := \ \int \prod_{i,j=1}^k \psi_{i,j}(  d({v_i, v_j}) ) \prod_{i=1}^k \phi_{i} (x_i) \nu(dv_i \otimes dx_i).
	$$
	Then
	$$
		\E\left[\Psi(M_{\text{\textnormal{CPP}}_T}) \right] \ = \ k ! T^{k} \E\left[ \prod_{i,j=1}^k \psi_{i,j}( U_{\sigma_i,\sigma_j } )   \right] \left[\prod_{i=1}^k \int  m(dx) \phi_i(x) \right],
	$$
	where
	$(U_{i}; i\in[k-1])$ is a vector of uniform i.i.d. random variables on $[0,T]$,
	$$U_{i,j}= U_{j,i} = \max\{ U_{l}: l=i,\cdots, j-1\},$$ $\sigma$ is a random uniform permutation of $\{1, \dots, k\}$, and $\sigma$ and $(U_i)$ are independent.
\end{proposition}
\begin{proof}
	The proof is identical to Proposition 4 in \cite{foutel22}.
\end{proof}

{
\begin{proposition}[Sampling from the CPP]\label{SAmpling-CPP}
	Let $k\in\N$ and
	sample $k$ points, denoted by $(v_1,x_1),\cdots,(v_k,x_k)$, uniformly at random from the CPP.
	Then $(\left(d_T(v_i,v_j)\right)_{i,j}, (x_{i})_i )$ is identical in law to  $\left((H_{\sigma_i,\sigma_j}),(w_{\sigma_i})\right)$
	where $(H_{i,j})$ is defined as in Theorem \ref{thm:Yaglom} (ii), the $(w_i)$ are independent random variables with law $\frac{m}{|m|}$ and $\sigma$ is a random uniform permutation of $\{1,\cdots,k\}$, independent of $(H_{i,j})_{i,j}$ and $(w_i)_i$.
\end{proposition}
\begin{proof}
	The proof is identical to the one in the case of the unmarked CPP. See \cite[Proposition 4.3]{boenkost2022genealogy}.
\end{proof}
}

\subsection{Convergence of mmm}\label{sec:MMM:BBM}
{ Fix $t>0$. Recall that ${\cal N}_{t}$ refers the set of particles alive at time $t$ in the BBM $\mathbf{X}$. Set

\begin{equation*}
	\nu_{t}  :=  \sum_{v\in {\cal N}_{t}} \delta_{v, x_v}, \quad \text{and} \quad  \forall v,v'\in{\cal N}_{t}, \quad  d_t(v,v') =  t - |v\wedge v'|,
\end{equation*}
where $|v\wedge v'|$ denotes the most recent time at which $v$ and $v'$ had a common ancestor.
Let $M_{t} := [{\cal N}_{t},d_t ,\nu_{t}]$ be the resulting random mmm-space.
Finally, set
\begin{equation*}
	\bar \nu_{t}  :=  \frac{1}{N} \sum_{v\in {\cal N}_{tN}} \delta_{v, x_v},\quad \text{and} \quad  \forall v,v'\in{\cal N}_{tN},    \quad \bar d_t(v,v') = \ \left(t - \frac{1}{N}|v\wedge v'|\right),
\end{equation*}
and define the rescaled metric space $\bar M_t \ := \ [{\cal N}_{tN},\bar d_t ,\bar \nu_{t}]$. The main idea underlying Theorem \ref{thm:Yaglom} is to prove the convergence of $\bar M_{t}$ to a limiting CPP.

\begin{theorem}\label{thm:main-theorem}
	Conditional on the event $\{Z_{tN}>0 \}$, $(\bar M_{t}; N\in\N)$ converges in distribution to a marked Brownian CPP with parameters $(t,\frac{\Sigma^2}2 \tilde h^\infty)$.
\end{theorem}}

The proof of the theorem relies on a cut-off procedure. Let
\begin{equation}\label{def:L}
	L := \frac{1}{\mu-\beta} \log(N).
\end{equation}
Recall that $\mathbf{X}^L$ refers to the BMM killed at $0$ and $L$. Let $\nu_t^L$ (resp.~$\bar\nu_t^L$) be the  empirical measure obtained  by replacing ${\cal N}_t$ by ${\cal N}^L_t$ in the definition of $\nu_t$ (resp.~$\bar \nu_t$).
Let $M_{t}^L$  be the mmm-space obtained from $\mathbf{X}^L$, that is $M_{t}^L=[{\cal N}_{t}^L,d_t,\nu_{t}^L]$.
$\bar M_t^L$ is defined analogously to $\bar M_t$ (i.e.~accelerating time by $N$ and rescaling the empirical measure by $1/N)$. { Finally, define for all $(t,x)\in[0,\infty)\times [0,L]$,
\begin{equation}\label{def:hL}
	h(t,x):=\frac{1}{\tilde c \|v_1\|^2}e^{(\linf-\lambda_1)t}e^{\mu x}v_1(x), \quad \text{and} \quad \tilde h(t,x):=\tilde ce^{(\lambda_1-\linf)t}e^{-\mu x}v_1(x),
\end{equation}
where  $\tilde c$ is as in \eqref{def:h}.}

We will proceed in two steps. For our choice of $L$,  we will show that
\begin{enumerate}
	\item $\bar M_{t}^L$ converges to the limit described in Theorem \ref{thm:main-theorem}.
	\item $\bar M_{t}^L$ and $\bar M_{t}$ converge to the same limit.
\end{enumerate}
The choice for $L$ will be motivated in Section \ref{sect:choice-of-L}.
We start by motivating the fact that   $\bar M_{t}^L$ converges to the desired limit using a spinal decomposition introduced in \cite{foutel22} in a discrete time setting.

\subsection{The $k$-spine}\label{sect:k-spine-CV}

\begin{definition}\label{def:1-spine}
	The $1$-spine process is the stochastic process on $[0,L]$ with generator
	$$
		\frac{1}{2} \partial_{xx}u \ + \frac{v_1'(x)}{v_1(x)} \partial_x u, \ \ u(0)=u(L)=0.
	$$
	In the following, $q_t(x,y) \equiv q_t^L(x,y)$ will denote the probability kernel of the $1$-spine.
\end{definition}

\begin{lemma}[Many-to-one]\label{lem:many-to-one0}
	For every bounded measurable function $f:\mathbb{R}^+\to\mathbb{R}$
	\begin{equation*}
		\mathbb{E}_{x}\left[\sum_{v\in{\cal N}^L_t} f({x}_{v})\right] \ = \ \int_0^L f(y)  q_{t}(x,y)  \frac{h(0,x)}{h(t,y)} dy.
	\end{equation*}
\end{lemma}
\begin{proof}
	It is known (see e.g.~\cite[p.188]{Lawler:2018vn}) that for every bounded measurable function $f:\mathbb{R}^+\to\mathbb{R}$
	\begin{equation*}
		\mathbb{E}_{x}\left[\sum_{v\in{\cal N}^L_t} f({x}_{v})\right] \ = \ \int_0^L f(y)  p_{t}(x,y)  dy,
	\end{equation*} where $p_t$ is the fundamental solution of \eqref{PDE:A}. A direct calculation shows that $\frac{h(0,x)}{h(t,y)}q_t(x,y)$ is also a fundamental solution of \eqref{PDE:A}. The result follows by recalling  that the fundamental solution of  \eqref{PDE:A} is unique.
\end{proof}

The next result is standard.

\begin{proposition}
	\label{prop:invariant-1-spine}
	The $1$-spine has a unique invariant probability measure given by
	$$\Pi(dx)  =  h(t,x)\tilde h(t,x) dx =  \left(\frac{v_1(x)}{||v_1||}\right)^2dx.$$
\end{proposition}

We now define the $k$-spine tree. Let $(U_1,...,U_{k-1})$ be i.i.d.~random variables uniformly distributed on $[0,t]$. Define
\begin{equation}
	\label{eq:UPP}
	\forall \; 1\leq i < j\leq k-1, \quad U_{i,j}=U_{j,i} =\;\max\{U_i,...,U_{j-1}\}.
\end{equation}
Let ${\mathbb T}$ be the tree of depth $t$ with $k$ leaves such that the tree
distance between the $i^{th}$ and $j^{th}$ leaves is given by $U_{i,j}$.
This tree is \textit{ultrametric} and {\it planar} in the sense that
$$
	\forall i,j,\ell\in[k], \ \ U_{i,j} \leq  U_{i,\ell} \vee U_{{\ell},j}
$$
(ultrametric) and the inequality becomes an equality if $i<\ell < j$ (planar). The depth $\tau$ of the first branching point is thus given by
\begin{equation*}
	\tau=\max_{i\in[k-1]}U_i.
\end{equation*}
Marks are then assigned as follows. On each branch of the tree, the marks evolve according to the $1$-spine process (on $[0,L]$) and branch into independent particles at the branching points of ${\mathbb T}$. The resulting planar marked ultrametric tree will be referred to as the $k$-spine tree and denoted by $\mathcal{T}$. We write $\mathcal{T}_0$ (resp.~$\mathcal{T}_1$)
for the left (resp.~right) marked subtree  attached to the first branching point of the tree. Note that these two subtrees are also planar marked ultrametric trees and that they are both rooted at $x_{MRCA}(\mathcal{T})$, the mark of the first branching point.

In the following, $\cal B$ will denote the set of  $k-1$ branching points of the $k$-spine tree $\mathcal{T}$ and $\cal L$ will denote the set of $k$ leaves.
We will write $\zeta_v$ for the mark (or the position) of the spine at a node $v\in \cal B\cup \cal L$. For $v\in{\cal B}$, $|v|$ will denote the time component of the branching point.
Finally, $(V_i; i\in[k])$ is the enumeration of the leaves from left to right in the $k$-spine tree (i.e., $V_i$ is the leaf with label $i$).
See Figure \ref{fig:k-spine} for an illustration of these definitions. We refer to \cite[Section 3.1.3]{foutel22} for a more formal construction of the $k$-spine tree.

\begin{figure}[t]
	\centering
	\includegraphics[width=.6\textwidth]{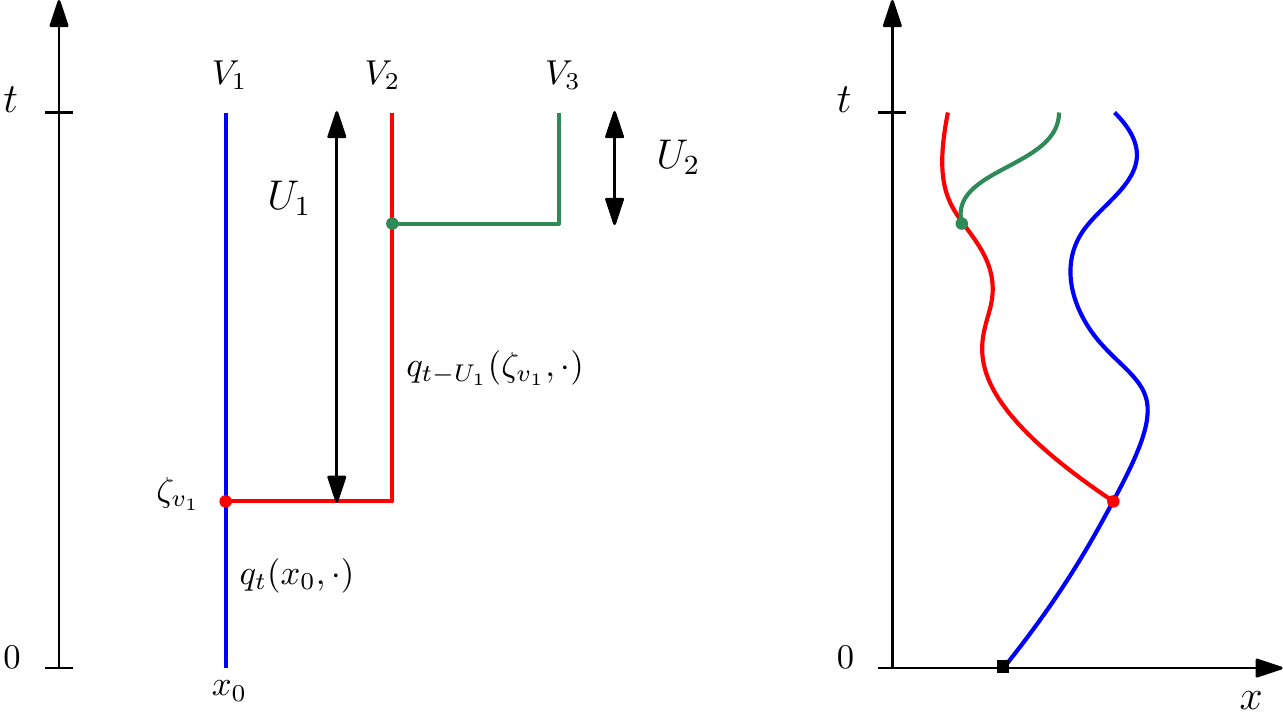}
	\caption{ The $k$-spine tree ($k=3$). Left panel: planar tree $\mathbb{T}$ generated from $2$ i.i.d. uniform random variables $(U_1,U_2)$.  In this tree, we have $\tau=U_1$. The corresponding node in $\mathcal{B}$ is denoted by $v_1$. The left and right subtrees $\mathcal{T}_0$ and $\mathcal{T}_1$ rooted at $\zeta_{v_1}$ both have depth $\tau$ , $\mathcal{T}_0$ has a single leaf, $V_1$, and $\mathcal{T}_1$ has two leaves, $V_2$ and $V_3$.
		Right panel: branching $1$-spines running along the branches of the tree ${\mathbb T}$.}
	\label{fig:k-spine}
\end{figure}

We will denote by $Q^{k,t}_x$ the distribution of the  $k$-spine rooted at $x$.
The $t$ superscript refers to the depth of the underlying genealogy.
Note that $Q^{k,t}_x$ has an implicit dependence on $N$ by our choice of $L$ -- see (\ref{def:L}).
To ease the notation, this dependence will be dropped in the notation.
We will also need to define  the accelerated version of the $k$-spine.

\begin{definition}[Accelerated $k$-spine]
	Consider the $1$-spine accelerated by $N$, i.e. the  transition kernel of the $1$-spine
	is now given by $ q_{tN}(x,y)\equiv q^L_{tN}(x,y)$. {We denote this kernel by $\bar q_t(x,y)$}.
	Consider the same planar structure as before: the depth of the tree is $t$
	and the distance between the leaves is given by (\ref{eq:UPP}).
	We denote by $\bar Q_{x}^{k,t}$ the distribution of the  $k$-spine obtained by running accelerated  spines along the branches.
	For any vertex $v$ in the accelerated $k$-spine, $\bar \zeta_v$ will denote the mark of the vertex $v$.
\end{definition}

\begin{proposition}[Rescaled many-to-few]\label{th:many-to-few2}
	Let $t>0$.
	Let $(\phi_i; i\in[k])$ and $(\psi_{i,j}; i,j \in[k])$ be measurable bounded functions and define
	$$
		\forall M=[X,d,\nu], \ \ \Psi(M) \ = \ \int \prod_{i,j} \mathbf{1}_{\{v_i\neq v_j\}} \psi_{i,j}(d(v_i,v_j)) \prod_{i} \phi(x_i) \nu(dv_i\otimes d x_i).
	$$
	{  Let
			$
				h$} be defined as in (\ref{def:hL}).
	Then
	$$
		{\mathbb E}_x\left[ \Psi(\bar M^L_{t}) \right] \ = \ \frac{1}{N} k! \  h(0,x) \ t^{k-1} \bar Q_{x}^{k,t}\left(  \bar \Delta \prod_{i,j} \psi_{i,j}( U_{\sigma_i,\sigma_j})  \prod_{i} \phi_{i} (\bar {\zeta}_{V_{\sigma_i}})    \right),
	$$
	where $(U_{i,j})$ is as in \eqref{eq:UPP}, $\sigma$ is an independent uniform permutation of $\{1,...,k\}$,  and
	$$
		\bar \Delta := \prod_{v\in {\cal B}} r({ \bar{\zeta}_v}) h(|v|N, \bar {\zeta}_{v}) \prod_{v\in {\cal L}} \frac{1}{h(tN,\bar{\zeta}_{v })}.
	$$
\end{proposition}

The proof of Proposition \ref{th:many-to-few2} will be the object of Section \ref{sect:many-to-few}. Our second crucial result is the following convergence theorem.

\begin{theorem}\label{thm:k-spine-cv}
	Let $(\tilde \phi_i; i\in[k])$ and $(\psi_{i,j}; i,j \in[k])$ be continuous bounded functions.
	Assume further that the $\tilde \phi_i$'s are compactly supported in $(0,\infty)$.
	As $N\to\infty$,
	\begin{multline*}
		\bar Q_{x}^{k,t}\left(  \bar \Delta \prod_{i,j} \psi_{i,j}( U_{\sigma_i,\sigma_j})  \prod_{i}  \tilde \phi_{i} (\bar {\zeta}_{V_{\sigma_i}}) { h^\infty(\bar {\zeta}_{V_{\sigma_i}})}   \right) \longrightarrow   \left(\frac{\Sigma^2}{2}\right)^{k-1} {\mathbb E}\left( \prod_{i,j} \psi_{i,j}( U_{\sigma_i,\sigma_j}) \right)
		\prod_{i} \int_{{\mathbb R}_+}  \tilde \phi_i(x)  \Pi^\infty(dx),
	\end{multline*} where $\Pi^\infty$ is as in \eqref{eq:def:pi:inf}.
\end{theorem}
Let us now give  brief heuristics for the previous result. By definition
\begin{multline*}
	\bar Q_{x}^{k,t}\left(  \bar \Delta \prod_{i,j} \psi_{i,j}( U_{\sigma_i,\sigma_j})  \prod_{i} { \tilde \phi_{i}} (\bar {\zeta}_{V_{\sigma_i}}) {h^\infty(\bar {\zeta}_{V_{\sigma_i}})}   \right)\\ = \
	\bar Q_{x}^{k,t}\left(   \prod_{v\in \cal B} r(\bar {\zeta}_v) {h(|v|N,\bar {\zeta}_{v})} \prod_{i,j}   \psi_{i,j}( U_{\sigma_i,\sigma_j}) \prod_{ i} \tilde \phi_{i} (\bar {\zeta}_{V_{\sigma_i}} ){\frac{h^\infty(\bar \zeta_{V_{\sigma_i}})}{h(tN, \bar \zeta_{V_{\sigma_i}})}}
	\right).
\end{multline*}

The branching structure for the $k$-spine is binary a.s.~and the spines running along the branches are accelerated by $N$.
We will show later on that under \eqref{hfp}, $h(tN,x)\approx h^\infty (x)$ and $ \tilde h(tN,x)\approx \tilde h^\infty (x)$  for $N$ large.
In addition, we see fom Proposition \ref{prop:invariant-1-spine}, as $N\to\infty$,
$$
	\Pi(dx) \approx \Pi^\infty(dx) =  h^\infty(x) \tilde h^\infty(x) dx \ = \ \frac{v_1^\infty(x)^2 }{||v_1^\infty||^2} dx.
$$
It is now reasonable to believe that, provided enough mixing, the RHS can be approximated by
$$
	\left(\int_0^\infty h^\infty(x) r(x) \Pi^\infty(dx) \right)^{k-1}  \E\left(  \prod_{i,j}   \psi_{i,j}( U_{\sigma_i,\sigma_j})  \right)  \prod_i \int \tilde \phi_i(x)  \Pi^\infty(dx),
$$
assuming the values of the spine at the branching points and the leaves converge to a sequence of i.i.d. random variables with law $\Pi^\infty$.
This yields the content of Theorem \ref{thm:k-spine-cv}.

\begin{challenge}
	The previous argument relies on a $k$-mixing property of the $1$-spine.
	This analysis will be carried out in Section \ref{sect:preliminaries} using Sturm--Liouville theory.
\end{challenge}

\subsection{Limiting moments}
\label{sect:lim-moments}
Let us now demonstrate the importance of Theorem \ref{thm:k-spine-cv}.
Let
$$
	\forall M = [X,d,\nu],  \ \  \tilde \Psi(M) \ := \ \int_{(X\times E)^k}  \prod_{i,j} \psi_{i,j}( d({v_i, v_j}) )  \prod_{i} \tilde\phi_{i}(x_i) h^\infty(x_i) \ \nu(dv_i \otimes dx_i),
$$
where $ \tilde \phi_i,\psi_{i,j}$  are bounded continuous functions  such that the $\tilde \phi_i$ are also compactly supported.
From the many-to-few formula, Theorem \ref{thm:k-spine-cv} entails

\begin{eqnarray*}
	{\mathbb E}_x\left[ \tilde \Psi(\bar M^L_{t}) \right]
	& = & \ \E_x\left[ \int\prod_{i,j}  \psi_{i,j}( \bar d_t({v_i, v_j}) )  \prod_{i} \tilde\phi_{i}(x_i) h^\infty(x_i) \ \bar   \nu_t^L(dv_i \otimes dx_i) \right] \\
	& \approx & \ \E_x\left[ \int\prod_{i,j} \mathbf{1}_{\{v_i\neq v_j\}} \psi_{i,j}( \bar d_t({v_i, v_j}) )  \prod_{i} \tilde\phi_{i}(x_i) h^\infty(x_i) \ \bar   \nu_t^L(dv_i \otimes dx_i) \right] \\
	&  \approx &  \frac{2 h^\infty(x)}{ N \Sigma^2 t }  \ \times \  k ! \left( t \frac{\Sigma^2}{2}\right)^{k} {\mathbb E}\left[ \prod_{i,j} \psi_{i,j}( U_{\sigma_i,\sigma_j}) \right]
	\prod_{i} \int_{{\mathbb R}_+} \tilde \phi_i(x) \Pi^\infty(dx).
\end{eqnarray*}
Let us formally take $\psi_{i,j}\equiv 1$ and $\tilde \phi_i \equiv 1/h^\infty$ in the previous expression (note that this is problematic since $\tilde \phi_i$ is neither bounded nor compactly supported, see Challenge \ref{thm:k-spine-cv} below). Then for large $N$,
$$
	{\mathbb E}_x\left[ \left(\frac{1}{N} Z_{tN}^L\right)^k \right]  \ \approx \ \  \frac{2 h^{\infty}(x)}{ N \Sigma^2 t } \  \ \times \ \ \underbrace{ k ! \left( t \frac{\Sigma^2}{2}\right)^{k}}_{\mbox{exponential moments}} \times \quad { \underbrace{\left(\int_0^\infty \tilde h^\infty(x) dx\right)^k}_{=1}},
$$
where we used the fact that $\tilde h^\infty = \Pi^\infty/h^\infty$.
Now, the RHS coincides with the moments of a r.v. with law
$$
	\left( 1-\frac{2 h^\infty(x)}{ N \Sigma^2 t }\right) \delta_0(dx) + \frac{2 h^\infty(x)}{ N \Sigma^2 t } \exp\left(- \frac{2 x}{\Sigma^2 t}  \right) \frac{2 dx}{\Sigma^2 t}.
$$
If we identify the Dirac measure at $0$ with the extinction probability,
this suggests the Kolmogorov estimate and the Yaglom law exposed in Theorem \ref{thm:Kolmogorov} and Theorem \ref{thm:Yaglom}.
Furthermore, if we replace $\tilde \phi_i$ by $\phi_i/h^\infty$ in Theorem \ref{thm:k-spine-cv} (again a problematic step), the previous estimates entail
$$
	{\mathbb E}_x\left[ \Psi(\bar M^L_{t}) \big| Z_{tN}^L>0  \right] \approx   k !  {\mathbb E}\left[ \prod_{i,j} \psi_{i,j}( U_{\sigma_i,\sigma_j}) \right]
	\prod_{i} \int_{{\mathbb R}_+}  \phi_i(x) \left( t \frac{\Sigma^2}{2}\right){ \tilde h^\infty(x) dx.}
$$
where $\Psi(M)$ is now an arbitrary product polynomial of the form
$$
	\Psi(M) \ := \ \int \prod_{i,j} \psi_{i,j}( {d(v_i,v_j)} ) \prod_{i} \phi_{i} (x_i) \nu(dv_i \otimes dx_i).
$$
According to Proposition \ref{moments:CPP}, this coincides with the moments of the Brownian CPP described in Theorem \ref{thm:main-theorem}.

\begin{challenge}
	The previous computation only suggests that the probability for the population size to be $o(\frac{1}{N})$ is given by the Kolmogorov estimate.
	Intuitively, the Dirac mass above corresponds to a population whose size becomes invisible at the limit after rescaling the population by $N$.
	It thus remains to show that if the population is small compared to $N$ then it must be extinct. This will be carried out in Section \ref{sect:survival proba}.
\end{challenge}
\begin{challenge}\label{ref:challeng-explosion}
	Going from Theorem \ref{thm:k-spine-cv} to the convergence of the $\bar M_{t}^L$ requires to use test functions exploding at the boundary.
	To overcome this technical difficultly, we will impose an extra thinning of the population by killing all the particles close to the boundaries at time $tN$. This final technical step will be carried out in Section \ref{sect:final-cutoff} using some general properties of the Gromov-weak topology.
\end{challenge}

\subsection{Choosing the cutoff $L$}
\label{sect:choice-of-L}

We now motivate our choice for $L$.
According to the previous arguments, we want to choose $L$ large enough such that, on a time scale $O(N)$,
\begin{itemize}
	\item [(i)]
	      The particles do not reach $L$ with high probability.
	      This will imply that $\bar M_{t}^L$ and $\bar M_{t}$ coincide with high probability.
	\item[(ii)] The $1$-spine reaches equilibrium in a time $o(N)$ {\it regardless of its initial position} in $[0,L]$. This is needed to justify the calculations
	      of Section \ref{sect:k-spine-CV}.
\end{itemize}

{\bf(i) Hitting the right boundary.}
Let $E$ be a compact set in the vicinity of the boundary $L$ (say $[L-2,L-1]$).
Recall from the discussion after Proposition \ref{prop:first} that for $y\geq 1$,
$$
	h^\infty(y)  { =} \frac{1}{\tilde c \|v_1^\infty\|^2}\, e^{\beta} e^{(\mu -\beta)y} \quad \text{and} \quad \tilde h ^\infty(y){=}\tilde c e^{\beta}e^{-(\mu+\beta)y}.
$$
A direct application of Lemma \ref{lem:many-to-one} to $f(y) = \mathbf{1}_{y\in E}$ implies that
\begin{align*}
	\mathbb{E}_x\left[\sum_{v\in \mathcal N^L_{tN}}\mathbf{1}_{x_v\in E} \right] & =  \int_{E} \frac{h(0,x)}{h(tN,y)} q_{tN}(x,y) dy  \approx  \int_{E} \frac{h^\infty(x)}{h^\infty(y)} \Pi^\infty(dy)                                \\
	                                                                             & \approx h^\infty(x) \int_E \tilde h^{\infty}(y)dy \approx   h^\infty(x) O ( \overbrace{e^{-\frac{\mu+\beta}{\mu-\beta} \log(N)}}^{=N^{-\alpha}} ),
\end{align*}
where we used that $\tilde h(tN,y)\approx \tilde h^\infty(y)$ for $N$ large. The last approximation holds under the assumption that $E$ is a compact set close to $L$ and follows from \eqref{def:L}.
Integrating this  on $[0,tN]$ shows that the occupation time of the set $E$ on the time interval $[0,tN]$ is $O(N^{1-\alpha})$.
We then recall that the probability of survival is of order $1/N$ so that the occupation time of the conditioned process is $O(N^{2-\alpha})$.
Using that $\alpha>2$ under \eqref{hfp}, this yields the desired estimate.

	{\bf (ii) Mixing time.}
Recall from the discussion after Proposition \ref{prop:first} that $v_1(x) \approx e^{-\beta(x-1)}$ for $x\geq 1$ and $L$ large enough.
As a consequence, the $1$-spine (see Definition \ref{def:1-spine}) is well approximated on $[1,\infty)$ by the diffusion
$$
	dz_t = -\beta dt + d B_t.
$$
A good proxy for the mixing time is the first returning time at $1$ which is of the order $\log(N)=o(N)$
for every $x\in[1,L]$, as desired. A more refined analysis will be carried out in Section \ref{sect:preliminaries}.

\section{The many-to-few theorem}
\label{sect:many-to-few}

\subsection{The general case}
In this section, we
consider a  general branching diffusion killed at the boundary of a regular domain $\Omega \subset \mathbb{R}^d$.
Unless otherwise specified, we use the same notation as in the previous sections. We will assume that
\begin{enumerate}
	\item The generator of a single particle is given by a differential operator (in the divergence form)
	      \begin{align}\label{generator}
		      \mathcal{G}f(x) \  & = \  \frac{1}{2} \sum_{i,j}\partial_{x_i}\left( a_{ij}(x) \partial_{x_j} f(x)\right) \ + \ \sum_{i} b_i(x) \partial_{x_i} f(x),\quad x\in\Omega, \\ f(x)&=0,\quad x\in\partial \Omega. \nonumber
	      \end{align}
	      We assume that $(a_{ij})$ is uniformly elliptic, which means that there exists a constant $\theta>0$ such that for all  $\xi \in \mathbb{R}^d$ and a.e.~$x\in\Omega$, $\sum_{i,j=1}^{d}a_{ij}(x)\xi_i,\xi_j\geq \theta\|\xi\|^2$ (see \cite[§6.1]{evans2010partial}). In addition, we assume that $a_{i,j}\in\mathcal{C}^1(\Omega)$ and $\sup_{x\in\Omega}|b_i(x)|<\infty$ for all $1\leq i,j\leq d$.
	\item A particle at location $x$ branches into two particles at rate $r(x)$ (we only consider binary branching).
\end{enumerate}
We denote by ${\cal N}_t$  the set of particles alive at time $t$.
For any pair of particles $v,w\in{\cal N}_t$,  we write $|u\wedge v|$ for the most recent time at which $v$ and $u$ had a common ancestor and their genealogical distance is defined as
\begin{equation*}
	d_t(u,v)=t-|u\wedge v|.
\end{equation*}
Finally, we define the random mmm space
\begin{equation}\label{eq:mmm_bbm}
	M_t \ = \ [{\cal N}_t, d_t,\nu_t], \ \ \mbox{where $\nu_t = \sum_{v\in{\cal N}_t} \delta_{v,x_v}$}.
\end{equation}

We say that a function $h(x)$ is harmonic if and only $h$ is positive on $\Omega$ and satisfies the Dirichlet problem
\begin{align}\label{eq:def_A}
	\mathcal{A}h(x)  := {\cal G}h(x) + r(x) h(x) & = 0, \ \ \mbox{for  $x\in \Omega$,}                 \\
	h(x)                                         & =0, \ \ \mbox{for $x\in\partial \Omega$} .\nonumber
\end{align}
In our application, we will consider the domain  $\Omega=[0,\infty)\times [0,L]$ where the first coordinate will correspond to the time variable seen as a mark (see Section \ref{sec:subc:bbm}).
In this case, there exists a unique harmonic function (up to constant multiplies) given by \eqref{def:hL}. In the general case, if $\mathcal{A}$ has multiple harmonic function, we choose one and define a $1$-spine process as follows.

\begin{definition}\label{def:spine:g}
	Let  $h$ be a harmonic function.
	The $1$-spine process (associated to $h$) is the process whose generator is given by the Doob $h$-transform
	of the differential operator ${\cal A}$
	\begin{align*}
		\frac{{\cal A}(hf)}{h}(x) & = \mathcal{G}f(x)+\sum_{i,j}a_{i,j}(x)\frac{\partial_{x_j}h(x)}{h(x)}\partial_{x_i} f(x), \\
		f(x)                      & =0, \quad x\in \partial \Omega,
	\end{align*}
	where the first equality is a direct consequence of the fact that $h$ is harmonic. We will denote by
	$q_{t}(x,y)$  the transition probability of the $1$-spine process. {The $k$-spine distribution $Q^{k,t}_x$ is defined as in
			Section \ref{sect:k-spine-CV}.} In particular, we recall that we write $\mathcal{B}$ for the set of branching points of the $k$-spine, $\mathcal{L}$ for the set of leaves, $V_i,\cdots,V_K$ for an enumeration of the leaves from left to right,  and $\zeta_v$ for the mark of $v\in\mathcal{B}\cup\mathcal{L}$. We further define
	\begin{equation}
		\label{eq:def_delta}
		\Delta:=\prod_{v\in\mathcal{B}}r(\zeta_v)h(\zeta_v)\prod_{i=1}^k\frac{1}{h(\zeta_{V_i})}.
	\end{equation}
\end{definition}
This section is devoted to the proof of the following result.
\begin{theorem}[Many-to-few]\label{many-to-few00}
	Let $t>0$ and $x\in\Omega$.
	Let $(\phi_i; i\in[k])$ and $(\psi_{i,j}; i,j \in[k])$ be continuous bounded functions.
	Consider the product polynomial defined by
	$$
		\forall M=[X,d,\nu], \ \  \Psi(M) \ := \ \int \prod_{i,j} \mathbf{1}_{\{v_i\neq v_j\}} \psi_{i,j}( d({v_i, v_j}) ) \prod_i \phi_{i} (x_i) \nu(dv_i \otimes dx_i).
	$$
	Then
	$$ \mathbb {E}_x\left[ \Psi(M_t) \right] \ =  k ! \  h(x) \ t^{k-1} Q_{x}^{k,t}\left( \Delta \prod_{i,j} \psi_{i,j}( U_{\sigma_i,\sigma_j}) \prod_{i} \phi_i(\zeta_{V_{\sigma_i}})    \right),
	$$
	where $(U_{i,j})$ is as in \eqref{eq:UPP} and $\sigma$ is an independent random permutation of $[k]$.
\end{theorem}

The main idea for the proof of Theorem \ref{many-to-few00} consists in using the branching property  to derive a recursion formula satisfied both by the planar moments of the branching diffusion and by the law of the $k$-spine.
First, this formula will be derived for the \textit{biased spine measure} $\mathbf{L}^{k,t}_x$  given by
\begin{equation}\label{eq:def_L}
	\frac{\mathrm{d} \mathbf{L}^{k,t}_x}{\mathrm{d} Q^{k,t}_x}= t^{k-1}\Delta,
\end{equation}
where $\Delta$ is defined in \eqref{eq:def_delta}.
Second,  we will show that the same relation holds for the moments of $M_t$. This step will rely on
a uniform planarisation of the branching diffusion that we now describe.

At every time $t>0$, every particle is endowed with a mark $x_v$ (the position of the particle) and a Ulam-Harris label $p_v$,
where $p_v\in \cup_{n\in \mathbb{N}} \{0,1\}^n$. As before, ${x_v}$ denotes the position of the particle. The planarisation labels $p_v$
are assigned recursively as follows. We label the root with $\emptyset$ and
\begin{enumerate}
	\item At every branching point $v$, we distribute the
	      labels $(p_v,0)$ and $(p_v,1)$ uniformly among the two children: $(p_v,0)$ (resp.~$(p_v,1)$)
	      is said to the left (resp.~right) child of $v$.
	\item The label  $p_v$ does not vary between two branching points, i.e., $p_{v_1}=p_{v_2}$
	      if the trajectory connecting $v_1$ and $v_2$ does not encounter any branching points.
\end{enumerate}
Let ${\cal N}_t^{pl}$  be the set of particles at time $t$ in the planar branching diffusion. The genealogical distances and the marks of the planar branching diffusion will be encoded by marked binary planar ultrametric matrices that we now define.

We say that a matrix $(U_{i,j})_{1\leq i,j\leq k}$ is is \textit{planar ultrametric} if
\begin{equation*}
	\forall i<j<l, \quad U_{i,l}=U_{i,j}\vee U_{j,l}.
\end{equation*}
Moreover, the matrix $(U_{i,j})_{1\leq i,j\leq k}$ is said \textit{binary} if
\begin{equation*}
	U_{i,j}\neq U_{k,l} \quad \Leftrightarrow\quad (i,j)\neq (k,l) \quad \text{and} \quad  (i,j)\neq (l,k).
\end{equation*}
We denote by $\mathbb U_k$ the set of binary  planar ultrametric matrices of size $k$. Let $\mathbb{U}_k^*=\mathbb{U}_k\times E^k$ be  the set of marked binary planar distance matrices.

Note that the Ulam-Harris $(p_v)$ labelling induces an order on ${\cal N}_t^{pl}$. In particular, for  every $k$-uplet $v_1<v_2\cdots< v_k$ in ${\cal N}_t^{pl}$ and   $\vec{v}=(v_1,...,v_k)$, the marked distance matrix of the sample $\vec{v}$,
\begin{equation*}
	U(\vec{v}):=\left((d_t(v_i,v_j),(x_{v_i}))\right),
\end{equation*}
is an element of $\mathbb{U}_k^*$.

Our recursion formula on the moments of $M_t$ will be obtained by dividing every ordered $k$-uplet in ${\cal N}_t^{pl}$ into two subfamilies,  the descendants of the left (resp.~right) child of the Most Recent Common Ancestor (MRCA) of the sample. This will by achieved by partioning  $[k]$  as follows.
For  ${U}=((U_{i,j}),(w_i))\in\mathbb{U}_k^*$, define $\tau(U) =\max_{i\neq j} U_{i,j}$. In words, $\tau$ is the time to the MRCA of the sample.
We say  the integers $i$ and $j$ are in the same block iff $U_{i,j}<\tau$. Since $U$ is a binary planar ultrametric matrix, there exists $n\leq k-1$ such that this partition can be written as $\{\{1,...,n\},\{n+1,...,k\}\}$.
We  denote by $T_0(U)$ and $T_1(U)$ the corresponding sub-matrices obtained from this partition and write $|T_0(U)|$ and $|T_1(U)|$ for the sizes of the two blocks.
Note that $U_{i,j}$ is equal to $\tau(U)$ if $i$ and $j$ do not belong to the same block, and to $(T_0)_{i,j}$ (resp.~$(T_1)_{i-|T_0|,j-|T_0|}$) if they both belong to the first (resp.~second) block.

\begin{proposition}\label{lem:23455}
	Let $k\in\mathbb{N}$, $t>0$ and $x>0$. Let $(U_{i,j})$ be as in \eqref{eq:UPP} and recall the definition of $\Delta$ from \eqref{eq:def_delta}. Let $R^{k,t}_x$  be the measure on $\mathbb{U}_k^*$ such that for every bounded measurable function $F:\mathbb{U}_k^*\to\mathbb{R}$,
	$$
		R^{k,t}_x(F) \ := \ \mathbb{E}_{x}\left(\sum_{\substack{v_1 < \cdots < v_k\\ v_i \in {\cal N}_t^{pl}}} F( U(\vec{v}))     \right).
	$$
	Then
	$$
		R^{k,t}_x(F) \ = \ \ h(x) \mathbf{L}_x^{k,t}(F),
	$$
	where,  by a slight abuse of notation, we write
	\begin{equation}\label{eq:def_not_Q}
		\mathbf{L}_x^{k,t}(F)= t^{k-1}Q_x^{k,t}\left(\Delta F\right)=t^{k-1}Q_x^{k,t}\left(\Delta F((U_{i,j}),(\zeta_{V_i})\right).
	\end{equation}
\end{proposition}
In order to prove this result, we first show that $\mathbf{L}_x^{k,t}(F)$ satisfies a recursive relation for functionals $F:\mathbb{U}_k^*\to\mathbb{R}$ of the product form
\begin{equation}
	\label{eq:prod_form}
	F(U) \ = \ \mathbf{1}_{\left\{|T_0|=k-n,|T_1|=n\right\}}f(\tau(U)) \psi_0(T_0(U)) \psi_1(T_1(U)),
\end{equation}
where   $n\in[k-1]$ and $f:\mathbb{R^+}\to \mathbb{R}$, $\psi_1:\mathbb U_n^*\to \mathbb{R}$, $\psi_0:\mathbb U_{k-n}^*\to \mathbb{R}$  are bounded measurable functions.

\begin{proposition}\label{cor:M}
	Let $t>0$ and $x>0$. For functionals $F$  of the product form \eqref{eq:prod_form},
	\begin{equation*}
		\mathbf{L}_x^{k,t}(F)=\int_0^t f(s)\mathbb{E}_x\left[r(\zeta_{t-s})h(\zeta_{t-s})\mathbf{L}_{\zeta_{t-s}}^{n,s}(\psi_0)\mathbf{L}_{\zeta_{t-s}}^{k-n,s}(\psi_1)\right]ds.
	\end{equation*}
\end{proposition}
\begin{proof}
	This result is a direct consequence of the construction of the $k$-spine tree (see Section~\ref{sect:k-spine-CV}). Indeed, this construction implies that:
	\begin{enumerate}
		\item[(i)] The depth $\tau$ of the deepest branching point is distributed as $\max(U_1,\cdots,U_{k-1})$,
		      where the $U_i$'s are i.i.d. random variables uniform on $[0,t]$. The density of this variable is given by $s\mapsto \tfrac{k-1}{t^{k-1}} s^{k-2}$.
		\item[(ii)] The index $I$ of the deepest branch is chosen uniformly at random in $[k-1]$.
		\item[(iii)] Let us condition on $I=n$,  $x_{MRCA}(\mathcal{T})=y$ and $\tau=s$. Then the left and right subtrees $\mathcal{T}_0$ and $\mathcal{T}_1$ are independent
		      and $\mathcal{T}_0$ (resp.~$\mathcal{T}_1$) is distributed as a $n$- (resp.~$k-n$-) spine tree with depth $s$.
		      By the Markov property, both spine trees are rooted at $y$.
		\item[(iv)] $x_{MRCA}(\mathcal{T})$ is distributed as the $1$-spine at time $t-s$.
		\item[(v)]If $x_{MRCA}(\mathcal{T} )=y, \tau=s$, then
		      $$
			      \Delta(\mathcal{T}) ={r(y)} h(y) \Delta(\mathcal{T}_0 )  \Delta(\mathcal{T}_1).
		      $$
	\end{enumerate}
	Putting all of this together, we obtain
	\begin{align*}
		 & Q^{k,t}_x(\Delta F)                                                                                                                                                                                                                                                     \\
		 & =   \int_0^t f(s) \underbrace{\frac{(k-1) s^{k-2}}{t^{k-1}}}_{(i)}  \int_{y\in\Omega}  \overbrace{q_{t-s}(x,y)}^{(iv)} \underbrace{r(y) h(y)}_{(v)}\underbrace{\frac{1}{k-1}}_{(ii)} \underbrace{Q_y^{n,s}(\Delta \psi_0)  Q_{y}^{k-n,s}(\Delta \psi_1)}_{(iii)}  dy ds \\
		 & =\frac{1}{t^{k-1}}\int_0^t f(s) \mathbb{E}_x\left[r(\zeta_{t-s})h(\zeta_{t-s})(s^{n-1}Q_y^{n,s}(\Delta \psi_0))(s^{k-n-1}Q_y^{k-n,s}(\Delta \psi_1)) \right].
	\end{align*}
	This concludes the proof of the proposition.
\end{proof}

We now move to the proof of our many-to-few formula in the case $k=1$.
\begin{lemma}[Many-to-one lemma]\label{lem:many-to-one}
	For every bounded measurable function $f$ $$
		\mathbb{E}_{x}\left[\sum_{v\in{\cal N}_t} f({x}_{v})\right] \ = \ \int_{\Omega} f(y)  q_{t}(x,y)  \frac{h(x)}{h(y)} dy=h(x)Q_x^{1,t}(\Delta f(\zeta_t)),
	$$
	with $\Delta=\tfrac{1}{h(\zeta_t)}$ and $\zeta_t$ is the unique leaf of the $1$-spine tree of depth $t$.
\end{lemma}
\begin{proof} The proof is similar to that of Lemma \ref{lem:many-to-one0}.
	One can readily check that $p_t(x,y):= q_{t}(x,y)  \frac{h(x)}{h(y)}$ is the fundamental solutions of the PDE
	\begin{equation*}
		\begin{cases}
			\partial_tu(t,y)=\mathcal{A}^*u(t,y), & y\in\Omega,          \\
			u(t,y)=0,                             & y\in\partial \Omega,
		\end{cases}
	\end{equation*}
	where $\mathcal{A}^*:=\frac{1}{2}\nabla \cdot a\nabla-b\cdot \nabla-\nabla \cdot b+ r(x)$ is the adjoint of the differential operator $\mathcal{A}$.
\end{proof}

We are ready to address the case $k>2$.

\begin{proof}[Proof of Proposition~\ref{lem:23455}]
	We will show the result by induction on $k$.
	The case $k=1$ is the many-to-one lemma (see Lemma \ref{lem:many-to-one}).
	For $k\geq 2$, it is sufficient to prove the result for functionals $F$ of the product form \eqref{eq:prod_form}. Indeed, this set of functions is separating for $\mathbb{U}_k^*$.

	Consider the branching diffusion starting from a singe particle at $x$ up to time $t$. Let $s_1<s_2<\cdots$ be the successive branching times of the process
	and let $x_i$ be the spatial coordinate of the branching particle at time $s_i$. Recall that our branching diffusion is planarised so that we can distinguish between the left and right descendants of $(x_i,t_i)$.
	For every $i$, let ${\cal N}_{0,i}$  (resp.,  ${\cal N}_{1,i}$) be the the set of left (resp., right) descendants of $(x_i,s_i)$.
	Let $F$ be a functional of the product form \eqref{eq:prod_form}. Then, we have
	$$
		R_{x}^{k,t}{(F)} \ = \ \E_{x}\left( \sum_{i: s_i<t} f(t-s_i)
		\sum_{\substack{v_1<\cdots<v_n  \\ v_j\in {\cal N}_{0,i}, |v_j|=t}}\psi_0(U(\vec{v}))
		\sum_{\substack{w_1<\cdots<w_{k-n} \\ w_j\in {\cal N}_{1,i}, |w_j|=t}}\psi_1(U(\vec{w}))   \right).
	$$
	We also recall that the particles branch independently with rate $r(x)$ when at $x$. This means that, given $\mathcal{N}^{pl}_{t-s}$ and $(x_v)_{v\in\mathcal{N}^{pl}_{t-s}}$, the probability that the particle at $x_v$
	branches between times $t-s$ and $t-s+h$ is $r(x_v)h+o(h)$. Conditioning on $\mathcal{N}^{pl}_{t-s}$ and using the Markov property, the previous formula yields
	that
	\begin{eqnarray*}
		R^{k,t}_x(F)  =  \mathbb{E}_{x}\left(  \int_0^t f(s) \sum_{v\in{\cal N}_{t-s}^{pl}} { r({x_v})}  \times  \mathbb{E}_{{x_v}} \left(\sum_{\substack{v_1<\cdots<v_n\\ |v_j|=s}} \psi_0({U}(\vec{v})) \right)  \mathbb{E}_{{x_v}}\left( \sum_{\substack{w_1<\cdots<w_{k-n}\\ |w_j|=s}} \psi_1({U}(\vec{w})) \right) ds \right).
	\end{eqnarray*}
	By the many-to-one formula, the RHS is equal to
	\begin{eqnarray*}
		h(x)  \int_0^t f(s) \int_{y\in\Omega} r(y) \frac{q_{t-s}(x,y)}{ h(y)}
		\mathbb{E}_{y} \left(\sum_{\substack{v_1<\cdots<v_n\\ |v_j|=s}} \psi_0({U}(\vec{v})) \right)  \mathbb{E}_{y}\left( \sum_{\substack{w_1<\cdots<w_{k-n}\\ |w_j|=s}} \psi_1({U}(\vec{w})) \right)  dy ds .
	\end{eqnarray*}
	By induction
	\begin{eqnarray*}
		\mathbb{E}_{y}\left( \sum_{v_1<\cdots<v_{n}, |v_j|=s} \psi_0({U}(\vec{v})) \right) & =  & h(y) s^{n-1} Q^{n,s}_y(\Delta \psi_0), \\
		\mathbb{E}_{y}\left( \sum_{w_1<\cdots<w_{k-n}, |w_j|=s} \psi_1({U}(\vec{w})) \right) & = &  h(y) s^{k-n-1} Q^{k-n,s}_y(\Delta \psi_1).
	\end{eqnarray*}
	As a consequence,
	\begin{eqnarray*}
		R^{k,t}_x(F)
		& = &   h(x)   \int_0^t f(s) s^{k-2} \int_{y\in\Omega}  r(y)\frac{q_{t-s}(x,y) h(y)^2}{h(y)}  Q_y^{n,s}(\Delta\psi_0)  Q_{y}^{k-n,s}(\Delta \psi_1)  dy ds  \\
		& = &  h(x) t ^{k-1}   \int_0^t f(s) \frac{ s^{k-2}}{t^{k-1}}  \int_{y\in\Omega} r(y) q_{t-s}(x,y) h(y) \\
		&&\qquad\qquad\qquad \times  Q_y^{n,s}(\Delta \psi_0)  Q_{y}^{k-n,s}(\Delta \psi_1)  dy ds.
	\end{eqnarray*}
	It remains to show that the RHS of the above is equal to $h(x)t^{k-1}Q_x^{k,t}(\Delta F)$. Yet, this is precisely the content of Proposition \ref{cor:M}.
\end{proof}

\begin{proof}[Proof of Theorem \ref{many-to-few00}]
	Let $[X,d,\nu]$ be a random mmm-space  such that the support of $\nu$
	is a set of cardinality $k$ a.s.. In our case, the support is given by the set of $k$ leaves of the sample. Define
	$$
		\forall U\in \mathbb{U}_k^*, \quad F(U) \ = \  \frac{1}{k!} \sum_{\sigma \in S_k}  \prod_{i,j} \psi_{i,j}(d(V_{\sigma_i}, V_{\sigma_j})) \prod_{i} \phi_{i}(x_{V_{\sigma_i}}) ,
	$$
	where $(V_1,\cdots,V_k)$ is an arbitrary labelling of the support of $\nu$ and $S_k$ is the set of permutations of $[k]$.
	Note that this functional does not depend on the labelling so that $F$ is constant on a given isometry class.

	We have
	\begin{eqnarray*}
		&&\E_{x}\left[\sum_{v_1\neq \cdots \neq v_k \in{\cal N}_t} \prod_{i,j} \psi_{i,j}(d(v_{i}, v_j)) \prod_{i} \phi_{i}(x_{v_i})  \right]\\
		& = & \E_{x}\left[k! \sum_{v_1< \cdots < v_k \in{\cal N}_t^{pl}} \frac{1}{k!} \sum_{\sigma \in S_k }  \prod_{i,j} \psi_{i,j}(d(v_{\sigma_i}, v_{\sigma_j})) \prod_{i} \phi_{i}(x_{v_{\sigma_i}})  \right] \\
		& = &  k! \E_{x}\left[\sum_{v_1< \cdots < v_k \in{\cal N}_t^{pl}} F( { U}(\vec{v})  )  \right] .
	\end{eqnarray*}

	The result immediately follows from Proposition \ref{lem:23455}.
\end{proof}

\subsection{Subcritical operators}\label{ref:subcritical_diff}

Let $\mathcal{G}$ and $\mathcal{A}$ be as in \eqref{generator} and \eqref{eq:def_A}.

Assume that there exist a positive function $H:\Omega\to \mathbb{R}$  and a real  positive number $w>0$ such that
\begin{equation}\label{eq:def_H}
	\mathcal{A}H(x)=-wH(x), \quad \text{for $x\in\Omega$}, \quad H(x)=0, \quad \text{for $x\in\partial \Omega$}.
\end{equation}
(Note that, in this case, the differential operator $\mathcal{A}$ is subcritical in the sense of \cite{pinsky95}, see Proposition 4.2.3 in \cite{pinsky95}.)

Then ${h}(t,x):=e^{wt}H(x)$  is harmonic for the operator
\begin{align*}
	\mathbf{A}f(t,x):=\partial_t f(t,x)+\mathcal{A}f(t,x), \quad & (t,x)\in (0,\infty)\times\Omega ,          \\
	f(t,x)=0, \quad                                              & (t,x)\in (0,\infty)\times \partial \Omega,
\end{align*}
defined on the domain $(0,\infty)\times \Omega$. The generator of the $1$-spine process $(t,\zeta_t)$ associated to the critical operator $\mathbf{A}$ is thus given by (see Definition~\ref{def:spine:g})
\begin{equation}\label{eq:spine_sc}
	\frac{\mathbf{A} (hf)}{h}(t,x)=\partial_t f(t,x)+\mathcal{G}f(t,x)+\sum_{i,j}a_{i,j}(x)\frac{\partial_{x_j}h(t,x)}{h(t,x)}\partial_{x_i} f(t,x).
\end{equation}
We write $Q_{(s,x)}^{k,t}$ for the distribution of the $k$-spine tree of depth $t$ rooted at $(s,x)$.

By a slight abuse of notation, we will often drop the time component (that is trivial) and refer to the spatial component $\zeta_t$ as the $1$-spine process of the branching diffusion. The generator of the spatial process $\zeta$ is given by
\begin{align} \label{eq:spine_process}
	\frac{\mathcal{A}(Hf)}{H}=\mathcal{G}f(x)+\sum_{i,j}a_{i,j}(x)\frac{\partial_{x_j}H(x)}{H(x)}\partial_{x_i} f(x),\quad & x\in \Omega,                    \\
	f(x)=0,\quad                                                                                                           & x\in  \partial \Omega.\nonumber
\end{align}
Similarly, we will write $Q^{k,t}_x$ for the law of the $k$-spine tree of depth $t$ rooted at $(0,x)$.

A direct application of Theorem \ref{many-to-few00} entails the following result.

\begin{cor}[Many-to-few]\label{th:many-to-few}
	Let $t>0$. Recall the definition of $M_t$ from \eqref{eq:mmm_bbm}.
	Let $(\phi_i; i\in[k])$ and $(\psi_{i,j}; i,j \in[k])$ be measurable bounded functions.
	Consider the product polynomial defined as
	$$
		\forall M=[X,d,\nu], \ \  \Psi(M) \ := \ \int \prod_{i,j} \mathbf{1}_{\{v_i\neq v_j\}} \psi_{i,j}( d({v_i, v_j}) ) \prod_i \phi_{i} (x_i) \nu(dv_i \otimes dx_i).
	$$
	Then
	$$ \mathbb {E}_x\left[ \Psi{ (M^L_{t})} \right] \ = \  k! \  { h(0,x)} \ t^{k-1} Q_{x}^{k,t}\left( \Delta \prod_{i,j} \psi_{i,j}( U_{\sigma_i,\sigma_j}) \prod_{V_i\in{\cal L}} \phi_{i} (\zeta_{V_{\sigma_i}})    \right),
	$$
	where  $Q^{k,t}_{x}\equiv Q^{k,t}_{(0,x)}$, the matrix $(U_{i,j})$ is as in \eqref{eq:UPP}, $\sigma$ is an independent permutation of $[k]$,
	\begin{equation}\label{eq:delta_sc}
		\Delta:= \prod_{v\in {\cal B}} r({\zeta_v}) h(|v|,\zeta_{v}) \prod_{v\in{ \cal L}} \frac{1}{h(t,{\zeta_v})},
	\end{equation}
	and $|v|$ refers to the time component of the branching point $v\in\mathcal{B}$.
\end{cor}

\begin{proposition}[Recursive definition of the biased spine measure]\label{prop:rec}
	Recall the definition of the 1-spine process $(t,\zeta_t)_{t\geq 0}$ from \eqref{eq:spine_sc}. We also recall from \eqref{eq:def_not_Q} that, for  $F:\mathbb{U}_k^*\to \mathbb{R}$ and $s,t>0$,
	\begin{equation*}
		\mathbf{L}_{(s,x)}^{k,t}(F)=t^{k-1}Q_{(s,x)}^{k,t}(\Delta F(U_{i,j},\zeta_{V_i})),
	\end{equation*}
	is a measure on the $k$-spine trees of depth $t$ rooted at $(s,x)$, referred to as the biased spine measure.
	To ease the notation, we write $\mathbf{L}_x^{k,t}$ for the measure on the $k$-spine trees of depth $t$ rooted at $(0,x)$.
	Then, the family of  biased spine measures $\mathbf{L}$  is such that
	\begin{itemize}
		\item for every bounded function $f:\Omega\to\mathbb{R}$,
		      \begin{equation} \label{eq:biased_init}
			      \forall x\in \Omega, \; \forall t>0, \quad \mathbf{L}_x^{1,t}(f)=e^{-wt}\mathbb{E}_x[f(\zeta_{t})],
		      \end{equation}
		\item for all test functions of the product form \eqref{eq:prod_form}, we have
		      \begin{equation}\label{eq:biased_rec}
			      \forall x\in\Omega, \;\forall  t>0,\quad  \mathbf{L}^{k,t}_{x}(F)=\int_{0}^{t}f(s)e^{-w(t-s)}\mathbb{E}_x\left[r(\zeta_{t-s})H(\zeta_{t-s})\mathbf{L}_{\zeta_{t-s}}^{n,s}(\psi_0)\mathbf{L}_{\zeta_{t-s}}^{k-n,s}(\psi_1)\right]ds,
		      \end{equation}
		      where $H$ is as in \eqref{eq:def_H}.
	\end{itemize}
\end{proposition}
\begin{proof} 
	We first remark that
	\begin{equation*}
		h(t-s,y)=e^{w(t-s)} h(0,y)=e^{w(t-s)}H(y).
	\end{equation*}
	Moreover, note that
	\begin{equation*}
		Q_{(t-s,y)}^{k_i,s}(\Delta \psi_i)=e^{-w(t-s)}Q_{(0,y)}^{k_i,s}( \Delta \psi_i), \quad i=0,1, \quad (k_0,k_1)=(n,k-n).
	\end{equation*}
	The latter formula is obtained by shifting time by $t-s$ in the $k_i-1$ factors $h$ and in the $k_i$ factors $\tfrac{1}{h}
	$ of $\Delta$ (see Equation \eqref{eq:delta_sc}). Equation \eqref{eq:biased_rec} then follows from Proposition \ref{cor:M} . For $k=1$, \eqref{eq:biased_init} follows from Corollary~\ref{th:many-to-few}.
\end{proof}

\subsection{The BBM in an interval}\label{sec:subc:bbm}

We now apply the previous setting to the problem at hand and consider the BBM $\mathbf{X}^L$ defined in Section~\ref{sec:model}.
\begin{lemma}\label{lem:check1}
	Let $\mu,\lambda_1,\lambda_1^\infty, r(x)$ be defined as in Section \ref{sect:intro}.
	Consider the operator
	\begin{align*}
		\mathbf{A}f(t,x) & =\partial_t f(t,x)+\frac{1}{2}\partial_{xx}f(t,x)-\mu\partial_xf(t,x), \quad \
		\forall (t,x)\in (0,\infty)\times (0,L)                                                           \\
		f(t,0)           & =f(t,L)=0, \quad \forall t\in(0,\infty).
	\end{align*}
	Then, the function $(t,x)\mapsto h(t,x)$ defined  in \eqref{def:hL}
	is harmonic for the operator $\mathbf{A}$ and
	\begin{equation}\label{eq:generator_sc_spine}
		\frac{{\mathbf{A}}(h f)}{h}(t,x)  \ = \ \partial_t f(t,x) + \  \frac{1}{2} \partial_{xx}f(t,x)+\frac{v_1'(x)}{v_1(x)} \partial_{x} f(t,x).
	\end{equation}
	The $1$-spine process $(\zeta_t)$ is the diffusion whose generator is given by
	\begin{align*}
		\  \frac{1}{2} \partial_{xx}f(x)+\frac{v_1'(x)}{v_1(x)} \partial_{x} f(x),\quad  x\in(0,L), \quad 	f(0)=f(L)=0.
	\end{align*}
\end{lemma}
\begin{proof}
	{This is a straightforward consequence of \eqref{eq:spine_sc} and \eqref{eq:spine_process}.}
\end{proof}

\begin{proof}[Proof of Proposition \ref{th:many-to-few2}]
	This is a direct consequence Corollary \ref{th:many-to-few} after rescaling the measure by $N$ and time by $N$.
\end{proof}

\section{Spectral theory}
\label{sect:preliminaries}
{
	In this section, we examine the density of particles in $\mathbf{X}^L$. In Section \ref{sec:spectral} and Section \ref{sec:heat:kernel}, we give precise estimates on the heat kernel $p_t$ associated to \eqref{PDE:A} and compute the \textit{relaxation time} of the system. All the lemmas in these sections hold under \eqref{hpushed} and do not require the additional assumption \eqref{hfp}. Sections \ref{sec:green} and \ref{sec:bulk} are aimed at quantifying the fluctuations in $\mathbf{X}^L$ and are specific to the fully pushed regime.}
\subsection{Preliminaries}
\label{sec:spectral}

Consider the Sturm--Liouville problem \eqref{SLP} together with boundary conditions \eqref{BC}.
Let us first recall some well-known facts about Sturm--Liouville theory following \cite[Section 4.6]{zettl10}: \begin{itemize}
	\item[(i)] A solution of \eqref{SLP} is defined as a function $v:[0,L]\to \mathbb{R}$ such that $v$ and $v'$ are absolutely continuous on $[0,L]$ and satisfies \eqref{SLP} a.e.~on $(0,L)$. In particular, any solution $v$ is continuously differentiable on $[0,L]$. Since $W$ is continuous on $[0,L]$, the solutions are also twice differentiable on $[0,L]$ and (\ref{SLP}) holds for all $x\in(0,L)$.
	\item[(ii)] A complex number $\lambda$ is an eigenvalue of the Sturm--Liouville problem \eqref{SLP} with boundary conditions \eqref{BC}  if Equation \eqref{SLP} has a solution $v$ which is not identically zero on $[0,L]$ and that satisfies \eqref{BC}. This set of eigenvalues will be referred to as the spectrum.
	\item[(iii)] It is known  that this set of eigenvalues is infinite, countable and it has no finite accumulation point. Besides, it is upper bounded and all the eigenvalues are simple and real so that they can be numbered
	      \begin{equation*}
		      \lambda_1>\lambda_2>...> \lambda_k>...
	      \end{equation*}
	      where
	      \begin{equation*}
		      \lambda_k\rightarrow -\infty \quad \textnormal{ as } \quad  k\rightarrow+\infty.
	      \end{equation*}
	      We will denote by $K$ the largest integer such that
	      \begin{equation*}
		      \lambda_k>0.
	      \end{equation*}
	\item[(iv)] As a consequence, the eigenvector $v_k$ associated to $\lambda_k$ is unique up to constant multiplies. Furthermore, the sequence of eigenfunctions can be normalised to be an orthonormal sequence of $\mathrm{L}^2([0,L])$. This orthonormal sequence is complete in $\mathrm{L}^2([0,L])$ so that the fundamental solution of PDE \eqref{PDE:B} can be written as
	      \begin{equation}
		      g_t(x,y)=\sum_{k=1}^\infty e^{\lambda_k t}\frac{v_k(x)v_k(y)}{\|v_k\|^2}.\label{def:qt1}
	      \end{equation}
	\item[(v)] The function $v_1$ does not change sign in $(0,L)$.
		      {More generally,  the eigenfunction $v_k$ has exactly $k-1$ zeros on $(0,L)$.}
	\item[(vi)] The eigenvalues and the eigenvectors of \eqref{SLP} with boundary conditions \eqref{BC} can be characterised through the \textit{Prüfer transformation} (see \cite[Section 4.5]{zettl10}).
	      For all $\lambda\in\mathbb{R}$, consider the Cauchy problems
	      \begin{equation}\label{cp:theta}
		      \dot \theta_\lambda (x)=\cos^2(\theta_\lambda(x))+(W(x)-2\lambda)\sin^2(\theta_\lambda(x)), \quad \theta_\lambda (0)=0,
	      \end{equation}
	      and
	      \begin{equation}
		      \label{cp:rho}
		      \dot \rho_\lambda (x)=\left(\frac{1-W(x)}{2}+\lambda\right)\sin(2\theta_\lambda(x))\rho_\lambda(x), \quad \rho_\lambda(0)=1.
	      \end{equation}
	      Note that \eqref{cp:theta} and \eqref{cp:rho} have a unique solution defined on $[0,+\infty)$ for each $\lambda\in\mathbb{R}$. The eigenvalue $\lambda_k$ is characterised as the unique solution of $\theta_\lambda(L)=k\pi$.

		      {Note that for all $\lambda^*<\lambda$, }
	      \begin{equation}\label{eq:crth}
		      0\leq \theta_{\lambda}(x)\leq\theta_{\lambda*}(x), \quad \forall x\in[0,L].
	      \end{equation}
	      If $u_k$ is the eigenvector associated to $\lambda_k$ such that $u_k'(0)=1$, then
	      \begin{equation}\label{def:uk}
		      u_{k}(x)=\rho_{\lambda_k}(x)\sin(\theta_{\lambda_k}(x)), \quad
		      u_{k}'(x)=\rho_{\lambda_k}(x)\cos(\theta_{\lambda_k}(x)),\quad \forall x\in[0,L].
	      \end{equation}
	\item[(vii)]
	      Denote by $\bar \lambda_1>\bar \lambda_2>...$ the eigenvalues of the Sturm--Liouville problem
	      \begin{equation*}
		      \frac{1}{2} v''(x)+\frac{1}{2} \|W\|_\infty\mathbf{1}_{[0,1]}(x)v(x)=\lambda v(x), \quad v(0)=v(L)=0.
	      \end{equation*}
	      and by $\underline \lambda_1>\underline \lambda_2>...$ the eigenvalues of the Laplacian with homogeneous Dirichlet boundary conditions (at $0$ and $L$). Then, for all $k\in\mathbb{N}$,
	      \begin{equation}\label{eq:lambda:bounds}
		      \underline \lambda_k \leqslant \lambda_k\leqslant \bar \lambda_k.
	      \end{equation}
	      See \cite[Theorem 4.9.1]{zettl10} for a proof of this comparison principle.
	      Recall that
	      \begin{equation}\label{eq:Laplacian}
		      \underline \lambda _k =-\frac{k^2\pi^2}{2L^2},
	      \end{equation}
	      and that the eigenvalues $(\bar \lambda _k)$ have been fully characterised in \cite[Section 2.1]{tourniaire21}.   In particular, we know that there exists $\bar K\geq K$  (that does not depend on $L$) such that, for $L$ large enough, $\bar \lambda_k>0$ for all $k\leq \bar K$ and $\bar \lambda_k<0$ for all $k> \bar K$ . See \eqref{eq:ex} for a characterisation of $(\bar{\lambda}_k)_{k>K}$ when $\|W\|_\infty=10$.
	\item[(viii)] For fixed $k\in\mathbb{N}$, the eigenfunction $\lambda_k$ is an increasing function of $L$ (see e.g.~\cite[Theorem 4.4.4]{zettl10}).
	      Since $\lambda_1$ converges, the $k$-th eigenvalue $\lambda_k$ also converges.
	      Furthermore, by (\ref{eq:lambda:bounds}), this implies that the number of positive eigenvalues $K$ is fixed for $L$ large enough.
	      For $k\leq K$, this limit, that we denote by $\lambda_k^\infty$, is positive and we have
	      \begin{equation*}
		      \linf\geq\lambda_2^\infty\geq \lambda_3^\infty\geq...\geq\lambda_K^\infty.
	      \end{equation*}
	      We will prove below that these inequalities are strict inequalities.

\end{itemize}
Throughout the article, the eigenvector  $v_1$ will be chosen  such that $v_1(1)=1$ and $u_1$ in such a way that $u_1'(0)=1$. Note that for $x\in[1,L]$, we have
\begin{equation}\label{eq:v1L}
	v_1(x)=\sinh\left(\sqrt{2\lambda_1}(L-x)\right)/\sinh\left(\sqrt{2\lambda_1}(L-1)\right).
\end{equation}
On $[0,1]$, the eigenvector $u_1$ is the unique solution of the Cauchy problem
\begin{equation}
	u''(x)= \left(2\lambda_1-W(x)\right)u(x),\quad u(0)=0, \quad u'(0)=1. \label{ODEv}
\end{equation}

{\begin{lemma}[\cite{zettl10}, Theorem 4.5.1] Let $a,b,c\in\mathcal{C}([0,1])$. Let $\theta$ and $\rho$ be the solutions of the Cauchy problems
	\begin{equation*}
		\begin{cases} \dot \theta (x)=a(x)\cos(\theta (x))^2+b(x)\sin(\theta(x))^2, \quad \theta(0)=0, \\
			\dot \rho (x) =c(x)\sin(2\theta(x))\rho(x), \quad \rho(0)=1.
		\end{cases}\end{equation*}
	Let $\Phi$ be the application from $(C([0,1])^3 ,||\cdot||_\infty)$ to $(C(0,1)^2,||\cdot||_\infty)$ which maps $(a,b,c)$
	to the solutions $(\rho,\theta)$ of the Cauchy problems. Then $\Phi$ is continuous.
	\label{lem:continuity}
\end{lemma}}

The next result is a generalisation of Proposition \ref{prop:first}.
\begin{lemma}\label{lem:spectral}
	Assume  that \eqref{hpushed} holds. Then
	\begin{equation}
		\linf-\lambda_1= \frac{\beta e^{2\beta}}{\|v_1^\infty\|^2}e^{-2\beta L}+o\left(e^{-2\beta L}\right) = O\left(\frac{1}{N^{\alpha-1}}\right) \label{cv:lambda}.
	\end{equation}
	As a consequence, for all $0<t\lesssim N$,
	\begin{equation}
		h(t,x)=\left(1+O\left(e^{(\mu-3\beta)L}\right)\right)h(0,x), \quad x\in[0,L]. \label{approx:h}
	\end{equation}
	In addition,
	\begin{equation} \label{exp:d:v1}
		v_1(x)\asymp (1\wedge x\wedge (L-x))e^{-\beta x}, \quad x\in[0,L],
	\end{equation}
	and we have
	$v_1\mathbf{1}_{[0,L]}\to v_1^\infty$ as $L\to\infty$ pointwise and in $\mathrm{L}^2(\mathbb{R}^+)$. The function $v_1$ also converges uniformly on compact sets and in ${\cal C}^1([0,1])$. Moreover, $$v_1^\infty(x)=e^{-\beta (x-1)}, \quad x\geq 1.$$

\end{lemma}
\begin{proof}
	{We first prove the convergence of $u_1$. Recall from \eqref{def:uk} that
		\begin{equation*}
			\forall x\in(0,L), \quad u_1(x) = \rho_{\lambda_1}(x)  \sin(\theta_{\lambda_1}(x)), \quad u'_1(x)=  \rho_{\lambda_1}(x) \cos(\theta_{\lambda_1}(x)).
		\end{equation*}
		By Lemma \ref{lem:continuity}, $u_1$ converges to $u_1^\infty=\rho_{\linf}\sin(\theta_{\linf})$ in $\mathcal{C}^1([0,1])$ as $L\to\infty$.
		The pointwise convergence on $[0,\infty]$ follows from the fact (see (iv) and \eqref{eq:v1L}) that
		\begin{equation}\label{eq:u1L}
			u_1(x) = u_1(1)\frac{\sinh( \sqrt{2\lambda_1}(L-x) )}{\sinh( \sqrt{2\lambda_1}(L-1) )}, \quad \forall x\in[1,L].
		\end{equation}
		The convergence in $\mathrm{L}^2$ follows from the uniform convergence on $[0,1]$ and the explicit form of $u_1$ on $[1,L]$.} The same arguments show that $u_1$ converges uniformly on compact sets.

	It follows from the integrated version of \eqref{ODEv} that $u_1''$ converges to $(u_1^\infty)''$ in $\mathcal{C}([0,1])$ and that $u_1^\infty$ is the unique solution of the Cauchy problem
	\begin{equation}\label{Cauchypb}
		\frac{1}{2}u''(x)+\frac{1}{2}u(x)=\linf u(x), \quad x\in(0,1), \quad u_1(0)=0,\;u_1'(0)=1.
	\end{equation}
	As a consequence, $u_1^\infty(x)\neq 0$ for every $x\in(0,1)$. Indeed, $u_1^\infty\geq 0$ so that if $u_1^\infty(x_0)=0$ for some $x_0\in(0,1)$ then $u_1^\infty$ reaches a minimum at $x_0$ and   $(u_1^\infty)'(x_0)=0$. Since $0$ is a solution of the Cauchy problem with initial condition $(u_1(x_0),u_1'(x_0))=(0,0)$, this implies that $u_1^\infty\equiv 0$ on $[0,L]$, which contradicts the fact that $(u_1^\infty)'(0)=1$. It easily follows that
	$ u_1^\infty(x) \asymp x$ on $[0,1]$. This implies that $ u_1(x) \asymp x$ for $L$ large enough. Indeed, if we consider $x_0^L:=\inf\{x\in[0,1]:u_1'(x)<\frac{1}{2}\}$ (resp.~$x_0$ for $u_1^\infty$) and $w^L:=\min\{u_1(x):x\in [x_0^L,1]\}$ (resp.~$w$ for $u_1^\infty$), we see  (using the uniform convergence of $u_1$ and $u_1'$ on $[0,1]$) that $x_0^L\to x_0$ and $w^L\to w$ as $L\to\infty$.  One can then deduce that for $L$ large enough, $u_1(x)\geq \tfrac{w\wedge 1}{2} x$ for all $x\in[0,1]$. Similarly, one can show that $u_1(x)\lesssim x$.

	The convergence of $v_1$ then follows from the relation $v_1(x) = u_1(x)/u_1(1)$ (see (iv)) and the convergence of $u_1$. Finally,
	Equation (\ref{exp:d:v1})  stems from our bounds on $u_1$ in [0,1] and from (\ref{eq:u1L}).

	Let us now prove the first part of the lemma. According to \cite[Theorem 4.4.4]{zettl10}, $L\mapsto \lambda_1(L)$ is differentiable on $(1,+\infty)$ and we have
	\begin{equation}\label{eq:diff:lambda}
		\lambda_1'(L)=\frac{1}{2}\frac{|u_1'(L)|^2}{\|u_1\|^2}.
	\end{equation}
	Yet,
	\begin{equation*}
		u_1'(L)=u_1(1)v'_1(L)=-u_1(1)\sqrt{2\lambda_1}\left(\sinh{\sqrt{2\lambda_1}(L-1)}\right)^{-1}=O(e^{-\sqrt{2\lambda_1}L}).
	\end{equation*}
	Thus,
	\begin{equation*}
		\lambda_1'(L)=O(e^{-2\sqrt{2\lambda_1}L}).
	\end{equation*}
	Define $\tilde L$ such that $\lambda_1>\frac{\linf}{2}$ for all $L>\tilde L$. Hence, for $L>\tilde L$
	\begin{equation*}
		\linf-\lambda_1\leqslant C\int_L^\infty e^{-2\sqrt{\linf}L'}dL'\leqslant Ce^{-2\sqrt{\linf}L}.
	\end{equation*}
	This implies that $\linf-\lambda_1\ll\frac{1}{L}$ so that
	\begin{equation*}
		\sinh(\sqrt{2\lambda_1}(L-1))\sim \frac{1}{2}e^{\beta(L-1)}, \quad L\to\infty.
	\end{equation*}
	We then see from \eqref{eq:diff:lambda} that
	\begin{equation*}
		\lambda_1'(L)=\frac{\lambda_1}{\|u_1\|^2}u_1(1)\sinh(\sqrt{2\lambda_1}(L-1))^{-2}\sim \frac{2\beta ^2}{\|v_1^\infty\|^2}e^{-2\beta (L-1)}.
	\end{equation*}
	Integrating this on $L$, we get that
	\begin{equation*}
		\linf-\lambda_1\sim \frac{\beta e^{2\beta}}{\|v_1^\infty\|^2}e^{-2\beta L}, \quad L\to\infty.
	\end{equation*}
	Finally, Equation \eqref{approx:h} follows from \eqref{def:hL}.

\end{proof}

\begin{cor}\label{cor:cv:st:multivar} Assume that \eqref{hpushed} holds.
	Let $k\in\mathbb{N}$. Let $g:[0,\infty)^k\to\mathbb{R}$ and $f:[0,\infty)^{2k}\to\mathbb{R}$ be continuous bounded test functions. Let $t>0$. We have
	\begin{equation*}
		\int g(x_1,...,x_k) \prod_{i=1}^k\Pi(dx_i)\to \int g(x_1,...,x_k)\prod_{i=1}^k\Pi^\infty(dx_i), \quad \text{as} \quad L\to\infty,
	\end{equation*}
	and uniformly in $(t_1,...,t_k)\in[0,T]^k$, as $N$ goes to $\infty$, we have
	\begin{equation*}
		\int f(x_1,...,x_k,h(t_1N,x_1),...,h(t_kN,x_k))
		\prod_{i=1}^k\Pi(dx_i) \to \int f(x_1,...,x_k,h^\infty(x_1),...,h^\infty(x_k))
		\prod_{i=1}^k\Pi^\infty(dx_i).
	\end{equation*}
\end{cor}
\begin{proof}
	We restrict ourself to the case $k=1$.
	The case $k>1$ can be proved along the same lines.
	We start with the first limit. We know from Lemma \ref{lem:spectral} that $||v_1||\to ||v_1^\infty||$.
	Hence, it is sufficient to prove that, as $L\to\infty$,
	\begin{equation*}
		\int g(x_1,...,x_k)\prod_{i=1}^k v_1^2(x_i) dx_i \to \int g(x_1,...,x_k)\prod_{i=1}^k (v_1^\infty)^2(x_i) dx_i.
	\end{equation*}
	In addition, we remark from \eqref{eq:v1L} (and a direct calculation) that
	$$
		v_1(x) \leq v_1^\infty(x) \quad \forall x\in [1,L].
	$$
	The result then follows from the dominated convergence theorem and the fact that $v_1$ converges to $v_1^\infty$ uniformly on $[0,1]$. {The second limit follows from similar arguments, using \eqref{approx:h} and the dominated convergence theorem combined with a truncation argument for large values of $h$.}
\end{proof}

\begin{lemma}\label{lem:vp}  Assume that \eqref{hpushed} holds.
	For all $k\in\mathbb{N}$, we have
	\begin{equation*}
		|u_k(x)|\leqslant e^{3\bar c}e^{|\lambda_k|}(1+2|\lambda_k|)x, \quad x\in[0,1],
	\end{equation*}
	with $\bar c=1+\|W\|_\infty$.
\end{lemma}
\begin{proof}
	First, note that for all $\lambda\in\mathbb{R}$, the solution  $\rho_\lambda$  of the Cauchy problem \eqref{cp:rho} can be expressed as
	\begin{equation*}
		\rho_\lambda(x)=\rho_{\lambda}(0)\exp\left(\int_{y=0}^x\left(\frac{1-W(y)}{2}+\lambda\right)\sin(2\theta_\lambda(y))dy\right),\quad  \quad x\in[0,1].
	\end{equation*}
	Therefore,
	\begin{equation}\label{ub:rho}
		0\leqslant \rho_\lambda(x)\leqslant \exp(\bar c+|\lambda|), \quad x\in[0,1].
	\end{equation}
	On the other hand, note that for all $\lambda\in\mathbb{R}$, the unique solution $\theta_\lambda$ of \eqref{cp:theta} is such that
	\begin{equation*}
		\dot \theta_\lambda(x) -\dot \theta_0(x)\leqslant |\dot \theta_\lambda(x) -\dot \theta_0(x)|\leqslant 2\bar c|\theta_\lambda(x)-\theta_0(x)|+2|\lambda|, \quad x\in[0,1],
	\end{equation*}
	where we use that $\cos^2$ and $\sin^2$ are $2$-Lipschitz.
	Moreover, recall from \eqref{eq:crth} that for all $\lambda\leq 0$, $\theta_\lambda\geq \theta_0$  on $[0,L]$. Hence, for all $\lambda<0$, we have
	\begin{equation*}
		\theta_\lambda(x)-\theta_0(x)\leqslant -2\lambda x+\int_{y=0}^x   2\bar c(\theta_\lambda(y)-\theta_0(y))dy, \quad x\in[0,1],
	\end{equation*}
	and Grönwall's lemma yields
	\begin{equation}\label{eq:dist:theta}
		\theta_\lambda(x)\leqslant \theta_0(x)-2\lambda x \exp(2\bar c x), \quad x\in[0,1].
	\end{equation}
	In addition, we see from \eqref{cp:theta} that  $0\leq\theta_0(x)\leq \bar cx$ for all $x\in[0,L]$, so that
	\begin{equation}\label{ub:theta}
		\theta_\lambda(x)\leqslant (\bar c+2|\lambda|e^{2\bar c})x, \quad x\in[0,1],
	\end{equation}
	for all $\lambda<0$. For $\lambda>0$, we use  that $0\leq\theta_\lambda(x)\leq \theta_0(x)$ for all $x\in[0,L]$ (see \eqref{eq:crth}) so that the last inequality still holds for positive $\lambda$. We finally get the result by combining \eqref{def:uk}, \eqref{ub:rho}, \eqref{ub:theta} and the facts that $|\sin(x)|\leqslant x$ for all $x\geq 0$ and that $\bar c\leqslant e^{2\bar c}$.
\end{proof}

\begin{lemma}\label{lem:K}
	There exists a constant $c_{\ref{lem:K}}>0$ such that for sufficiently large $L$, we have $\sqrt{-\lambda_{K+1}}L>c_{\ref{lem:K}}$.
\end{lemma}
\begin{proof} Recall that $\lambda_{K+1}\leq 0\leq \lambda_K$, $\theta_0$ and $\theta_{\lambda_{K+1}}$ are non-decreasing functions, $\theta_{\lambda_{K}}(L) = K \pi$ and that $\theta_{\lambda_{K+1}}(L) = (K+1) \pi$.
	Moreover, we see from \eqref{eq:crth} that
	$$
		\theta_{\lambda_K} \leq \theta_0 \leq \theta_{\lambda_{K+1}} \quad \text{on} \quad [0,L].
	$$
	Hence, $\theta_0(L)$ converges to $(K+\frac{1}{2})\pi$ (i.e., the unique fixed point of the ODE \eqref{cp:theta} belonging to $[K\pi,(K+1)\pi]$). It follows that $\theta_0(1)\in((K-\frac{1}{2})\pi,(K+\frac{1}{2})\pi]$ since $(K-\frac{1}{2})$ is also a fixed point for the dynamics of $\theta_0$.

	We now argue by contradiction and assume that $\lim_{L\to\infty} \sqrt{-2\lambda_{K+1}} L=0$. Let $\vep:=\sqrt{-\lambda_{k+1}}$.
	We see from \eqref{eq:dist:theta} that for $L$ large enough
	\begin{equation}\label{eq:40}
		0 \leq \theta_{\lambda_{K+1}}(1)-\theta_0(1)\leq C|\lambda_{K+1}|\leq C\vep^2.
	\end{equation}
	Let $I_\vep:=\theta_{\lambda_{K+1}}^{-1}((K+\frac{1}{2})\pi,(K+\frac{1}{2})\pi+\vep))\cap[1,L]=[x_0,x_\vep]$. Recalling that $\theta_0(1)\leq (K+\frac{1}{2})\pi$ and using \eqref{eq:40}, we see that $x_\vep>1$. Define
	$
		D(x):=\theta_{\lambda_{K+1}}(x)-\left(K+\frac{1}{2}\right)\pi.
	$ Note that $D(x_0)\in [0, C \vep^2]$. In order see this, we distinguish two cases (1) if $\theta_0(1)=(K+\frac{1}{2})\pi$ then $x_0=1$. Using the fact that $\theta_{\lambda_{K+1}} > \theta_0$, and (\ref{eq:40}) entails the result; (2) if $\theta_0(1)<(K+\frac{1}{2})\pi$ then $x_0>1$ (for $L$ large enough) and $D(x_0)=0$.  Since $D$ is non-decreasing, we have $|D(x)|\leq \eps$ for $x\in[x_0,x_\eps]$. Furthermore,
	\begin{equation*}
		|\dot D(x)|\leqslant \sin^2(D(x))+2|\lambda_{K+1}|\cos^2(D(x))\leqslant C\vep^2.
	\end{equation*}
	Hence, for $L$ large enough
	\begin{equation*}
		\frac{\vep}{2}\leq\vep-C\vep^2\leq D(x_\eps)-D(x_0)\leqslant C\vep^2 (x_\eps-x_0)\leq C\vep^2L,
	\end{equation*}
	which leads to a contradiction since $\vep L \to 0$ as $L\to\infty$.

\end{proof}

\begin{lemma}\label{lem:l2} Assume that \eqref{hpushed} holds.
	There exists $c_{\ref{lem:l2}}>0$ such that for all $k\in\mathbb{N}$,
	\begin{equation*}
		\|u_k\|^2\geqslant c_{\ref{lem:l2}}\left(1\wedge \frac{1}{|\lambda_k|}\right).
	\end{equation*}
\end{lemma}

\begin{proof}
	Define $g_1:[-1,\linf]\to[0,+\infty),\lambda \mapsto \int_{x=0}^1\rho_\lambda(x)^2\sin^2(\theta_\lambda(x))dx.$
	By Lemma \ref{lem:continuity}. the integrand is continuous in $x$ for every $\lambda$. Using (\ref{ub:rho}), it can be bounded by $e^{\bar{c}+1\vee|\linf|}$ for  $\lambda\in[-1,\linf]$.
	By a standard continuity theorem,
	the function $g_1$ is continuous and it attains a positive minimum at some $\lambda_0$.

	Let us now consider $\lambda<-1$ and remark that
	\begin{align}\label{eq:lb:norm}
		\int_{x=0}^1 \rho_{\lambda}(x)^2\sin(\theta_\lambda(x))^2dx & \geqslant \int_{x=0}^{\frac{1}{|\lambda|}} \rho_{\lambda}(x)^2\sin(\theta_\lambda(x))^2dx                            \\
		                                                            & \geqslant\frac{1}{|\lambda|}\int_{x=0}^1\rho_{\lambda}(x/|\lambda|)^2\sin(\theta_\lambda(x/|\lambda|))^2dx.\nonumber
	\end{align}
	For all $x\in[0,1]$, define $\tilde \rho_\lambda (x)=\rho_{\lambda}(x/|\lambda|)$ and $\tilde\theta _\lambda (x)=\theta_{\lambda}(x/|\lambda|)$. These functions correspond to the unique solutions of the Cauchy problems
	\begin{equation*}
		\dot{ \tilde\theta}_\lambda(x)=\frac{1}{|\lambda|}\left(\cos^2(\tilde\theta_\lambda(x))+(W(x/|\lambda|)-2\lambda)\sin^2(\tilde \theta_\lambda(x))\right), \quad \tilde\theta_\lambda (0)=0,
	\end{equation*}
	and
	\begin{equation*}
		\dot{ \tilde\rho}_\lambda(x)=\frac{1}{|\lambda|}\left(\left(\frac{1-W(x/|\lambda|)}{2}+\lambda\right)\sin(2\tilde\theta_\lambda(x))\tilde\rho_\lambda(x)
		\right), \quad \tilde\rho_\lambda(0)=1,
	\end{equation*}
	on $[0,1]$.
	As before, $g_2:[-1,0]\to[0,+\infty), \lambda\mapsto \int_{x=0}^{1}\tilde{\rho}_{1/\lambda}(x)^2\sin^2(\tilde \theta_{1/\lambda}(x))dx$
				attains its minimum $m$ at some positive value. This combined with \eqref{eq:lb:norm} shows that $g_3:(-\infty,-1]\to[0,+\infty),\lambda\mapsto \int_{x=0}^1 \rho_{\lambda}(x)^2\sin(\theta_\lambda(x))^2dx$ satisfies $g_3(\lambda)\geq \frac{m}{|\lambda|}$. We conclude the proof of the lemma by recalling from \eqref{def:uk} that $$\forall k\geq 1;\quad \|u_k\|^2=\int_0^L\rho_{\lambda_k}(x)^2\sin(\theta_{\lambda_k}(x))^2dx\geq\int_0^1\rho_{\lambda_k}(x)^2\sin(\theta_{\lambda_k}(x))^2dx.$$

\end{proof}

\begin{proposition}\label{prop:ratio-u_k} Assume that \eqref{hpushed} holds.
	There exists a constant $c_{\ref{prop:ratio-u_k}}$ such that for $L$ large enough,  for all $x\in[0,L]$,
	\begin{eqnarray*}
		\forall k\ \leq K, & \frac{1}{\|u_k\|} \frac{|u_k(x)|}{u_{1}(x)} & \leqslant c_{\ref{prop:ratio-u_k}} \ \ e^{\beta x},\\
		\forall k> K, & \ \frac{1}{\|u_k\|}\frac{|u_k(x)|}{u_1(x)} & \leqslant c_{\ref{prop:ratio-u_k}} e^{\beta x} \  e^{|\lambda_k|}(1+2|\lambda_k|)^2.
	\end{eqnarray*}
\end{proposition}
\begin{proof}
	{ Let us first consider the case $k\leq K$.
		In this case, one can show that $u_k$ converges to a limiting function $u_k^\infty$ in $\mathrm{L}^2$ and that
		\begin{equation*}
			|u_k(x)| \asymp  (x\wedge 1\wedge (L-x)) e^{-\sqrt{2 \lambda^\infty_k} x}, \ \ \ \mbox{on $[0,L]$,}
		\end{equation*}
		by applying similar arguments to those used to prove \eqref{exp:d:v1}.
		This shows the result for $k\leq K$.}

	We now turn to the case $k>K$. First, we use Lemma \ref{lem:vp} along with Lemma \ref{lem:l2} to bound $u_k/\|u_k\|$ on $[0,1]$: there exists $c_2>0$ such that
	\begin{equation}\label{U1}
		\frac{|u_k(x)|}{\|u_k\|}\leqslant c_2 e^{|\lambda_k|}(1+2|\lambda_k|)(1\vee |\lambda_k|)^{1/2}x\leqslant c_2 e^{|\lambda_k|}(1+2|\lambda_k|)^{2}x, \quad \forall x\in[0,1], \; k>K.
	\end{equation}
	On $[1,L]$, we  have $u_k(x)=d_k\sin(\sqrt{-2\lambda_k}(L-x))$ for some $d_k\in\mathbb{R}$ (see \eqref{ODEv}). Hence,
	\begin{equation*}
		\|u_k\|^2\geqslant \int_{x=1}^Lu_k(x)^2dx =d_k^2(L-1)\left(1-\frac{\sin(2\sqrt{-2\lambda_k}(L-1))}{2\sqrt{-2\lambda_k}(L-1)}\right).
	\end{equation*}
	Recall from Lemma \ref{lem:K} that there exists $c_3>0$ such that
	\begin{equation*}
		\left(1-\frac{\sin(2\sqrt{-2\lambda_k}(L-1))}{2\sqrt{-2\lambda_k}(L-1)}\right)>c_3, \quad \forall k>K.
	\end{equation*}
	Using that $|\sin(x)|\leqslant 1\wedge x$ for all $x\geq 0$, we get that for sufficiently large $L$,
	\begin{equation}\label{U2}
		\frac{|u_k(x)|}{\|u_k\|}\leqslant \frac{1}{\sqrt{c_3(L-1)}}\leqslant 1,\quad \forall x\in[1,L-1],
	\end{equation}
	and for all $x\in[L-1,L]$,
	\begin{equation}\label{U3}
		\frac{|u_k(x)|}{\|u_k\|}\leqslant \frac{1}{\sqrt{c_3(L-1)}}\sqrt{-2\lambda_k}(L-x)\leqslant \sqrt{-2\lambda_k}(L-x)\leqslant (1+2|\lambda_k|)^2(L-x).
	\end{equation}
	Equations \eqref{U1}, \eqref{U2} and \eqref{U3} and Lemma \ref{lem:spectral}
	yield the result for ${k> K}$.
\end{proof}

\begin{proposition}[Spectral gap] Assume that \eqref{hpushed} holds.
	\label{lem:spectral:gap}
	Then,  $\linf>\lambda_2^\infty$.
\end{proposition}
\begin{proof}

	Assume, for the sake of contradiction, that $\linf=\lambda_2^\infty$.
	Then, it follows from Lemma \ref{lem:continuity}  that $u_2$ converges  to $u_1^\infty$ (see \eqref{Cauchypb}) in $\mathcal{C}^1([0,1])$. Recall from Lemma \ref{lem:spectral} that $u_1^\infty(x)>0$ for all $x\in(0,1]$. On the other hand, we know from (v) that $u_2$ has a single $0$, located in $(0,1]$ (it can not be located in $[1,L)$ since $u_2(x)$ is a multiple of $\sinh(\sqrt{2\lambda_2}(L-x))$ in this interval). We denote by $x_L$ the position of this $0$ and remark that $u_2(2x_L)<0$. This stems from (v), the fact that $u_2\geq 0$ on $[0,x_L]$  and that $u_2'(x_L)\neq 0$ (here we use the same argument as in the proof of Lemma \ref{lem:spectral} to show that $u_1^\infty>0$). Hence, since $u_2$ converges to a positive function, $x_L\to 0$ as $L\to\infty$, and one can show that  $u_2(2x_L)/(2x_L)<0$ converges to $(u_1^\infty)'(0)>0$ (e.g. using the mean value theorem). This leads to a contradiction.

\end{proof}

\subsection{Heat kernel estimate}\label{sec:heat:kernel}
\begin{proposition}\label{lem:hk}
	Assume that \eqref{hpushed} holds. There exists a constant $c_{\ref{lem:hk}}>0$ such that  for $L$ large enough and $t>c_{\ref{lem:hk}}L$, we have
	\begin{equation*}
		\left|p_t(x,y)-h(0,x)\tilde h(t,y)\right|\leqslant e^{-\beta L}h(0,x)\tilde h(t,y), \quad x,y\in[0,L].
	\end{equation*}
	Recalling from Lemma \ref{lem:many-to-one0} and Proposition \ref{prop:invariant-1-spine} that $q_t(x,y)=\tfrac{h(t,y)}{h(0,x)}p_t(x,y)$ and $\Pi(x)=h(t,x)\tilde h(t,x)$, this implies that  for $L$ large enough and $t>c_{\ref{lem:hk}}L$,\begin{equation*}
		\left|q_t(x,y)-\Pi(y) \right|\leqslant e^{-\beta L}\Pi(y), \quad x,y\in[0,L].
	\end{equation*}
\end{proposition}
\begin{proof}

	By Proposition \ref{prop:ratio-u_k}, for $L$ large enough and $t>2$,
	\begin{equation*}
		\sum_{k=2}^{K} e^{\lambda_k t}\frac{1}{\|u_k\|^2}\frac{u_k(x)u_k(y)}{v_1(x)v_1(y)}\leqslant  K (c_{\ref{prop:ratio-u_k}})^2e^{2 \beta L+\lambda_2^\infty t},
	\end{equation*}
	and
	\begin{align}
		\label{eq:ineqw}
		\sum_{k=K+1}^{\infty} e^{\lambda_k t}\frac{1}{\|u_k\|^2}\frac{u_k(x)u_k(y)}{v_1(x)v_1(y)}\leqslant (c_{\ref{prop:ratio-u_k}})^2e^{2\beta L}\sum_{k=K+1}^{\infty}(1+2|\lambda_k|)^4e^{\lambda_k(t-2)}.\end{align}
	In order to evaluate the latter sum, we rely on the comparison principle \eqref{eq:lambda:bounds}. First, assume that $\|W\|_\infty\neq(k-\frac{1}{2})^2\pi^2$, $\forall k\in\mathbb{N}$. Define
	\begin{equation*}
		A_i=\frac{1}{2}\left(\left(\bar K +\frac{1}{2}+i\right)^2-\|W\|_\infty\right), \ \ \ N_i=\left\lfloor\frac{(L-1)}{\pi}\sqrt{2A_i}+\frac{1}{2}\right\rfloor+i,
	\end{equation*}
	for all $i\geq 0$ and $A_{-1}=N_{-1}=0$. One can show by an explicit calculation that for all $j\in\mathbb{N}$ such that $N_{i-1}<j\leqslant N_i$, we have
	\begin{equation}\label{eq:distrib:lambar}
		\lambda_{\bar K+j}\in(-A_i,-A_{i-1}).
	\end{equation}
	We refer to \cite[Lemma 2.1]{tourniaire21} for further details on this calculation. If $\|W\|_\infty=(k-\frac{1}{2})^2\pi^2$ for some $k\in\mathbb{N}$, one can replace $\|W\|_\infty$ by $\|W\|_\infty+\vep$ for some small $\vep>0$ and get similar estimates.
	We now use  \eqref{eq:lambda:bounds}, \eqref{eq:Laplacian} and \eqref{eq:distrib:lambar} to bound the sum on the RHS of (\ref{eq:ineqw}). We get that for $L$ large enough,
	\begin{align*}
		 & \sum_{k=K+1}^{\infty}(1+2|\lambda_k|)^4e^{\lambda_k(t-2)}                                                                                                                              \\
		 & \leqslant \sum_{k=K+1}^{\bar K}\left(1+\frac{k^2\pi^2}{L^2}\right)^4+\sum_{j=1}^{\infty}\left(1+\frac{(\bar K+j)^2\pi^2}{L^2}\right)^4e^{\bar \lambda_{\bar K+j } (t-2)}               \\
		 & \leqslant \bar K \left(1+\frac{\bar K^2\pi^2}{L^2}\right)^4+\sum_{i=0}^\infty \sum_{j=N_{i-1}+1}^{N_i}\left(1+\frac{(\bar K+j)^2\pi^2}{L^2}\right)^4e^{\bar \lambda_{\bar K+j } (t-2)} \\
		 & \leqslant 2\bar K+\sum_{i=0}^{\infty}(N_i-N_{i-1})\left(1+\frac{(\bar K+N_i)^2\pi^2}{L^2}\right)^4e^{-A_{i-1}  (t-2)}.
	\end{align*}
	One can easily show that, for $i\in\mathbb{N}$, $\frac{i^2}{2} \lesssim A_i  $, $N_i\lesssim iL $ and  $N_i-N_{i-1}\lesssim iL$. Therefore, there exists $c_0,c_1>0$ such that
	\begin{equation*}
		\sum_{k=K+1}^{\infty}(1+2|\lambda_k|)^4e^{\lambda_k(t-2)}\leqslant c_0\left(1+L\sum_{i=0}^\infty i(1+i^2)^{4}e^{-c_1(i-1)^2(t-2)}\right).
	\end{equation*}
	Note that $\sum_{i=0}^\infty i(1+i^2)^4e^{-c_1(i-1)^2(t-2)}<\sum_{i=0}^\infty i(1+i^2)^4e^{-c_1(i-1)^2}<\infty$ for $t>3$. Thus, there exists
	$c_2>0$ such that, for all $L$ large enough, for all $t>3$,
	\begin{equation*}
		\sum_{k=K+1}^{\infty}(1+2|\lambda_k|)^4e^{\lambda_k(t-2)}\leqslant c_2 L.
	\end{equation*}
	Putting all of this together, we see that, for all $L$ large enough, for all $t>3$,
	\begin{equation*}
		\sum_{k=2}^\infty e^{\lambda_k t}\frac{|v_k(x)v_k(y)|}{\|v_k\|^2}\leqslant (c_{\ref{prop:ratio-u_k}})^2 v_1(x)v_1(y)e^{2\beta L}\left(K e^{\lambda_2^\infty t}+c_2L\right).
	\end{equation*}
	Recalling that $p_t$ can be written as
	\begin{equation*}
		p_t(x,y)=e^{\mu(x-y)}e^{-\linf t}\sum_{k=1}^\infty e^{\lambda_k}\frac{|v_k(x)v_k(y)|}{\|v_k\|^2},
	\end{equation*}
	(see \eqref{def:q} and \eqref{def:qt1}), that  $||v_1||\to ||v_1^\infty||$ (see  Lemma \ref{lem:spectral})  and that $\linf-\lambda_2^\infty>0$ under \eqref{hpushed} (see Proposition \ref{lem:spectral:gap}), we obtain
	\begin{align*}
		 & \left|p_t(x,y)-\frac{1}{\|v_1\|^2}e^{(\lambda_1-\linf)t}e^{\mu(x-y)}v_1(x)v_1(y)\right|               \\
		 & \leqslant CLe^{2\beta L-(\lambda_1-\lambda_2^\infty)t}e^{(\lambda_1-\linf)t}e^{\mu(x-y)}v_1(x)v_1(y),
		\\&\leqslant CLe^{2\beta L-\tfrac{\linf-\lambda_2^\infty}{2}t}e^{(\lambda_1-\linf)t}e^{\mu(x-y)}v_1(x)v_1(y)
	\end{align*}
	for all $t>3$ and $L$ large enough. This together with \eqref{def:hL} concludes the proof of  Proposition~\ref{lem:hk}.
\end{proof}

\subsection{Green's function}\label{sec:green}
The Green's function  can be expressed thanks to the \textit{fundamental solutions} of the ODE
\begin{equation}\label{ODE:G}
	\frac{1}{2}u''+\frac{1}{2}W(x)u=\lambda u.
\end{equation}
Let $\varphi_\lambda$ (resp.~$\psi_\lambda$) be a solution of \eqref{ODE:G} such that $\varphi_\lambda(0)=0$ (resp.~$\psi_\lambda(L)=0$).  Define
the Wronskian as
\begin{equation*}
	\omega_\lambda=\psi_\lambda(1)\varphi_\lambda'(1)-\psi_\lambda'(1)\varphi_\lambda(1).
\end{equation*}
Note that $\varphi_\lambda$ and $\psi_\lambda$ are unique up to constant multiplies. Without loss of generality, we can set $\psi_\lambda$ and $\varphi_\lambda$ so that
\begin{equation}\label{eq:norm}
	\psi_\lambda(x)=\frac{\sinh(\sqrt{2\lambda}(L-x))}{\sinh(\sqrt{2\lambda_1}(L-1))}, \qquad \text{and} \qquad \varphi_\lambda'(0)=v_1'(0).
\end{equation}

\begin{proposition}\label{prop:Green}
	Let $\xi>0$ and $\lambda(\xi)\equiv \lambda :=\linf+\xi$. Define
	$$
		G_{\xi}(x,y) \ = \ \int_0^\infty e^{-\xi t} p_{t}(x,y) dt.
	$$
	Then
	\begin{equation}
		G_\xi(x,y)=\begin{cases}
			(\omega_{\lambda(\xi)})^{-1}e^{\mu(x-y)}\psi_{\lambda(\xi)}(x)\varphi_{\lambda(\xi)}(y), & 0\leq y\leq x \leq L, \\
			(\omega_{\lambda(\xi)})^{-1}e^{\mu(x-y)}\varphi_{\lambda(\xi)(x)}\psi_{\lambda(\xi)}(y), & 0\leq x\leq y\leq L.
		\end{cases}\label{def:green}
	\end{equation}
\end{proposition}
\begin{proof}

	Recall the definition of $g$ from \eqref{def:q} and define
	\begin{equation*}
		H_\lambda(x,y):=\int_0^\infty e^{-\lambda t} g_t(x,y)\,dt.
	\end{equation*}
	By \cite[p.19]{Borodin12}, $H_\lambda$ is given by
	\begin{equation*}
		H_\lambda(x,y)=\begin{cases}
			(\omega_\lambda)^{-1}\psi_\lambda(x)\varphi_\lambda(y), & 0\leq y\leq x \leq L, \\
			(\omega_\lambda)^{-1}\varphi_\lambda(x)\psi_\lambda(y), & 0\leq x\leq y\leq L.
		\end{cases}
	\end{equation*}
	On the other hand, we see from \eqref{def:q} that
	\begin{equation*}
		H_\lambda(x,y)=\int_0^\infty e^{-\lambda t} g_t(x,y)\,dt=e^{\mu(y-x)}\int_0^\infty e^{(\linf-\lambda) t}p_t(x,y)dt.
	\end{equation*}
	This yields the result.
\end{proof}

\begin{lemma}
	Assume that \eqref{hfp} holds and that $\frac{1}{N}\lesssim \xi \ll \frac{1}{L}$. Let $\lambda(\xi)$ be as in Proposition \ref{prop:Green}. Then,
	\begin{equation}
		\varphi_{\lambda(\xi)}(x)=v_1(x)+O(\xi)(1\wedge x)\left(e^{\beta x}+Le^{-\beta x}\right) \label{est:phi1},
	\end{equation}
	and
	\begin{equation}
		\psi_{\lambda(\xi)}(x)=v_1(x)+O(\xi L)(1\wedge(L-x))e^{-\beta x}.
	\end{equation}
	In addition, we also have
	\begin{align}
		\vl(x) & =(1\wedge x)\left(O(1)e^{-\beta x}+O(\xi)e^{\beta x}\right)\label{est:phi2}, \\
		\pl(x) & =O(1)(1\wedge(L-x)){e^{-\beta x}}.
	\end{align}
	\label{lem:phipsi}
\end{lemma}

\begin{proof}
	We know from Lemma \ref{lem:spectral} that  $e^{-2\beta L}\ll N^{-1}\lesssim \xi$ under \eqref{hfp} so that $\lambda(\xi)>\la_1$ for sufficiently large $L$.

	We have
	\begin{equation}\label{eq:phiv}
		\vl(x)-v_1(x)=\int_{z=0}^x\int_{y=0}^z \left(\vl''(y)-v_1''(y)\right)dy\,dz, \quad x\in[0,1].
	\end{equation}
	We know (see \eqref{SLP} and \eqref{ODE:G}) that
	\begin{equation}
		\vl''(y)-v_1''(y)= (2\lambda(\xi)-W(y))(\vl(y)-v_1(y))+2(\la-\la_1)v_1(y).\label{eq:phiv2}
	\end{equation} Besides, we see from Lemma \ref{lem:spectral} that $2(\lambda(\xi)-\lambda_1)v_1=O(\xi)$.
	Hence, we obtain that for all $x\in[0,1]$,
	\begin{align*}
		|\vl(x)-v_1(x)| & \leqslant \int_{z=0}^1\int_{y=0}^x |2\lambda(\xi)-W(y)||\vl(y)-v_1(y)| dy\; dz +C\xi \\
		                & \leqslant \int_{y=0}^x (2\linf+\|W\|_\infty)|\vl(y)-v_1(y)| dy + C\xi.
	\end{align*}
	Grönwall's lemma then yields
	\begin{equation*}
		\vl(x)-v_1(x)=O(\xi), \quad x\in[0,1].
	\end{equation*}
	Putting this together with  \eqref{eq:phiv} and \eqref{eq:phiv2}, we get that
	\begin{equation}\label{eq:int1}
		\vl(x)-v_1(x)=O(\xi)x=O(\xi)(1\wedge x)e^{\beta x},\quad x\in[0,1].
	\end{equation}
	Using a similar argument, one can easily show that
	\begin{equation}
		\vl'(1)-v_1'(1)=O(\xi). \label{eq:phiv3}
	\end{equation}
	Then, we see from \eqref{ODE:G} that, on $[1,L]$, $\vl$ can be written as
	\begin{equation*}
		\varphi_{\lambda(\xi)}(x)=\vl(1)\frac{\sinh(\sqrt{2\lambda(\xi)}(L-x))}{\sinh(\sqrt{2\lambda(\xi)}(L-1))}+ Ae^{\sqrt{2\la(\xi)}(x-1)}+Be^{-\sqrt{2\lambda(\xi)}(x-1)},
	\end{equation*} for some $A,B\in\mathbb{R}$. Applying this equality to $x=1$, we see that $A=-B$.
	In addition, we know that $\vl(1)=v_1(1)+O(\xi)=1+O(\xi)$  and one can show by an explicit computation (e.g.~using the mean value theorem on $[1,L-1]$ and expanding $\sinh(\sqrt{2\lambda(\xi)}(L-x))=\sinh((\sqrt{2\lambda_1}+O(\xi))(L-x))$ on $[L-1,L]$)  that
	\begin{equation}\label{est:xiL}
		\frac{\sinh(\sqrt{2\lambda(\xi)}(L-x))}{\sinh(\sqrt{2\la(\xi)}(L-1))}=(1+O(\xi L))v_1(x), \quad x\in[1,L].
	\end{equation}
	Besides, a direct calculation shows that
	\begin{equation*}
		\vl'(1)=\sqrt{2\la}\left(2A-\vl(1)\frac{\cosh(\sqrt{2\lambda(\xi)}(L-1))}{\sinh(\sqrt{2\lambda(\xi)}(L-1))}\right).
	\end{equation*}
	Recalling from \eqref{eq:phiv3} that $\vl'(1)=v'_1(1)+O(\xi)$ with $v_1'(1)=-\sqrt{2\la_1}/\tanh(\sqrt{2\lambda_1}(L-1))$ and remarking that $\sqrt{2\lambda(\xi)} =\sqrt{2\linf}(1+O(\xi))$, $\sqrt{2\lambda_1}=\sqrt{2\linf}(1+O(e^{-2\beta L})),$ $\tanh(\sqrt{2\lambda_1}(L-1))^{-1}=1+O(e^{-2\beta L})$, $\tanh(\sqrt{2\lambda(\xi)}(L-1))^{-1}=1+O(e^{-2\beta L})$ and that $e^{-2\beta L}=O(\xi)$ under \eqref{hfp}, we get that
	\begin{equation*}
		A=O(\xi).
	\end{equation*}
	Therefore,
	\begin{align*}
		\vl(x) & =(1+O(\xi L))v_1(x)+O(\xi)\sinh(\sqrt{2\lambda(\xi)} (x-1))                           \\
		       & =v_1(x)+O(\xi L)(1\wedge x)e^{-\beta x}+O(\xi)(1\wedge x)e^{\beta x},\quad x\in[1,L].
	\end{align*}
	Putting this together with \eqref{eq:int1} yields \eqref{est:phi1}. Equation \eqref{est:phi2} can then be deduced from the first part of the above: one can easily show that $\sinh(\sqrt{2\lambda(\xi)}(x-1))=O(1)(1\wedge (x-1))e^{\beta x}$ for all $x\in[1,L]$. We then use \eqref{eq:int1} to conclude.

	We now move to the estimate on $\pl$. We recall from \eqref{eq:norm} that
	\begin{equation*}
		\pl(x)=\frac{\sinh(\sqrt{2\lambda(\xi)}(L-x))}{\sinh(\sqrt{2\lambda_1}(L-1))}, \quad x\in[1,L].
	\end{equation*}
	As for \eqref{est:xiL}, one can show that
	\begin{equation}
		\label{eq:int0}
		\pl(x)=(1+O(\xi L))v_1(x), \quad x\in[1,L].
	\end{equation}
	The same argument also yields that $\pl'(1)=v_1'(1)+O(\xi L)$.

	We now prove that this bound also holds on $[0,1]$.
	On [0,1], the function $\pl$ also satisfies \eqref{eq:phiv2}. Hence, we get that for all $x\in[0,1]$,
	\begin{align*}
		\pl(x)-v_1(x) & ={\pl(1)-v_1(1)}+\int_{z=1}^x \pl'(z)-v_1'(z)dz                                                                                                    \\
		              & =\underbrace{\pl(1)-v_1(1)}_{=O(\xi L)}+\int_{z=1}^x \left(\underbrace{\pl'(1)-v_1'(1)}_{=O(\xi L)}+\int_{y=1}^z (\pl''(y)-v_1''(y)) \;dy\right)dz \\
		              & =O(\xi L)+\int_{z=1}^x\int_{y=1}^z (2\lambda(\xi) - W(y))(\pl(y)-v_1(y)) dy\,dz.
	\end{align*}
	Applying Grönwall's inequality in the same way as for $\vl$ we get that
	\begin{equation*}
		\pl(x)=v_1(x)+O(\xi L), \quad x\in[0,1].
	\end{equation*}
	This concludes the proof of the lemma.

\end{proof}

\begin{lemma}\label{lem:wronsk}
	Assume that \eqref{hfp} holds and that $\frac{1}{N}\lesssim \xi \ll \frac{1}{L}$. Let $\lambda(\xi)$ be as in Proposition \ref{prop:Green}.
	Then,
	\begin{equation*}
		\omega_{\lambda(\xi)}=\xi\left(\int_{y=0}^Lv_1(y)^2dy+O(\xi L)\right).
	\end{equation*}
\end{lemma}

\begin{proof}
	Recall that $q_t(x,\cdot)=\tfrac{h(t,\cdot)}{h(0,x)}p_t(x,\cdot)$ is a probability density function and that $h(t,y)=e^{(\linf-\lambda_1)t}h(0,y)$. Thus, by definition of $G$ (see Proposition \ref{prop:Green}),
	\begin{equation}
		\label{green:spine}
		\int_{0}^\infty e^{-(\xi+(\linf-\lambda_1)) t}q_t(x,y)dt=\frac{h(0,y)}{h(0,x)}\int_0^\infty e^{-\xi t}p_t(x,y)dt=\frac{h(0,y)}{h(0,x)}G_{\xi}(x,y),
	\end{equation}
	so that for $\tilde \xi=\xi+(\linf-\lambda_1)$, we have
	\begin{align*}
		\frac{1}{\tilde \xi} & =\int_{0}^\infty e^{-\tilde{\xi} t}\overbrace{\int_0^L q_t(x,y)dy}^{=1}\ dt=  \frac{1}{h(0,x)}\int_0^Lh(0,y)G_{ \xi}(x,y)dy.
	\end{align*}
	We see then from \eqref{def:green} that
	\begin{equation*}
		\int_{y=0}^L h(0,y)G_{\xi} (x,y)dy
		=\frac{e^{\mu x}}{\omega_{\lambda( \xi)}}\left(\psi_{\lambda( \xi)}(x)\int_{y=0}^x\varphi_{\lambda( \xi)}(y)v_1(y)dy+\varphi_{\lambda( \xi)}(x)\int_{y=x}^L\psi_{\lambda( \xi)}(y)v_1(y)dy\right).
	\end{equation*}
	Using Lemma \ref{lem:spectral} and Lemma \ref{lem:phipsi}, and noting that for $0<x_1<x_2<L$, we have $$\int_{x_1}^{x_2}e^{-2\beta y}dy=O(1) (1\wedge(x_2-x_1)),$$ we get that
	\begin{align*}
		\psi_{\lambda( \xi)}(x)\int_{y=0}^x\varphi_{\lambda( \xi)}(y)v_1(y)dy & = \psi_{\lambda( \xi)}(x)\int_{y=0}^x\left(v_1(y)^2+O(\xi)(1+Le^{-2\beta y})\right)dy \\
		                                                                      & =\psi_{\lambda( \xi)}(x)\left(\int_{y=0}^x v_1(y)^2dy+ O(\xi L)(1\wedge x)\right)     \\
		                                                                      & =v_1(x)\int_{y=0}^x v_1(y)^2dy+ O(\xi L)(1\wedge x \wedge (L-x))e^{-\beta x}          \\
		                                                                      & = v_1(x)\left(\int_{y=0}^x v_1(y)^2dy+ O(\xi L)\right),
	\end{align*}
	and
	\begin{align*}
		\varphi_{\lambda( \xi)}(x)\int_{y=x}^L\psi_{\lambda( \xi)}(y)v_1(y)dy & =\varphi_{\lambda( \xi)}(x)\int_{y=x}^L\left(v_1(y)^2+O(\xi L)e^{-2\beta y}\right)dy                \\
		                                                                      & = \varphi_{\lambda( \xi)}(x)\left(\int_{y=x}^Lv_1(y)^2dy+O(\xi L)(1\wedge(L-x)e^{-2\beta x})\right) \\
		                                                                      & =v_1(x)\int_{y=x}^Lv_1(y)^2dy+O(\xi L)(x\wedge1\wedge(L-x))e^{-\beta x}                             \\
		                                                                      & =v_1(x)\left(\int_{y=x}^Lv_1(y)^2dy+O(\xi L)\right).
	\end{align*}
	Putting all of this together, we get that
	\begin{equation*}
		\omega_\lambda=\tilde{\xi}\left(\int_0^Lv_1(y)^2dy+O(\xi L)\right).
	\end{equation*}
	We conclude the proof by remarking that under \eqref{hfp}, $\tilde \xi =O(\xi)$.
\end{proof}

\begin{rem}\label{rem:green}
	For all $t>0$,
	\begin{equation*}
		\int_0^t p_s(x,y)ds\leqslant eG_{\frac{1}{t}}(x,y).
	\end{equation*}
	This follows from the fact that $\mathbf{1}_{s\in[0,t]}\leqslant e^{\frac{t-s}{t}}$ for all $0<s<t$.
\end{rem}

\subsection{The number of particles escaping the bulk}\label{sec:bulk}
{ In this section, $\gamma$ denotes a real number in $(0,1]$. We are interested in the expected number of particles reaching the level $\gamma L$ during the time interval $[0,tN]$. The following lemma allows to prove that, for a suitable choice of $\gamma$, this number  is exponentially small (in $L$).

	For $t>0$, we denote by  $R^\gamma([0,tN])$ the number of particles whose ancestors did not reach the level $\gamma L$ before time $tN$. We also define $p_t^\gamma (x,y)$, $v_1^\gamma$, $\lambda_1^\gamma$ and $G_{\xi}^\gamma(x,y)$ (as well as $\vl^\gamma$, $\pl^\gamma$ and $\omega_{\lambda(\xi)}^\gamma$) in the same way as $p_t(x,y)$, $v_1$, $\lambda_1$ and $G_{\xi}(x,y)$ but with $\gamma L$ instead of $L$.

	\begin{lemma}\label{lem:R} Assume that \eqref{hfp} holds.
		Let $T>0$ and $\gamma\in(0,1]$. For $N$ large enough, for all $x\in[0,\gamma L]$ and $t\in[0,T]$,
		\begin{equation*}
			\mathbb{E}_x\left[R^\gamma([0,tN])\right]=
			O(1)(1\wedge x)\left[e^{(\mu-\beta)x}e^{\left[(\mu-\beta)-\gamma(\mu+\beta)\right]L}+e^{(\mu+\beta)(x-\gamma L)}\right],\end{equation*}

		where the $O(\cdot)$ may depend on $T$ but not on $x$ nor on $t$.
	\end{lemma}
	\begin{proof}
		It is known (see e.g.~\cite[Lemma 5.7]{maillard20}) that
		\begin{equation*}
			\mathbb{E}_{x}\left[R^{\gamma }([0,tN])\right]=-\frac{1}{2}\int_{s=0}^{tN}\frac{\partial}{\partial y}p_s^{\gamma}(x,y)|_{y=\gamma L}ds.
		\end{equation*}
		In words, this means \cite{maillard20} that the expected number of particles ``killed'' at $\gamma L$ between times $0$ and $tN$ is equal to the integral of the heat flow out of the boundary $\gamma L$. Given the boundary condition in \eqref{PDE:A}, the flow out of $\gamma L$ is exactly $-\frac{1}{2}\frac{\partial}{\partial y}p_s^{\gamma}(x,y)|_{y=\gamma L}$ .

		Hence, $\mathbb{E}_x\left[R^\gamma([0,tN])\right]$ is an increasing function of $t$ (as expected) and a bound similar to Remark \ref{rem:green} yields that
		\begin{equation*}
			\mathbb{E}_{x}\left[R^{\gamma }([0,tN])\right]\leqslant \mathbb{E}_{x}\left[R^{\gamma}([0,TN])\right]\leqslant -\frac{1}{2}e\frac{\partial}{\partial y}G^\gamma_{(TN)^{-1}}(x,y)_{|y=\gamma L}.
		\end{equation*}
		Let $\xi=\frac{1}{TN}$ and $\lambda(\xi)=\linf +\xi$.  We know from \eqref{def:green} that
		\begin{equation*}
			\frac{\partial}{\partial y}G^{\gamma}_{(TN)^{-1}}(x,y)_{|y=\gamma L^{}}=(\omega_{\lambda(\xi)}^\gamma)^{-1}e^{\mu(x-\gamma L)}\vl^\gamma(x)(\pl^\gamma)'(\gamma L).
		\end{equation*}
		We then see from  \eqref{eq:norm} and Lemma \ref{lem:spectral} that
		\begin{equation*}
			(\pl^\gamma)'(\gamma L)=\sqrt{2\lambda_1^\gamma}\sinh\left(\sqrt{2\lambda_1^\gamma}(\gamma L -1)\right)^{-1}=O(e^{-\beta \gamma L}).
		\end{equation*}
		By Lemmas \ref{lem:phipsi}  and \ref{lem:wronsk}, we have
		\begin{equation*}
			(\omega_{\lambda(\xi)}^\gamma)^{-1}\vl^{{\gamma}}(x)\\=(\omega_{\lambda(\xi)}^\gamma)^{-1}(1\wedge x)\left(O(1)e^{-\beta x}+O(\xi)e^{\beta x}\right)=(1\wedge x)\left(O(N)e^{-\beta x}+O(1)e^{\beta x}\right).
		\end{equation*}
		Putting all these estimates together, we see that
		\begin{equation*}
			-\frac{1}{2}\frac{\partial}{\partial y} G^{\gamma}_{(TN)^{-1}}(x,y)_{|y=\gamma L}=O(1)(1\wedge x)e^{-(\mu+\beta)\gamma L}\left[e^{(\mu-\beta) L}e^{(\mu-\beta)x}+e^{(\mu+\beta)x}\right],
		\end{equation*}
		which concludes the proof of the lemma.
	\end{proof}
	\begin{cor}\label{cor:R} Assume that \eqref{hfp} holds.
		Let $T>0$ and $\gamma\in\left(0,\frac{\mu-\beta}{2\beta}\right]$. For $N$ large enough, for all $x\in[0,\gamma L]$ and $t\in[0,T]$,
		\begin{equation*}
			\mathbb{E}_x\left[R^\gamma([0,tN])\right]=
			O(1)(1\wedge x)e^{(\mu-\beta)x}e^{\left[(\mu-\beta)-\gamma(\mu+\beta)\right]L}.\end{equation*}
	\end{cor}
	\begin{proof}
		Note that for
		$\gamma\leq\frac{\mu-\beta}{2\beta}$ and $x\leq \gamma L$,  we have
		\begin{equation*}
			e^{(\mu+\beta)x}=e^{(\mu-\beta)x}e^{2\beta x}=O(1)e^{2\beta\gamma L}e^{(\mu-\beta)x}=O(1)e^{(\mu-\beta)L}e^{(\mu-\beta)x}.
		\end{equation*}
	\end{proof}
	\begin{cor}\label{cor:bulk}
		Let $0<\delta<\left(\frac{\mu-\beta}{2\beta}\right)\wedge \left(\frac{1}{2\beta}\frac{(\mu-\beta)^2}{\mu+\beta}\right)$ and $\gamma=\frac{\mu-\beta}{2\beta}-\delta$. Then,
		$$\mathbb{P}_x(R^\gamma ([0,tN])>0)=O\left(N^{-\left(\frac{\mu-\beta}{2\beta}-\delta \frac{\mu+\beta}{\mu-\beta}\right)}\right)(1\wedge x)e^{(\mu-\beta)x}, \quad \forall x\in[0,\gamma L], \; t\in[0,T].$$
	\end{cor}}
\begin{proof}
	This follows from Markov's inequality and the fact that $(\mu-\beta)-\gamma(\mu+\beta)=(\mu-\beta)\left(-\frac{\mu-\beta}{2\beta}+\delta \frac{\mu+\beta}{\mu-\beta}\right).$
\end{proof}
\begin{cor}\label{cor:pb:survie}
	Assume  that \eqref{hfp} holds and let $x>0$. Then
	\begin{equation*}
		\frac{1}{h(0,x)}N\mathbb{P}_x(R^1([0,tN])>0)\to 0, \quad N\to \infty,
	\end{equation*}
	where $h$ is as in \eqref{def:hL}.
\end{cor}
\begin{proof}
	According to Lemma \ref{lem:R}, Lemma \ref{lem:spectral} and Markov's inequality, it is enough to prove that
	\begin{equation*}
		Ne^{-2\beta L}\to 0, \quad \text{and}, \quad Ne^{2\beta x}e^{-(\mu+\beta)L}\to 0,
	\end{equation*}
	as $N\to\infty$. The first assertion is a direct consequence of \eqref{hfp}. The second convergence is clear when $x$ is fixed.
\end{proof}

\section{Convergence of moments}

From now on, we will assume that \eqref{hfp} holds. In this section, we will make heavy use of the notations defined in Section \ref{sect:k-spine-CV}.

\subsection{$k$-Mixing}

\begin{proposition}\label{prop:cvloi}
	Let $(\varphi_i; i\in[k])$ be bounded measurable functions and $(X_i)$ be  a sequence of i.i.d.~random variables with density $\Pi^\infty(dx) = \tfrac{v_1^\infty(x)^2}{||v_1^\infty||^2}dx$.
	Under $\bar Q_{x}^{k,t}$, define
	\begin{equation*}
		\bar R^N:=\prod_{v\in \cal B} r\left(\bar \zeta_v\right) h\left(|v|N,\bar \zeta_v\right)\prod_{v\in \cal L} \varphi_i(\bar \zeta_v).
	\end{equation*}
	Then as $N\to\infty$,
	\begin{equation}
		\bar R^N \Rightarrow \left(\prod_{i=1}^{k-1}r\left(X_i\right)h^{{ \infty}}\left(X_i\right)\right)\left(\prod_{i=1}^k \varphi_i(X^{i+k-1}) \right)
	\end{equation}
	in distribution.
\end{proposition}
\begin{proof}
	Let $f:[0,\infty)\to\mathbb{R}$ be a continuous bounded function.
	Let $(u_1,...,u_{k-1})\in(0,t)^{k-1}$ be such that
	$$\forall i,j\in[k-1],\quad  N|u_i-u_j|>c_{\ref{lem:hk}}L \quad \text{and} \quad  N(u_i\wedge(t-u_i))>c_{\ref{lem:hk}}L.$$
	It then follows from Proposition \ref{lem:hk} that
	\begin{equation*}
		\forall i,j\in[k-1]: u_i<u_j, \quad |q_{u_j-u_i}(x,y)-\Pi(y)|\leq e^{-\beta L}\Pi(y),
	\end{equation*} and that
	\begin{equation*}
		\forall i\in[k-1],  \quad |q_{u_i}(x,y)-\Pi(y)|\leq e^{-\beta L}\Pi(y) \quad \text{and} \quad|q_{t-u_i}(x,y)-\Pi(y)|\leq e^{-\beta L}\Pi(y).
	\end{equation*}
	We now condition $\bar R^N$ on the tree structure of the $k$-spine (see Section \ref{sect:k-spine-CV}). Corollary \ref{cor:cv:st:multivar} shows that
	\begin{equation*}
		\lim\limits_{N \rightarrow +\infty}\bar Q^{k,t}_x\left(f(\bar R^N)|(U_1,...U_{k-1})=(u_1,...,u_{k-1})\right)
		=\int_{} f\left(\prod_{i=1}^{k-1}r\left(x_i\right)h^\infty\left(x_i\right)\prod_{i=1}^k\varphi_i(x_{i+k-1})\right)\left(\prod_{i=1}^{2K-1}\Pi^\infty(d x_i)\right).
	\end{equation*}
	Yet, the random variables $(U_i)$ are independent uniforms on $[0,t]$ so that
	\begin{equation*}
		\forall i,j\in[k-1], \quad \mathbb{P}(N|U_i-U_j|>c_{\ref{lem:hk}}L)\to 0, \quad \mathbb{P}(N(U_i\wedge(t-U_i))>c_{\ref{lem:hk}}L)\to 0,
	\end{equation*}
	as $N$ tends to $\infty$.
	A union bound and the dominated  convergence theorem then yield
	\begin{align*}
		\lim\limits_{N \rightarrow +\infty} \bar Q^{k,t}_x\left(f\left(\bar R^N\right)\right) & =\mathbb{E}\left[\lim\limits_{N \rightarrow +\infty} \bar Q^{k,t}_x\left(f\left(\bar R^N\right)\bigg|U_1,...,U_{k-1}\right)\right]                                  \\
		                                                                                      & =\int_{} f\left(\prod_{i=1}^{k-1}r\left(x_i\right)h^\infty\left(x_i\right)\prod_{i=1}^k\varphi_i(x_{i+k-1})\right)\left(\prod_{i=1}^{2K-1}\Pi^\infty(d x_i)\right),
	\end{align*}
	which concludes the proof of the proposition.

\end{proof}

\subsection{Uniform integrability}
\begin{lemma}\label{lem:I}
	Assume that \eqref{hfp} holds. Let $T>0$, $n\in \mathbb{N}$ and let
	\begin{equation}
		\label{eq:eps}
		0 \leq \vep<\frac{3\beta -\mu}{n(\mu-\beta)}.
	\end{equation}
	There exists a constant $C_{\ref{lem:I}}=C_{\ref{lem:I}}(T,\vep,n)$ such that for sufficiently large $N$,
	\begin{equation*}
		I_{n-1,\vep,t}(x):=\int_{s=0}^{t}\int_{y=0}^Lh(sN,y)^{1+\vep}e^{\vep(n-1)(\mu-\beta)y}\bar q_{s}(x,y) \,dy\, ds\leqslant C_{\ref{lem:I}}e^{\vep n(\mu-\beta)x},
	\end{equation*}
	for all  $x\in[0,L]$, $t\in[0,T].$\end{lemma}
\begin{proof} Recall from Lemma \ref{lem:spectral} that  $h(sN,y)\leq 2 h(0,y)$ for $N$ large enough (that only depends on  $T$).
	Using Fubini's theorem along with Remark \ref{rem:green}, we see that
	\begin{align*}
		I_{n,\vep,T}(x) & =\int_{s=0}^{t}\int_{y=0}^Lh(sN,y)^{1+\vep}e^{\vep(n-1)(\mu-\beta)y}\left(\frac{h(sN,y)}{h(0,x)}p_{sN}(x,y)\right) \,dy\, ds                                                           \\
		                & \leqslant 2^{2+\vep}\int_{s=0}^{T}\int_{y=0}^Lh(0,y)^{1+\vep}e^{\vep(n-1)(\mu-\beta)y}\left(\frac{h(0,y)}{h(0,x)}p_{sN}(x,y)\right)\,dy\, ds                                           \\
		                & \leqslant 2^{1+\vep}e\underbrace{\frac{1}{N}\int_{y=0}^{L}h(0,y)^{1+\vep}e^{\vep(n-1)(\mu-\beta)y}\left(\frac{h(0,y)}{h(0,x)}G_{(TN)^{-1}}(x,y)\right)dy}_{ \let\scriptstyle\textstyle
		\substack{=:J_{n,\vep,T}(x)}}.
	\end{align*}
	Let $\xi=\frac{1}{TN}$ and $\lambda(\xi):=\linf+\xi$. By definition of $G_\xi$ (see Section \ref{sec:green}),
	\begin{align*}
		v_1(x)J_{n,\vep,T}(x) & =(N\omega_{\lambda(\xi)})^{-1}\left(\pl(x)\int_{y=0}^xh(0,y)^{1+\vep}e^{\vep(n-1)(\mu-\beta)y}\vl(y)v_1(y)dy\right.     \\
		                      & \qquad \qquad\qquad\qquad \left.+\; \vl(x){ \int_{y=x}^L} h(0,y)^{1+\vep}e^{\vep(n-1)(\mu-\beta)y}\pl(y)v_1(y)dy\right) \\
		                      & =:(N\omega_{\lambda(\xi)})^{-1}(\pl(x)A(x)+\vl(x)B(x)).
	\end{align*}
	Recalling the definition of $h$ from \eqref{def:h}, we see from  \eqref{hfp},  \eqref{eq:eps}, Lemma \ref{lem:spectral} and Lemma \ref{lem:phipsi} that
	\begin{align*}
		A(x) & =\int_{y=0}^xe^{(\mu-2\beta )y}\left(O(1)e^{-\beta y}+O(\xi)e^{\beta y}\right)e^{\vep n(\mu-\beta )y}dy                 \\
		     & =O(1)\int_{y=0}^xe^{(\mu-3\beta)y}e^{\vep n(\mu-\beta )y}dy+O(\xi)\int_{y=0}^xe^{(\mu-\beta)y}e^{\vep n(\mu-\beta )y}dy \\
		     & =O(1)(1\wedge x)+O(\xi)(1\wedge x)e^{(\mu-\beta)x}e^{\vep n(\mu-\beta )x},
	\end{align*}
	and remarking  that $\xi e^{(\mu-\beta)x}=O(1)$ for $x\in[0,L]$, we get that
	\begin{align*}
		\pl(x)A(x)=\left(v_1(x)+O(\xi L)(1\wedge(L-x))e^{-\beta x}\right)A(x)=O(1)v_1(x)e^{\vep n(\mu-\beta) x}.
	\end{align*}
	Similarly,
	\begin{align*}
		B(x)=O(1)\int_{y=x}^L e^{(\mu-3\beta)y}e^{\vep n(\mu-\beta) y}dy=O(1)(1\wedge(L-x))e^{(\mu-3\beta)x}e^{\vep n(\mu-\beta) x},
	\end{align*}
	so that (using that $\xi e^{(\mu-\beta)x}=O(1)$ again)
	\begin{align*}
		\vl(x)B(x) & = (1\wedge x)(O(\xi)e^{\beta x}+O(1)e^{-\beta x})(1\wedge(L-x))e^{(\mu-3\beta)x}e^{\vep n(\mu-\beta) x} \\
		           & = O(1)v_1(x)e^{\vep n(\mu-\beta) x}.
	\end{align*}
	Applying Lemma \ref{lem:wronsk} to $\xi=(TN)^{-1}$, we get that $(\omega_{\lambda(\xi)})^{-1}=O(N)$ so that
	\begin{equation*}
		J_{n,\vep,T}(x)=O(1)e^{\vep n (\mu-\beta)x},
	\end{equation*}
	which concludes the proof of the lemma.
\end{proof}

\begin{lemma}\label{lem:UI}
	Assume that \eqref{hfp} holds. Let $k\geq1$ and  $T>0$. For all
	\begin{equation}
		\label{eq:vep}
		0 \leq \vep<\frac{1}{(k-1)\vee 1}\frac{3\beta-\mu}{\mu-\beta},
	\end{equation}
	there exists a constant $c_{\ref{lem:UI}}=c_{\ref{lem:UI}}(\vep,T,k)>0$ such that
	\begin{equation*}
		\bar Q_x^{k,t}\left(\prod_{v\in \cal B}h(|v| N,\bar \zeta_v)^{1+\vep}\right)\leqslant c_{\ref{lem:UI}} t^{-(k-1)}e^{\vep (k-1) (\mu-\beta) x} , \quad \forall  x \in [0,L], \;\forall t\in(0,T].
	\end{equation*}

\end{lemma}
\begin{proof}We prove by induction that, for every {$n\leq k$},
	there exists a constant $c_n=c_n(\vep,T)$ such that
	\begin{equation}\label{eq:IH}
		\bar Q_x^{l,t}\left(\prod_{v\in \cal B} h(|v|N, \bar \zeta_v)^{1+\vep}\right)\leqslant c_n t^{-(l-1)}e^{\vep (l-1) (\mu-\beta) x }, \quad \forall x \in [0,L], \;\ t\in(0,T], \; \forall l\in[n].
	\end{equation}

	For $l =1$, we know that
	\begin{equation*}
		\bar{Q}_x^{1,t}[1]=1.
	\end{equation*}
	Let us now assume that \eqref{eq:IH} holds for some $n\leq k-1$.

	We claim that
	\begin{align*}
		 & \bar  Q_x^{n+1,t}\left(\prod_{v\in \cal B}h(|v|N,\bar \zeta_v)^{1+\vep}\right) \\
		 & =\int_0^t \frac{s^{n-1}}{t^{n}}e^{-(\linf-\lambda_1)(t-s)}
		\mathbb{E}_x\left[r(\bar \zeta_{t-s})h(0,\bar \zeta_{t-s})^{1+\vep}
			\sum_{n=1}^k\bar Q_{\bar \zeta_{t-s}}^{n,s}\left(
			\prod_{v\in \cal B}h(|v|N,\bar \zeta_v)^{1+\vep}\right)
			\bar Q_{\bar \zeta_{t-s}}^{k+1-n,s}
			\left(\prod_{v\in \cal B}h(|v|N,\bar \zeta_v)^{1+\vep}\right)\right]ds.
	\end{align*}
	This formula is an extension of Proposition \ref{prop:rec}. The proof of this fact is obtained by replacing the functional $F:\mathbb{U}^*_k\to\mathbb{R}$ by a functional $F$ of the $k$-spine tree $\mathcal{T}$ with depth $t$ of the form
	\begin{equation*}
		F(\mathcal{T})=\prod_{v\in\mathcal{B}}h(|v|,\zeta_v)^\vep\prod_{i=1}^k h(t,\zeta_{V_i})
	\end{equation*}
	and by summing over all the possible sizes for the left and right subtrees in the $(n+1)$-spine.  The proof then goes along the exact same lines.

	It  follows by induction (see \eqref{eq:IH}) that
	\begin{equation*}
		\bar Q_x^{n+1,t}\left(\prod_{v\in \cal B} h(|v|N,\bar \zeta_v)^{1+\vep}\right)\leqslant (c_n)^2  t^{-n} I_{n-1,\vep,t}(x).
	\end{equation*}
	Finally, we see from Lemma \ref{lem:I}  that
	\begin{equation*}
		\bar Q_x^{n+1,t}\left(\prod_{v\in \cal B} h(|v|N, \bar \zeta_v)^{1+\vep}\right)\leqslant (c_n)^2  C_{\ref{lem:I}}(n,\vep,T)    t^{-n}  e^{\vep n(\mu-\beta)x}.
	\end{equation*}
	This concludes the proof of \eqref{eq:IH} at rank $n+1$.
\end{proof}

\begin{cor}\label{cor:UI}
	The sequence of r.v.'s
	\begin{equation*}
		\left( \prod_{v\in \cal B} r(\bar {\zeta}_v) {h(|v|N,\bar {\zeta}_{v})} \prod_{i,j}   \psi_{i,j}( U_{\sigma_i,\sigma_j}) \prod_{i} \tilde \phi_{i} (\bar {\zeta}_{V_{\sigma_i}} ){ \frac{h^\infty(\bar \zeta_{V_{\sigma_i}})}{h(tN, \bar \zeta_{V_{\sigma_i}})}} , \ N\in\mathbb{N}\right)
	\end{equation*}
	from Theorem \ref{thm:k-spine-cv} is uniformly integrable under $\bar Q^{k,t}_x$.
\end{cor}
\begin{proof}

	Recall that the functions $\tilde \phi_i$ are compactly supported. Hence, there exists a constant $A>0$ such that $\tilde \phi_i\equiv 0$ on $(A,\infty)$, $\forall i\in[k]$. We know from  Lemma \ref{lem:spectral} that $v_1$ converges uniformly to $v_1^\infty$ on $[0,A]$, $v_1'$ converges uniformly to $(v_1^\infty)'$ on $[0,1]$ and that $v_1(x)\gtrsim x$ on $[0,1]$ for $L$ large enough. This implies that, as $N\to\infty$,
	\begin{equation}
		\label{eq:unif_ratio}
		\sup_{x\in[0,A]}\left|\frac{v_1^\infty(x)}{v_1(x)}-1\right|\to 0
	\end{equation}
	In particular, we see from  \eqref{approx:h} that there exists a constant $C>0$ such that $h^\infty(y)\leq C h(tN,y)$ for all $y\in[0,L]$, so that
	$$
		\bigg| \prod_{v\in \cal B} r(\bar {\zeta}_v) {h(|v|N,\bar {\zeta}_{v})} \prod_{i,j}   \psi_{i,j}( U_{\sigma_i,\sigma_j}) \prod_{ i} \tilde \phi_{i} (\bar {\zeta}_{V_{\sigma_i}} ){ \frac{h^\infty(\bar \zeta_{V_{\sigma_i}})}{h(tN, \bar \zeta_{V_{\sigma_i}})}}   \bigg| \leq
		C \prod_{v\in \cal B} h(|v|N, \bar \zeta_v).
	$$
	The result then follows from Lemma \ref{lem:UI}.
\end{proof}

\begin{proof}[Proof of Theorem \ref{thm:k-spine-cv}] Let $A>0$ be as in the proof of Corollary \ref{cor:UI}.
	Equations \eqref{approx:h} and \eqref{eq:unif_ratio} show that $h^\infty(\cdot)/h(tN, \cdot)\to 1$ uniformly on $[0,A]$ as $N\to\infty$. As a consequence, one can show that the sequence of random variables
	\begin{equation*}
		\left( \prod_{v\in \cal B} r(\bar {\zeta}_v) {h(|v|N,\bar {\zeta}_{v})} \prod_{i} \tilde \phi_{i} (\bar {\zeta}_{V_{\sigma_i}} ){ \frac{h^\infty(\bar \zeta_{V_{\sigma_i}})}{h(tN, \bar \zeta_{V_{\sigma_i}})}} , \ N\in\mathbb{N}\right)
	\end{equation*}
	converges weakly to the same limit as $\bar R^N$ in Proposition \ref{prop:cvloi}. Theorem \ref{thm:k-spine-cv} then follows from  Corollary~\ref{cor:UI}.
\end{proof}

\section{Survival probability ($0^{th}$ moment)}
\label{sect:survival proba}
Let $L$ be as in (\ref{def:L}) and $\tilde c$ as in \eqref{def:h}. Define
\begin{align*}
	u(t,x) & :=\mathbb{P}_x(Z^L_t>0),
\end{align*}
and
\begin{equation*}
	a(t):= \int_0^L\tilde h(0,y)u(t,y)dy= \tilde c\int_{0}^L e^{-\mu y}v_1(y) u(t,y)dy.
\end{equation*}
This section is aimed at proving Theorem \ref{thm:Kolmogorov}. Essentially, we will show that for $N$ large
\begin{equation*}
	u(tN,x) \ \approx h(0,x) a(tN),
\end{equation*}
so that the problem boils down to estimating $a(tN)$.
Fix $0<\eta<T$. The idea is to prove that
for $N$ large enough, $a$ satisfies
\begin{equation}\label{eq:diff-apporx}
	\dot a(tN) \approx -\frac{\Sigma^2}{2} a(tN)^2, \ \ \forall t\in[\eta,T].
\end{equation}
As a consequence,
$$
	a(tN) \approx \frac{1}{\frac{\Sigma^2}{2} (t-\eta)N + \frac{1}{a(\eta N)}}.
$$
Finally, the result will follow provided that
\begin{equation}\label{eq:infinf}
	\liminf_{\eta\to0} \liminf_{N\to\infty} \ \ {\eta N a(\eta N)>0}.
\end{equation}

\subsection{Step 1: Rough bounds}
{In the remainder of this section, we will often make use of a union bound that we describe now.
	We know from the branching property that for all $t>0$,
	\begin{equation}
		u(t N,x)=\mathbb{P}_x\left(\cup_{v\in\mathcal{N}^L_{L^2}}\{Z_{tN-L^2}^{(v)}>0\}\right), \label{br:prop}
	\end{equation}
	where $Z^{(v)}_s$ refers to the number of descendants of the particles $v$ at time $L^2+s>0$ in the particle system ${\bf X}^L$.
	Using a union bound and the  many-to-one lemma (see Lemma \ref{lem:many-to-one}), we see that
	\begin{equation*}
		u(t N,x)\leqslant \int_0^Lp_{L^2}(x,y)u(tN-L^2,y)dy.
	\end{equation*}
	Proposition \ref{lem:hk} along with \eqref{def:hL}  then shows that}
\begin{align}\label{union:bound0}
	\int_0^Lp_{L^2}(x,y)u(tN-L^2,y)dy & =(1+O(e^{-\beta L}))e^{(\lambda_1-\linf)L^2}h(0,x)a(tN-L^2).
\end{align}
We recall from  Lemma \ref{lem:spectral} that $\lambda_1<\linf$. This implies that for $N$ large enough (that does not depend on $t$),
\begin{equation}
	\label{union:bound}
	u(t N,x)\leqslant (1+O(e^{-\beta L}))h(0,x)a(tN-L^2).
\end{equation}

\begin{lemma}[Rough upper bounds]
	\label{lem:rub}
	Let $T>0$. There exist two positive constants $\gamma$ and  $c_{\ref{lem:rest}}=c_{\ref{lem:rest}}(T)$ such that for all $\eta>0$, there exists $\tilde N=\tilde N(T,\eta)$ such that, for all $N\geq \tilde N$, we have
	\begin{equation*}
		a(tN)\leq\frac{c_{\ref{lem:rest}}}{N^{\gamma}}, \quad \forall t\in[\eta,T],
	\end{equation*}
	and
	\begin{equation*}
		u(tN,x)\leqslant c_{\ref{lem:rest}} \frac{h(0,x)}{N^\gamma}, \quad x\in[0,L],  \quad \forall t\in[\eta,T].
	\end{equation*}\label{lem:rest}
\end{lemma}
\begin{proof} In this proof, the quantities $O(\cdot)$ can depend on $T$ but not on $\eta$. We will make use of the notation defined in the beginning of Section \ref{sec:bulk}.

	Basically, we prove that, in order to survive a period of time of order $N$, the BBM has to reach the level $\frac{\mu-\beta}{2\beta}L<L$. We established in Corollary \ref{cor:bulk} that the probability to reach this level is of order $(N^{-\tilde \gamma})$ for some $\tilde \gamma>0$.

	Let $0<\delta<\left(\frac{\mu-\beta}{2\beta}\right)\wedge \left(\frac{1}{4\beta}\frac{(\mu-\beta)^2}{\mu+\beta}\right)$ and $\tilde \gamma=\frac{\mu-\beta}{2\beta}-\delta$.
	For all $x\in \left[0,\tilde \gamma L\right]$, we have
	\begin{align*}
		\mathbb{P}_x(Z^{L}_{tN}>0) & =\mathbb{P}_x(Z^{\tilde \gamma L}_{tN}>0,R^{\tilde \gamma}([0,tN])=0)+\mathbb{P}_x(Z^{L}_{tN}>0,R^{\tilde \gamma} ([0,tN])>0) \\
		                           & \leqslant \mathbb{P}_x(Z^{\tilde \gamma L}_{tN}>0)+\mathbb{P}_x(R^{\tilde \gamma} ([0,tN])>0).
	\end{align*}
	The first probability can be bounded by Proposition \ref{lem:hk}. Indeed, for all $x\in[0,{\tilde \gamma }L]$, we have
	\begin{align*}
		\mathbb{P}_x(Z^{{\tilde \gamma }L}_{tN}>0)\leqslant \mathbb{E}_x[Z^{{\tilde \gamma }L}_{tN}] & =\int_{0}^{\tilde \gamma L}p_{tN}^{\tilde \gamma } (x,y)dy                                                                                                                                                \\
		                                                                                             & \leqslant (1+O(e^{-\beta L}))e^{(\lambda_1^{\tilde \gamma }-\linf)tN}e^{\mu x}\frac{v_1^{\tilde \gamma }(x)}{\|v_1^{\tilde \gamma }\|^2}\int_{0}^{{\tilde \gamma }L}e^{-\mu y}v_1^{\tilde \gamma } (y)dy.
	\end{align*}
	 
	Lemma \ref{lem:spectral} yields that $\int_0^{\tilde \gamma L}e^{-\mu y}v_1^{\tilde \gamma }(y)dy<\infty$, $\|v_1^{\tilde \gamma }\|\to\|v_1^\infty\|$ as $N\to\infty$ and that
	\begin{equation}\label{v:delta}
		v_1^{\tilde \gamma } (x)\leqslant C(x\wedge 1\wedge ({\tilde \gamma }L -x))e^{-\beta x} \leqslant C v_1(x), \quad \forall x\in[0,{\tilde \gamma }L],
	\end{equation}
	for $L$ large enough so that $L-{\tilde \gamma }L >1$. In addition, it implies that there exists a positive constant $C=C(\linf)>0$ such that
	\begin{equation*}
		e^{(\lambda_1^{\tilde \gamma }-\linf)tN}\leq\exp(-C t e^{-2\beta {\tilde \gamma }L}e^{(\mu-\beta)L})=\exp(-C t e^{2\beta \delta L}).
	\end{equation*}
	Hence,
	\begin{equation}\label{ub:Z}
		\mathbb{P}_x(Z^{\tilde \gamma L}_{tN}>0)= O(\exp(-C\eta e^{2\beta \delta L}))h(0,x),\quad x\in[0,{\tilde \gamma }L].
	\end{equation}
	It then follows from Corollary \ref{cor:bulk} and \eqref{v:delta} that for all $x\in(0,\tilde \gamma L)$,
	\begin{equation*}
		\mathbb{P}_x(R^{\tilde \gamma} ([0,tN])>0)=O\left(N^{-\left(\frac{\mu-\beta}{2\beta}-\delta \frac{\mu+\beta}{\mu-\beta}\right)}\right)h(0,x).
	\end{equation*}
	Note that the exponential factor in \eqref{ub:Z} is smaller than any power of $N$ (for $N$ large enough) so that
	\begin{equation}\label{eq:ub:int}
		\mathbb{P}_x(Z^L_{tN}>0)=O\left(N^{-\left(\frac{\mu-\beta}{2\beta}-\delta \frac{\mu+\beta}{\mu-\beta}\right)}\right)h(0,x), \quad \forall x\in[0,\tilde \gamma L].
	\end{equation}
	Remark that this proves the second part of the lemma for all $x\leq \tilde \gamma  L<L$. We now establish an upper bound on $a(t)$ to get a control on $u(t,x)$ for larger values of $x$.

	We split the integral $a(tN)$ into two parts: we see from \eqref{eq:ub:int}, \eqref{def:hL} and \eqref{exp:d:v1} \begin{align*}
		a(tN) & =\int_{0}^{\tilde \gamma L}\tilde h(0,x)u(tN,x)dx+\int_{\tilde \gamma L}^L\tilde h(0,x)u(tN,x)dx \\
		      & =O\left(N^{-\left(\frac{\mu-\beta}{2\beta}-\delta \frac{\mu+\beta}{\mu-\beta}\right)}\right)
		\int_0^{\tilde \gamma L}v_1(y)^2dy+O(1)\int_{\tilde \gamma L}^{L}e^{-(\mu+\beta)x}dx,
	\end{align*}
	where we bound $u(tN,x)$ by $1$ in the second integral. Moreover, note that
	\begin{align*}
		\int_{\tilde \gamma L}^{L}e^{-(\mu+\beta)x}dx=O\left(e^{-(\mu+\beta)\tilde \gamma L}\right)=O\left( N^{-\frac{\mu+\beta}{\mu-\beta} \tilde \gamma}\right).
	\end{align*}
	Finally, since we chose $\delta<\frac{1}{2} \frac{(\mu-\beta)^2}{2\beta (\mu+\beta)}$ in the beginning of the proof, we get that for sufficiently large $N$,
	\begin{equation*}
		a(tN)=O\left(N^{-\frac{\mu-\beta}{4\beta}}\right), \quad \forall t\in[\eta,T].
	\end{equation*}
	Note that the constant depends on $T$ but not on $\eta$.

	The upper bound on $u$ follows from \eqref{union:bound} (and the remark following the equation). Indeed, for $N$ large enough, $L^2<\frac{1}{2}\eta N$ and we see from the first part of the result that for $N$ large enough,
	$$a(tN-L^2)\leqslant \frac{c_{\ref{lem:rest}}(\linf,T)}{N^{\frac{\mu-\beta}{4\beta}}}.$$
	
\end{proof}

\begin{lemma}[Rough lower bounds]\label{lem:pbs:lb}Let $0<T_1<T_2$. There exist  two positive constants  $C_1=C_1(\linf)$, $C_2=C_2(\linf)$ and an integer $\tilde N=\tilde N(T_1,T_2)$ such that for all $N\geq \tilde N$, we have
	\begin{equation*}
		u(tN,x)\geqslant C_1 \frac{h(0,x)}{1+{t}N},\quad \forall x\in[0,L],\quad \forall t\in[T_1,T_2],
	\end{equation*}
	and
	\begin{equation*}
		a(tN)\geq\frac{C_2}{1+{t}N}, \quad \forall t\in[T_1,T_2].
	\end{equation*}
\end{lemma}
\begin{proof}The idea of the proof is adapted from \cite[Lemma 7.2]{harris21}.
	{ Let $\tilde \P^{t}_{x}$ be the probability measure absolutely continuous w.r.t. to $\P_x$
	with Radon-Nikodym derivative
	$$
		\frac{d \tilde \P^{t}_x}{d \P_x} \ = \ \frac{1}{h(0,x)} \sum_{v \in {\cal N}_{{tN}}^L} h(t, x_v ).
	$$}
	This {change of measure}  combined with Jensen's inequality yields
	\begin{align*}
		\mathbb{P}_x(Z^L_{tN}>0)=\mathbb{E}_x\left[\mathbf{1}_{Z^L_{tN}>0}\right]= \tilde \P_x^{tN}\left[\frac{h(0,x)}{\sum_{v\in\mathcal{N}_{tN}^L}h(tN,x_v)}\right]\geq \frac{h(0,x)}{\tilde  \P_x^{tN}\left[\sum_{v\in\mathcal{N}_{tN}^L}h(tN,x_v)\right]}.\end{align*}
	Yet,
	\begin{equation}\label{chg:m}
		\tilde \P_x^{tN}\left[\sum_{v\in\mathcal{N}_{tN}^L}h(tN,x_v)\right]=\frac{\mathbb{E}_x\left[\left(\sum_{v\in\mathcal{N}_{tN}^L}h(tN,x_v)\right)^2\right]}{h(0,x)}.
	\end{equation}
	{ Corollary \ref{th:many-to-few}} then yields
	\begin{equation*}
		\mathbb{E}_x\left[\left(\sum_{v\in\mathcal{N}_{tN}^L}h(tN,x_v)\right)^2\right]=\mathbb{E}_x\left[\sum_{v\in\mathcal{N}_{tN}^L}h(tN,x_v)^2\right]+{ 2h(0,x)tNQ^{2,tN}_x\left(r(\zeta_v)h(|v|N, \zeta_v)\right)},
	\end{equation*}
	where $v$ refers to the unique branching point in the $2$-spine tree of depth $tN$.
	The first term on the RHS of the above can be calculated thanks to Corollary \ref{th:many-to-few}:
	\begin{equation*}
		\mathbb{E}_x\left[\sum_{v\in\mathcal{N}_{tN}^L}h(tN,x_v)^2\right]=h(0,x) Q_x^{1,{ tN}}(h(tN, \zeta_{\tilde v})),
	\end{equation*}
	where $\tilde v$ is the unique leaf of the $1$-spine at time $tN$.

	One can easily check that
	$Q^{1,tN}_x\left(h(|v|N,\zeta_{\tilde v})\right)$  and $Q^{2,tN}_x\left(r(\zeta_{v})h(|v|N,\zeta_{v})\right)$
	are  uniformly bounded in $x\in[0,L]$  by a constant that does not depend on $T_1$ nor on $T_2$.
	Putting all of this together, we see that there exists a constant $C_1>0$ such that
	\begin{equation*}
		\lim_{N\to\infty} \mathbb{E}_x\left[\left(\sum_{v\in\mathcal{N}_{tN}^L}h(tN,x_v)\right)^2\right]\leq C_1(1+tN)h(0,x),
	\end{equation*}
	and that $C_1$ does not depend on $T_1$ nor on $T_2$. This equation combined with \eqref{chg:m} yields the first part of the lemma. The second part of the result follows from an integration.
\end{proof}

\subsection{Step 2. Comparing $a(t)$ and $u(t,x)$}

\begin{lemma}\label{lem:equadif} For all $t>0$, we have
	\begin{equation*}
		\dot a(t)=-\int_{0}^L \tilde h(0,x)r(x)u(t,x)^2dx+O\left(\frac{1}{N^{\alpha-1}}\right)a(t).
	\end{equation*}
\end{lemma}

\begin{proof}
	By definition of $a(t)$, we see that
	\begin{equation*}
		\dot a(t)=\int_{0}^L \tilde  h(0,x)\partial_t u(t,x)dx.
	\end{equation*}
	Yet, $u$ is solution of the FKPP equation
	\begin{equation}
		\partial_t u(t,x)=\frac{1}{2}\partial_{xx}u(t,x)-\mu \partial_x u(t,x) +r(x)(u(t,x)-u(t,x)^2), \  \ \ u(t,0)=u(t,L)=0.
	\end{equation}
	On the other hand, note that $x\mapsto \tilde h(0,x)$ is solution to the ODE
	$$\frac{1}{2}y'' +\mu y'+r(x)y=(\lambda_1-\linf) y.$$
	An integration by parts then entails
	\begin{equation*}
		\dot a (t)=(\lambda_1-\linf)a(t) -  \int_0^Lr(x)\tilde h(0,x)u(t,x)^2dx.
	\end{equation*}
	Lemma \ref{lem:spectral} finally yields the result.
\end{proof}

\begin{cor}\label{cor:diff}
	Let $0<T_1<T_2$. Let $\gamma$ be as in Lemma~\ref{lem:rub}. For $N$ large enough,
	\begin{equation*}
		|\dot a(tN)|=O\left(\frac{1}{N^{\gamma}}\right)a(tN), \quad t\in[T_1,T_2].
	\end{equation*}
\end{cor}
\begin{proof}
	Recall from (\ref{union:bound}) that
	\begin{equation*}
		u(tN,x)=O(1)h(0,x)a(tN-L^2)
	\end{equation*}
	for $N$ large enough.
	In addition, we see from Lemma \ref{lem:rest} that for any $k\in\mathbb{N}$ and $N$ large enough (that only depends on $T_1$, $T_2$ and $k$), we have
	\begin{equation}
		a(tN-kL^2)\leqslant \frac{c_{\ref{lem:rest}}(T_2)}{N^\gamma}. \label{ub:a:gamma}
	\end{equation}
	{Thus,  the union bound (\ref{union:bound})}, Lemma \ref{lem:equadif}, \eqref{def:hL} and Proposition \ref{prop:first} yield
	\begin{equation*}
		|\dot a (tN)|=O\left(\frac{1}{N^{\alpha-1}}\right)a(tN)+O\left(\frac{1}{N^\gamma}\right)a(tN-L^2), \quad t\in[T_1,T_2].
	\end{equation*}
	Without loss of generality, one can assume that $\gamma<\alpha-1$. Hence, using that for $N\in\mathbb{N}$ fixed, the function $t\mapsto a(t)$ is decreasing, we see that
	\begin{equation}\label{eq:rec:a}
		|\dot a (tN)|=O\left(\frac{1}{N^\gamma}\right)a(tN-L^2), \quad t\in[T_1,T_2].
	\end{equation}
	Using the mean value theorem, we then obtain
	\begin{equation*}
		|a(tN)-a(tN-L^2)|=O\left(\frac{L^2}{N^\gamma}\right)a(tN-2L^2), \quad t\in[T_1,T_2],
	\end{equation*}
	so that \eqref{eq:rec:a} implies that
	\begin{equation*}
		|\dot a (tN)|=O\left(\frac{1}{N^\gamma}\right)a(tN)+O\left(\frac{L^2}{N^{2\gamma}}\right)a(tN-2L^2), \quad t\in[T_1,T_2].
	\end{equation*}
	Let $k\in\mathbb{N}$ be such that $k\gamma>1$. Iterating the above estimates, we see that
	\begin{equation*}
		|\dot a(tN)|=O\left(\frac{1}{N^\gamma}\right)a(tN)+O\left(\frac{L^{2(k-1)}}{N^{k\gamma}}\right)a(tN-kL^2), \quad t\in[T_1,T_2].
	\end{equation*}
	Note that it suffices to choose $N$ large enough such that $kL^2/N<T_1/2$. Then, recall from Lemma \ref{lem:pbs:lb} that $\frac{1}{N}\lesssim a(tN)$. This, combined with \eqref{ub:a:gamma} yields the result.
\end{proof}

\begin{cor}
	\label{lem:uvsa}
	Let $\eps>0$ and $0<T_1<T_2$. There exists $\tilde N=\tilde N(\eps,T_1,T_2)$  such that
	for every $N\geq \tilde N$,
	\begin{equation*}
		\forall x\in[0,L], \quad (1-\eps ) h(0,x) a(tN) \leq u(tN,x) \leq (1+\eps) h(0,x) a(tN).
	\end{equation*}
\end{cor}
\begin{proof}
	It follows from \eqref{br:prop} along with Bonferroni inequalities that
	\begin{align}
		u(tN,x) & \geq \mathbb{E}_x\left[\sum_{v\in\mathcal{N}^L_{L^2}}\mathbb{P}_{x_v}\left(Z^{(v)}_{tN-L^2}>0\right)\right]                                                                  
		-\frac{1}{2} \mathbb{E}_x\left[\sum_{v\neq w\in\mathcal{N}^L_{L^2}}\mathbb{P}_{x_v}\left(Z^{(v)}_{tN-L^2}>0\right)\mathbb{P}_{x_w}\left(Z^{(w)}_{tN-L^2}>0\right)\right]\label{bonferroni}                       \\
		        & \geq \mathbb{E}_x\left[\sum_{v\in\mathcal{N}^L_{L^2}}\mathbb{P}_{x_v}\left(Z^{(v)}_{tN}>0\right)\right] -\frac{1}{2} \mathbb{E}_x\left[\sum_{v\neq w\in\mathcal{N}^L_{L^2}}\mathbb{P}_{x_v}\left(Z^{(v)}_{tN-L^2}>0\right)\mathbb{P}_{x_w}\left(Z^{(w)}_{tN-L^2}>0\right)\right].\nonumber
	\end{align}
	Using a similar argument to that used in \eqref{union:bound0} and \eqref{union:bound}, we see that
	\begin{align*}
		\mathbb{E}_x\left[\sum_{v\in\mathcal{N}^L_{L^2}}\mathbb{P}_{x_v}\left(Z^{(v)}_{tN}>0\right)\right] & =\int_0^Lp_{L^2}(x,y)u(tN,y)dy=\left(1+O\left(e^{-\beta L}\right)\right)h(0,x)a(tN).
	\end{align*}
	Hence, for $N$ large enough (that only depends on $\vep$), we have
	\begin{equation}\label{eq:lb1moment}
		\mathbb{E}_x\left[\sum_{v\in\mathcal{N}^L_{L^2}}\mathbb{P}_{x_v}\left(Z^{(v)}_{tN}>0\right)\right]\geqslant (1-\vep)h(0,x)a(tN).
	\end{equation}
	Corollary \ref{th:many-to-few} then yields
	\begin{multline}\label{eq:ut1}
		\mathbb{E}_x\left[\sum_{v\neq w\in\mathcal{N}^L_{L^2}}\mathbb{P}_{x_v}
		\left(Z^{(v)}_{tN-L^2}>0\right)\mathbb{P}_{x_w}\left(Z^{(w)}_{tN-L^2}>0\right)\right]
		=2h(0,x)L^2Q_x^{2,L^2}\left[r(\zeta_v)h({|v|},\zeta_v)\prod_{i=1,2}
		\frac{u(tN-L^2,\zeta_{v_i})}{h(L^2,{\zeta_{v_i}})}\right],
	\end{multline}
	where $v$ is the unique branching point of the $2$-spine tree of depth $L^2$ and $v_1,v_2$ are its two leaves.
	Using \eqref{union:bound0} and \eqref{union:bound} again, we get that for $N$ large enough,
	\begin{align*}
		u(tN-L^2,y)\leqslant & 2h(0,y)a(tN-2L^2).
	\end{align*}
	On the other hand, $h(L^2,y)\geq\frac{1}{2}h(0,y)$ for $N$ large enough (see Lemma \ref{lem:spectral}).
	
	These upper and lower bounds, combined with \eqref{eq:ut1} yield
	\begin{equation*}
		\mathbb{E}_x\left[\sum_{v\neq w\in\mathcal{N}^L_{L^2}}\mathbb{P}_{x_v}\left(N^{(v)}_{tN-L^2}>0\right)
		\mathbb{P}_{x_w}\left(N^{(w)}_{tN-L^2}>0\right)\right]
		\leqslant 16  C L^2h(0,x)a(tN-2L^2)^2 { Q_{x}^{2,L^2}(r(\zeta_v) h(|v|, \zeta_v) ),}
	\end{equation*}
	for some $ C= C(\linf)$ and $N$ large enough. The term
	$Q_{x}^{2,L^2}( r(\zeta_v)h({|v|, \zeta_v}) )$ can be shown to be uniformly bounded in $N$ and $x$
	using the same techniques as in  Lemma \ref{lem:UI}.

	We know from { Corollary} \ref{cor:diff} and the mean value theorem that
	\begin{equation*}
		a(tN)=\left(1+O\left(\frac{L^2}{N^\gamma}\right)\right)a(tN-2L^2),
	\end{equation*} so that, for $N$ large enough,
	\begin{equation}\label{eq:sim:a}
		{(1-\vep)}a(tN)\leq a(tN-2L^2)\leq {(1+\vep)}a(tN).
	\end{equation}
	Putting this together with  \eqref{bonferroni}, \eqref{eq:lb1moment} and  Lemma \ref{lem:rest}, we get that for sufficiently large $N$,
	\begin{equation}\label{lbp:u}
		u(tN,x)\geq(1-2\vep)a(tN)h(0,x).
	\end{equation}
	The upper-bound of the lemma can be obtained in a similar manner, using only the many-to-one lemma and \eqref{eq:sim:a}.
\end{proof}

\subsection{Step 3. Kolmogorov estimates}

\begin{lemma}\label{lem:ineq}Let $\vep>0$ and $0<T_1<T_2$. For $N$ large enough, we have
	\begin{align*}
		-\frac{\Sigma^2}{2}(1+\vep)a(tN)^2\leqslant \dot a(tN)\leqslant -\frac{\Sigma^2}{2}(1-\vep)a(tN)^2, \quad  T_1<t<T_2,
	\end{align*}
	where $\Sigma$ is as in Theorem \ref{thm:Kolmogorov}.
\end{lemma}
\begin{proof}

	We see from \eqref{hfp} and Lemma \ref{lem:pbs:lb} that there exists a small $\delta>0$ such that
	\begin{equation}\label{cp:a:N}
		\frac{1}{N^{\alpha-1}}=O\left(\frac{1}{N^\delta}\right)a(tN).
	\end{equation}
	The result follows from a direct application of Lemma \ref{lem:equadif}
	and Corollary \ref{lem:uvsa}.
\end{proof}

\begin{proof}[{ Proof of Theorem \ref{thm:Kolmogorov}}]
	Integrating the inequality in Lemma \ref{lem:ineq}  allows to approximate
	$a(tN)$ by the solution of (\ref{eq:diff-apporx}) on $[\eta N,  T N]$.
	On the other hand,
	Lemma \ref{lem:pbs:lb} yields that
	(\ref{eq:infinf}) is satisfied. It then follows from Lemma \ref{lem:uvsa} that
	for all $\vep>0$ and $x>0$, there exists $\tilde N=\tilde N(\vep,t)$ such that for all $N\geq \tilde N$,
	\begin{equation*}
		(1-\vep)\frac{2h(0,x)}{\Sigma^2 t}\leq N\mathbb{P}_x(Z_{tN}^L>0) \leq(1+\vep)\frac{2h(0,x)}{\Sigma^2 t}.
	\end{equation*}
	It remains to prove that, as $N\to\infty$,
	\begin{equation*}
		N\left(\mathbb{P}_x(Z_{tN}^L>0)-\mathbb{P}_x(Z_{tN}>0)\right)\to 0.
	\end{equation*}
	This follows from Corollary \ref{cor:pb:survie} remarking that $\{Z_{tN}^L>0\}=\{Z_{tN}>0\}$ on the event $\{R^1([0,TN])=0\}$.
\end{proof}

\section{Convergence of metric spaces}\label{sect:final-cutoff}

\subsection{Truncation of marked metric spaces} Recall from Section \ref{sect:mmm} that $\mathbb{M}$ refers to the set of equivalence classes of mmm-spaces.
Let $M=[X,d,\nu] \in {\mathbb M}$. For any measurable set $X'\subset X$, we write $|X'|= \nu( X' \times E)$.
\begin{definition}\label{def:MGP}
	The Marked Gromov Prokhorov distance between two elements of ${\mathbb M}$
	is defined as
	$$
		\forall [X_i,d_i,\nu_i] \in {\mathbb M}, \quad   d_{MGP}\left([X_1,d_1,\nu_1],[X_2,d_2,\nu_2]\right) \ = \ \inf_{Z,\phi_1,\phi_2}d_{Pr}(\phi_1\star \nu_1, \phi_2\star \nu_2),
	$$
	where the infimum is taken over all complete metric spaces $(Z, d_Z)$ and over all isometric embeddings $\phi_i$
	from $X_i$ to $Z$, $i=1,2$. Finally, $d_{Pr}$ is the standard Prokhorov distance between measures.
\end{definition}
It is now a standard result that the Gromov weak topology is metrisable by the metric $d_{MGP}$ (see e.g.~\cite{Greven2009}).

\begin{lemma}\label{cor:pGP}
	Let $\left(M_n=[X_n, d_n, \nu_n]; n\geq0\right)$ be a sequence of random mmm-spaces
	and for every $n$, let $X_n'$ be a closed measurable subset of $X_n$. Assume that  $\E[|X_n| - |X_n'|]\to 0$ as $n\to\infty$ and that $[X_n',d_n,\nu_n]$ converges to $[X_\infty, d_\infty, \nu_\infty]$ in distribution. Then
	$[X_n,d_n,\nu_n]$ converges in distribution to the same limit.
\end{lemma}
\begin{proof} Let $\nu_n'$ be the restriction of $\nu_n$ to $X_n'\times E$.
	It follows from Definition \ref{def:MGP} that
	\begin{equation*}
		d_{MGP}([X_n',d_n,\nu'_n],[X_n,d_n,\nu_n])\leqslant d_{Pr}(\nu_n,\nu_n').\end{equation*}
	By definition of the Prokhorov distance,
	\begin{equation*}
		d_{Pr}(\nu_n,\nu_n')\leq ||X_n|-|X_n'||.
	\end{equation*}
	The result then follows from Markov's inequality.

\end{proof}
In the following, we restrict ourself to the case where the mark space is an open set $E\subseteq \R_+$.
We also consider an increasing sequence of closed finite sets $(E_{\eps},\eps>0)$
such that
$
	\bigcup_{\eps>0} E_{\eps} = E.
$
For every finite mmm space $M=[X, d, \nu]$, we define
$$
	X^\eps  :=  \bigg\{x:(x,m)\in  X \times E_\eps\bigg\}, \quad  \text{and } \quad M^\eps :=  [X^\eps, d, \nu ].
$$
\begin{cor}\label{cor:approx-mmm2}
	Consider a sequence of finite random mmm-spaces $(M_n=[X_n, d_n, \nu_n]; n\geq0)$. Assume that
	$$\limsup_{\eps\to0} \limsup_{n\to\infty}  \E\left[| X_n| - |X_n^\eps |\right]=0.$$
	Then
	$$
		\limsup_{\eps\to 0} \limsup_{n\to\infty}  \ d_{MGP}\left( M_n^\eps, M_n \right) \ = \ 0 \ \ \mbox{in probability.}
	$$
\end{cor}
\begin{proof}
	The argument goes along the same lines as the proof of Lemma \ref{cor:pGP}.
\end{proof}

\subsection{Proof of Theorem \ref{thm:Yaglom} and \ref{thm:main-theorem}}

For  $\eps\in(0,1)$, set $E_{\eps}:=[\eps,\frac{1}{\eps}]$ and define $M_{t}^{L,\eps}$ out of $M_{t}^{L}$ by killing all the particles which do not belong to the interval $E_\eps$ at time $t$. In other words, we only keep the particles in $E_\eps$ at time $t$, with no ancestor outside of $[0,L]$ on the time interval $[0,{t}]$. We write $\nu^{L,\vep}_t$ for the restriction of the sampling measure $\nu$ to $\mathcal{N}^L_{tN}\times E_\vep$.
Finally, $\bar M_{t}^{L,\eps}$ is defined analogously to $\bar M^L_t$, by rescaling the total mass of the space and by accelerating time by $N$.

\begin{proposition}\label{prop:conv-truncation}
	Fix $\eps>0$ and define
	$$
		\tilde h^\infty_\eps(x) := \ \mathbf{1}_{x\in E_\eps}\tilde h^\infty(x).
	$$
	Conditional on $\{Z_{{tN}}>0\}$, $\left(\bar M_{t}^{L,\eps}, N\geq0\right)$ converges in distribution to a marked Brownian CPP
	with parameters $(t, \frac{\Sigma^2}{2}\tilde h^\infty_\eps(x)dx)$.
\end{proposition}
\begin{proof}
	{ We follow the heuristics of Section \ref{sect:lim-moments}.}
	Let $(\tilde \phi_i, i \in[k])$ and $(\psi_{i,j}, i,j\in[k])$ be bounded continuous functions.
	Assume that  the $\tilde \phi_i$'s are compactly supported on $(0,\infty)$.
	Define the polynomial
	$$
		\forall M = [X,d,\nu], \quad \tilde \Psi( M ) \ = \ \int \prod_{i,j} \mathbf{1}_{\{v_i\neq v_j\}} \psi_{i,j}(d({v_i,v_j})) \prod_i \tilde \phi_i(x_i) h^\infty(x_i) \nu(dv_i \otimes d x_i).
	$$
	From Theorem \ref{thm:k-spine-cv}, our Kolmogorov estimate (see Theorem \ref{thm:Kolmogorov}) and the many-to-few formula (see Proposition \ref{th:many-to-few2}),
	we get that
	$$
		\lim_{ N\to\infty} \E_x\left[\tilde \Psi( \bar M_{t}^{L} ) \big| Z_{tN} >0\right]  \ = \ k! \left(\frac{\Sigma^2}{2} t\right)^k \E\left[ \prod_{i,j} \psi_{i,j}(U_{\sigma_i,\sigma_j})\right] \prod_{i} \int \Pi^\infty(x) \tilde \phi_i(x) dx.
	$$
	Let $\phi_i$ be a bounded measurable function and consider
	$$
		\tilde \phi_i(x) \ = \  \frac{\phi_i(x) \mathbf{1}_{x\in E_\eps}}{h^\infty(x)}.
	$$
	Note that $\tilde \phi_i$ is now bounded and compactly supported. It is also continuous on the support of the measure $\bar{\nu}^\vep_t$.
	Define
	$$
		\forall M = [X,d,\nu], \quad \Psi'( M ) \ := \ \int \prod_{i,j} \mathbf{1}_{\{v_i\neq v_j\}}\psi_{i,j}(d({v_i,v_j})) \prod_i \phi_i(x_i) \nu( dv_i \otimes d x_i).
	$$
	Recall that $\Pi^\infty(dx) = h^\infty(x) \tilde h^\infty(x) dx$. The previous limit translates into
	\begin{align*}
		\lim_{N\to\infty} \E_x\left[\Psi'(\bar M_{t}^{L,\eps} ) \big| Z_{tN}>0\right]
		 & =    k!  \E\left[ \prod_{i,j} \psi_{i,j}(U_{\sigma_i,\sigma_j})\right] \prod_{i} \int \frac{t \Sigma^2}{2}\tilde h^\infty_\eps(x) \phi_i(x) dx.
	\end{align*}
	{Let us now define
	$$
		\forall M = [X,d,\nu], \quad \Psi( M ) \ := \ \int \prod_{i,j} \psi_{i,j}(d({v_i,v_j})) \prod_i \phi_i(x_i) \nu( dv_i \otimes d x_i).
	$$
	From Proposition \ref{lem:convDetermining} and  Proposition \ref{moments:CPP}, it remains to show that
	$$
		\lim_{N\to\infty} \bigg|\E_x\left[\Psi(\bar M_t^{L,\eps}) - \Psi'(\bar M_t^{L,\eps})  \big| Z_{tN}>0 \right]  \bigg| \ = \ 0.
	$$
	Hence, it is sufficient to show by induction on $k$ that
	$$
		\lim_{N\to\infty} \E_{x}\bigg( \int \mathbf{1}_{\cup_{1\leq i<j\leq k} \{v_i = v_j\}} \prod_{n=1}^{k} \bar   \nu_t^{L,\eps}( dv_n \otimes d x_n)  \big|  Z_{tN}>0 \bigg)  \ = \ 0.
	$$
	On the one hand,
	\begin{eqnarray*}
		&&\E_{x}\bigg[ \int \mathbf{1}_{\cup_{1\leq i<j\leq k} \{v_i = v_j\}}  \prod_{n=1}^{k} \bar   \nu_t^{L,\eps}( dv_n \otimes d x_n) \ \big| \  Z_{tN} >0  \bigg]\\ && \leq  \sum_{1 \leq i<j \leq k } \frac{1}{N} \E_{x}\left[ \prod_{n=1}^{k-1} \bar   \nu_t^{L,\eps}( dv_n \otimes d x_n) \ \big| \  Z_{tN} >0  \right].
	\end{eqnarray*}
	On the other hand, the RHS vanishes since by induction and the first part of the proof
	\begin{eqnarray*}
		&&\lim_{N\to\infty} \E_{x}\left[ \prod_{n=1}^{k-1} \bar  \nu_t^{L,\eps}( dv_n \otimes d x_n) \ \big| \  Z_{tN} >0   \right]\\
		&& =\lim_{N\to\infty} \E_{x}\left[ \int \mathbf{1}_{\cup_{1\leq i<j\leq k-1} \{v_i \neq v_j\}}\prod_{n=1}^{k-1}\bar   \nu_t^{L,\eps}( dv_n \otimes d x_n) \ \big| \  Z_{tN} >0   \right] \\ & &= \ k!\left(\frac{\Sigma^2t}{2}  \int_{E_\eps} \tilde h^\infty(x)dx \right)^k.
	\end{eqnarray*}
	This completes the proof of the proposition.}
\end{proof}

\begin{proposition}\label{prop:cv-truncated-mm}
	Conditional on the event $\{Z_{tN}>0\}$,
	$\bar M_{t}^L$ converges in distribution for the Gromov weak topology to the marked Brownian CPP
	with parameters $(t, \frac{\Sigma^2}2\tilde h^\infty(x)dx)$.
\end{proposition}
\begin{proof}
	By observing the moments of Brownian CPPs in Proposition \ref{moments:CPP},
	it is clear that
	a marked Brownian CPP
	with parameters $(t,\frac{\Sigma^2}{2}\tilde h^\infty_\eps(x)dx)$ converges to a marked CPP with parameters $(t, \frac{\Sigma^2}{2}\tilde h^\infty(x)dx)$
	as $\eps\to 0$. Next, a triangular inequality shows that our proposition boils down to proving that
	$$
		\mbox{Conditional on $\{Z_{tN}>0\}$},   \ \ \ \ \   \lim_{\eps\to0} \lim_{N\to\infty} d_{MGP}(\bar M^{L,\eps}_{t}, \bar M_{t}^L) = 0 \ \mbox{in probability.}
	$$
	We know from Proposition \ref{prop:conv-truncation} that, conditional on $\{Z_{tN}>0\}$, $\frac{1}{N}Z_{tN}^{L,\eps}$ converges to an exponential random variable with mean $\frac{t\Sigma^2}{2} \int_{E_\eps} \tilde h^\infty(x)dx$ and by Corollary \ref{cor:approx-mmm2}, it remains to show that
	$$
		\limsup_{\eps\to0} \limsup_{N\to\infty} \frac{1}{N}\E_x\left[Z_{tN}^L - Z^{L,\eps}_{tN} \ \big| \ Z_{tN} >0 \right] \ = \ 0.
	$$
	Note that
	\begin{align*}
		\frac{1}{N}\E_x\left[Z_{tN}^L - Z_{tN}^{L,\eps} \ \big| \  Z_{tN} >0 \right]
		= \frac{1}{N \P_x(Z_{tN}>0)} \int_{y\notin E_\eps} p_{tN}(x,y) dy.
	\end{align*}
	Finally, Proposition \ref{lem:hk} and Theorem \ref{thm:Kolmogorov} imply that for $N$ large enough, we have
	\begin{equation*}
		\frac{1}{N}\E_x\left[Z_{tN}^L - Z_{tN}^{L,\eps} \ \big| \  Z_{tN} >0 \right]\leqslant   C\left(\frac{h(0,x)}{h^\infty(x)}\right)\int_{y\notin E_\vep}\tilde h(0,y)dy.
	\end{equation*}
	Equations \eqref{def:hL} and \eqref{exp:d:v1} yield the result, letting first $N\to \infty$, then $\vep\to0$.
\end{proof}

\begin{proof}[Proof of Theorem \ref{thm:main-theorem}.]
	By Proposition \ref{prop:cv-truncated-mm}, it is enough to prove that conditional on $\{Z_{{tN}}>0\}$,
	$M_{{tN}}$ and $M_{{tN}}^L$ are coupled in such a way that they coincide with a probability going to $1$
	as $N\to\infty$.
	In light of Lemma \ref{cor:pGP} and of our Kolmogorov estimate, it is sufficient to show that $N\mathbb{E}_x[R^1([0,tN])]\to 0$ as $N\to\infty$. {Corollary \ref{cor:pb:survie} then yields the result.}

\end{proof}

\begin{proof}[Proof of Theorem \ref{thm:Yaglom}]
	This is a  corollary of Theorem \ref{thm:main-theorem}.
	The proof goes along the exact same lines as Theorem 2 in \cite{boenkost2022genealogy},
	where the convergence of the population size and of the genealogy is deduced from
	the convergence of the mmm-space to the Brownian CPP. We recall the main steps of the argument for completeness.

	Both maps
	\[
		[X,d,\nu] \mapsto |X|,\qquad [X,d,\nu] \mapsto
		\left[X,d,\tfrac{\nu}{|X|}\right]
	\]
	are continuous w.r.t.~the marked Gromov-weak topology. Recall that  conditional on survival, $\bar M_{{t}}$  converges in the marked Gromov-weak topology. Hence, (i) readily follows from the fact that the limiting CPP
	has a total mass exponentially distributed with mean $\frac{\Sigma^2 t}2$ (see Remark~\ref{rem:exp}).

	Let us now prove (ii).  Let $[X,d,\nu]$ be a general random mmm-space. Sample $k$
	points  $(v_1,\cdots,v_k)$	 uniformly at random with replacement.
	Let $(x_{v_1},\cdots, x_{v_k})$ be the types of the sampled individuals.
	Then
	$\E\big[
			\psi\big( \big(d(v_i, v_j)\big), \ (x_{v_i}) \big) \big]$
	is
	nothing but the moment of order $k$ of $[X,d,\tfrac{\nu}{|X|}]$.
	Since, conditional on survival, $\bar M_{{t}}$  converges to a Brownian CPP, (ii) follows from Proposition~\ref{SAmpling-CPP}.

\end{proof}

\section*{Acknowledgements}
This project has received funding from the European Union’s Horizon 2020 research and innovation programme under the Marie Skłodowska-Curie grant agreement No 101034413.

\bibliographystyle{plain} 
\bibliography{genealogy_fullypushed.bib}

\end{document}